\documentclass[12pt]{article}
\usepackage{amsmath,amssymb,amsthm,bbm,graphicx,xcolor,fullpage,xargs,sectsty,txfonts}
\usepackage[colorlinks]{hyperref}
\usepackage[shortlabels]{enumitem}
\usepackage[title,titletoc,header]{appendix}
\numberwithin{equation}{section}
\subsubsectionfont{\itshape}

\newcommand\map{\varg} 
\newcommand\shinvmap{K} 
\newcommand\homap{H} 
\newcommand\shinvhomap{J} 

\newcommand{\Eset}{\mathsf{E}}
\newcommand{\Fset}{\mathsf{F}}
\newcommand{\Nset}{\mathbb{N}}
\newcommand{\Rset}{\mathbb{R}} 
\newcommand{\Zset}{\mathbb{Z}} 

\def\rmd{\mathrm{d}}
\def\rme{\mathrm{e}}
\def\rmi{\mathrm{i}}
\def\constant{\mathrm{cst}}

\newcommand{\pr}{\mathbb{P}} 
\newcommand{\esp}{\mathbb{E}}
\newcommand{\var}{\operatorname{var}} 
\newcommand{\law}{\mathcal{L}}
\newcommand{\eqdistr}{\stackrel{d}{=}}
\newcommand{\convdistr}{\stackrel{d}{\longrightarrow}} 
\newcommand{\convprob}{\stackrel{P}{\longrightarrow}} 
\newcommand{\convvague}{\stackrel{v}{\longrightarrow}}
\newcommand{\convweak}{\stackrel{w}{\Longrightarrow}} 
\newcommand{\fidi}{\ \stackrel{\mathrm{fi.di.}}{\longrightarrow}\ }

\newcommandx{\norm}[2][1=]{\left|#2\right|_{#1}}
\newcommandx{\lpnorm}[3][1=,3=]{\left\|#2\right\|_{#1}^{#3}}
\newcommandx{\supnorm}[3][1=,3=]{\left\|#2\right\|_{#1}^{#3}}
\def\id{\mathrm{Id}}
\newcommand\shift{B}
\newcommand\exc{\mathcal{E}}
\newcommand\nualpha{\nu_\alpha}
\newcommand\tailmeasure{\boldsymbol{\nu}}
\newcommand\tailmeasurestar{\boldsymbol{\nu}^*}
\newcommand\tildetailmeasurestar{\widetilde{\boldsymbol{\nu}}^*}
\newcommand\candidate{\vartheta}
\newcommand\spaceD{\mathcal{D}}
\newcommand\spaceDtilde{\widetilde{\mathcal{D}}}
\newcommand\dtilde{\tilde{d}}
\newcommand\unzerospaceDtilde{\spaceDtilde_0\setminus\{\bszero\}}
\newcommand\1[1]{\mathbbm{1}_{#1}}
\newcommand\ind[1]{\mathbbm{1}{\left\{#1\right\}}}

\def\bszero{\boldsymbol{0}}
\newcommand{\bsQ}{\boldsymbol{Q}}
\newcommand{\bsx}{\boldsymbol{x}}
\newcommand{\bsy}{\boldsymbol{y}}
\newcommand{\bsW}{\boldsymbol{W}}
\newcommand{\bsX}{\boldsymbol{X}}
\newcommand{\bsY}{\boldsymbol{Y}}
\newcommand{\bsZ}{\boldsymbol{Z}}
\newcommand{\bseta}{\boldsymbol{\eta}}

\newcommand{\bsTheta}{\boldsymbol{\Theta}}
\newcommand{\bszeta}{\boldsymbol{\zeta}}

\def\mca{\mathcal{A}}
\def\mcb{\mathcal{B}}
\def\mce{\mathcal{E}}
\def\mcf{\mathcal{F}}
\def\mch{\mathcal{H}}
\def\mci{\mathcal{I}}
\def\mcj{\mathcal{J}}
\def\mcm{\mathcal{M}}
\def\mcp{\mathcal{P}}
\def\mct{\mathcal{T}}

\def\cadlag{c\`adl\`ag}

\def\eg{e.g.}
\def\ie{i.e.}

\def\wrt{with respect to}
\def\ifft{if and only if}
\def\iid{i.i.d.}
\def\nondecreasing{non-decreasing}
\def\nonincreasing{non-increasing}
\def\nonnegative{non-negative}
\def\shiftinvariant{shift-invariant}
\def\shiftinvariance{shift-invariance}
\def\leb{\mathrm{Leb}}

\usepackage{cleveref}

\newtheorem{theorem}{Theorem}[section]
\newtheorem{lemma}[theorem]{Lemma}
\newtheorem{corollary}[theorem]{Corollary}
\newtheorem{proposition}[theorem]{Proposition}
\newtheorem{hypothesis}[theorem]{Assumption}
\newtheorem{definition}[theorem]{Definition}
\theoremstyle{remark}
\newtheorem{remark}[theorem]{Remark}
\newtheorem{example}[theorem]{Example}

\crefname{hypothesis}{Assumption}{Assumptions}

\crefname{equation}{}{}
\crefname{enumi}{}{}

\crefname{example}{Example}{Examples}

\begin{document}

\title{The tail process and tail measure of continuous time regularly varying stochastic processes}

\author{Philippe Soulier\thanks{Equipe Modal'X and LABEX MMEDII, Universit\'e Paris Nanterre, 92000
    Nanterre, France, philippe.soulier@parisnanterre.fr}}

\maketitle

\begin{abstract}
  The goal of this paper is to investigate the tools of extreme value theory originally introduced
  for discrete time stationary stochastic processes (time series), namely the tail process and the
  tail measure, in the framework of continuous time stochastic processes with paths in the space
  $\spaceD$ of \cadlag\ functions indexed by $\Rset$, endowed with Skorohod's $J_1$ topology. We
  prove that the essential properties of these objects are preserved, with some minor (though
  interesting) differences arising. We first obtain structural results which provide representation
  for homogeneous \shiftinvariant\ measures on $\spaceD$ and then study regular variation of random
  elements in $\spaceD$. We give practical conditions and study several examples, recovering and
  extending known results.
\end{abstract}


\clearpage

\tableofcontents 

\clearpage

\section{Introduction}

The goal of this paper is to study the extreme value theory of continuous time regularly varying
processes stochastic processes in the light of the recent developments of this theory for regularly
varying time series (that is stochastic processes indexed by $\Zset$ with values in $\Rset^d$),
which we briefly recall.

A time series $\bsX=\{X_j,j\in\Zset\}$ is said to be regularly varying with tail index $\alpha$ if
all its finite dimensional distributions are regularly varying with the same index of regular
variation $\alpha$ and under the same scaling. This means that the finite dimensional distributions
are in the domain of attraction of a mutivariate Fr\'echet distribution with the same tail index and
under the same scaling.  The extremal behaviour of such a stationary time series is now well
understood and characterized by either one of two objects: the tail process, introduced by
\cite{basrak:segers:2009} and the tail measure introduced in the unpublished manuscript
\cite{owada:samorodnitsky:2012}.

The tail process, which will be denoted by $\bsY$ throughout this paper, describes the asymptotic
behaviour of a stationary time series given an extreme value at time zero and provides convenient
representations of the limiting quantities which arise in the statistics of extremes for time
series. The tail process can be formally defined for a non stationary time series but contains too
little information to be useful.

Alternatively, a time series $\bsX$ can be considered as a random element of the space
$(\Rset^d)^\Zset$ endowed with the product topology which makes it Polish. Regular variation in
$(\Rset^d)^\Zset$ can be defined using the theory of vague convergence of \cite{kallenberg:2017}
which will be described more precisely in \Cref{sec:representation,app:vague}.  In that framework a
time series is regularly varying if there exists a non zero measure $\tailmeasure$ on
$(\Rset^d)^\Zset$, and a scaling sequence $a_n$ such that $n\pr(a_n^{-1}\bsX\in A)$ converges to
$\tailmeasure(A)$ for all Borel sets $A\subset(\Rset^d)^\Zset$ which are contained in the complement
of an open neighborhood of the null sequence $\bszero$, and are continuity sets
of~$\tailmeasure$. In the terminology of extreme value theory for finite dimensional random vectors,
$\tailmeasure$ is the exponent measure of~$\bsX$. Following \cite{owada:samorodnitsky:2012}, we will
call it the tail measure of the time series $\bsX$. If the time series is stationary, then its tail
measure is \shiftinvariant. The finite dimensional projections of the tail measure are the exponent
measures of the finite dimensional distributions of the time series $\bsX$ and characterize the tail
measure. Therefore, both definitions of regular variation of a time series are equivalent and since
the tail process is defined by the finite dimensional distributions, the tail measure entirely
determines the tail process.  The converse was proved by \cite{planinic:soulier:2018} and
\cite{dombry:hashorva:soulier:2018}: the tail measure can be recovered from the tail process. This
is essentially due to homogeneity and shift-invariance of the tail measure.

The tail process and tail measure of a stationary time series have been exhaustively studied and are
extremely useful tools to understand the extremal behavior of a time series and describe the
asymptotic distributions of the partial maximum process, the partial sum process (when the tail
index is in $(0,2)$), and many statistics such as estimators of the tail index, the extremal index,
and other extremal characteristics. A thorough treatment of the subject is given in
\cite{kulik:soulier:2020}.

The extreme value theory of continuous time stochastic processes is a very ancient and still active
field of research. An important part of it is dedicated to Gaussian and related processes such as
diffusion processes. See for instance the monographs \cite{leadbetter:lindgren:rootzen:1983},
\cite{berman:1992} and more recently \cite{azais:wschebor:2009}. There are some early references on
extremes of continuous time regularly varying processes such as \cite{rootzen:1978} which deals with
moving averages with stable innovations in discrete and continuous time, but the bulk of the
literature seems to be more recent. See for instance among many other,
\cite{samorodnitsky:2004:maxima} (stable processes), \cite{fasen:2005} (moving averages driven by a
L\'evy process), \cite{fasen:kluppelberg:2007} (mixed moving averages), \cite{wang:stoev:2010}
(max-stable processes). There is one important difference between the Gaussian and related processes
first mentioned and regularly varying processes. The former are typically extremally independent,
that is the extremal behavior of their finite dimensional distributions is in first approximation
the same as that of a vector with independent components, whereas the extremal behaviour of the
latter typically inherits some form of serial dependence. Therefore, the tail process and tail
measure are useless for the former class of processes (or rather for regularly varying
transformations) but can be considered for the latter.

The main purpose of this work is to extend the theory of the tail process and the tail measure
established for stationary time series to regularly varying stationary continuous time stochastic
process. Stationarity is a restriction, but it is a usual assumption, especially in view of
statistical applications, and the tail process is only of interest in the context of stationarity.

The tail measure and the tail process of an $\Rset^d$-valued regularly varying stochastic process
$\bsX=\{\bsX_t,t\in\mct\}$ indexed by an arbitrary index set $\mct$ can be defined exactly as in
discrete time. If the finite dimensional distributions of the process are regularly varying, they
admit an exponent measure and the family of these exponent measures satisfy a consistency
property. These exponent measures are not finite but \cite{owada:samorodnitsky:2012} proved that
there exists a measure $\tailmeasure$ on $(\Rset^d)^\mct$ endowed with the product topology, whose
finite dimensional projections are the exponent measures.  As previously, the tail process can be
defined as the weak limit of the finite dimensional distributions of $\bsX$ given that
$\norm{\bsX_0}>x$, as $x\to\infty$.

However, using this definition, no information is given on the paths of the tail process nor on the
support of the tail measure. In full generality, the tail measure need not even be $\sigma$-finite,
see \cite[Proposition~2.4]{owada:samorodnitsky:2012}.  For processes indexed by $\Rset$, a natural
framework is to consider only processes with almost surely \cadlag\ paths, that is random element in
the space $\spaceD(\Rset,\Rset^d)$ (hereafter simply written $\spaceD$) endowed with the $J_1$
topology, which is a Polish space. In this framework, it is then natural to define the regular
variation of random element in $\spaceD$ as the convergence of the measure
$T\pr(a_T^{-1}\bsX\in \cdot)$ to a measure $\tailmeasure$ in the following sense: for all Borel sets
$A$ which are continuity sets of $\tailmeasure$ and are separated from the null map $\bszero$, that
is sets which are contained in the complement of a neighborhood of $\bszero$, it holds that
$\lim_{n\to\infty} T\pr(\bsX\in a_TA) = \tailmeasure(A)$. This mode of convergence is simply called
vague convergence in \cite{kallenberg:2017}. This concept of regular variation in Polish spaces was
originally developed in \cite{hult:lindskog:2006} and regular variation of \cadlag\ stochastic
processes was first considered in \cite{hult:lindskog:2005}, only for processes indexed by $[0,1]$,
that is random elements in $\spaceD([0,1],\Rset^d)$.

The tail measure $\tailmeasure$ is an exponent measure and as such must be homogeneous. In addition,
if the process $\bsX$ is stationary, the tail measure is \shiftinvariant. As already mentioned, its
finite dimensional projections are the exponent measures of the finite dimensional distributions
of~$\bsX$. It is also immediate that the distribution of the tail process is the tail measure
restricted to the set of functions $f\in\spaceD$ such that $|f(0)|>1$.

Taking these properties as definitions, the tail measure and tail process can be studied without
reference to an underlying stochastic process. This is done in \Cref{sec:representation} whose main
purpose is to extend the structural results obtained for tail measures on $(\Rset^d)^\Zset$ by
\cite{planinic:soulier:2018} to tail measures on $\spaceD$, defined as \shiftinvariant\ homogeneous
measures, finite on sets separated from $\bszero$.

The main result of \Cref{sec:representation} is \Cref{theo:Y-determines-nu} which states that,
similarly to the discrete time case, the tail process determines the tail measure, and a tail
measure always has a spectral representation, that is a pseudo polar decomposition \wrt\ the
semi-norm $f\mapsto|f(0)|$ on $\spaceD$.

Then we obtain in \Cref{theo:equivalences-dissipative} necessary and sufficient conditions for
mixed-moving average representations of the tail measure. This result subsumes those originally
obtained for max-stable processes by \cite{dombry:kabluchko:2016} (in particular their Theorem~3),
where no reference is made to the tail measure, although the link is implicit since the tail measure
of a max-stable process determines its distribution. These results can also be expressed in the
language of ergodic theory, in terms of dissipative and conservative flow representations. We will
not pursue this direction in this paper to keep it at a reasonable length. See
\cite{wang:roy:stoev:2013} for similar results for sum-stable processes whose distribution is also
determined by the tail measure.

\Cref{theo:Y-determines-nu} and \Cref{theo:equivalences-dissipative} are mutatis mutandis the same
as in discrete time. One important difference between discrete and continuous time is the role of
certain maps, called anchoring maps by \cite{basrak:planinic:2021}, one of which being the
infargmax functional which finds the first time where the maximum of a sequence is achieved (see
\Cref{xmpl:infargmax-general} for a precise definition).  These maps play a crucial role in the
study of the tail process in discrete time. In particular, and under certain conditions,
$\pr(\mci(\bsY)=0)$ is positive and independent of the anchoring map $\mci$, and the tail measure
can be expressed in terms of the tail process conditioned on $\mci(\bsY)=0$. See
\cite[Section~3.3]{planinic:soulier:2018}. This probability is denoted by $\candidate$ and called
the candidate extremal index, since it is related to the classical extremal index which will be
discussed hereafter.

In continous time, the event $\mci(\bsY)=0$ has in general a zero probability, and conditioning is
more difficult to handle.  It turns out instead of conditioning on $\mci(\bsY)=0$, a change of
measure with density $\exc^{-1}(\bsY)$, where $\exc$ is the exceedance functional, defined for a
measurable function $f:\Rset\to\Rset^d$ by $\exc(f) = \int_{-\infty}^\infty \ind{|f(t)|>1}\rmd t$
yields a representation of the tail measure in terms of the tail process.  We will also see in
\Cref{sec:anchoring-maps,sec:examples} that anchoring maps may behave very differently than in
discrete time and that conditioning on different anchoring maps may produce different
results.

\Cref{sec:representation} is concluded with certain identities for quantities which appear as limits
of certain statistics of regularly varying processes. Depending on the method used to obtain these
limits, they can be expressed in terms of the different objects related to the tail process. It is
therefore convenient and important to know that these expressions are equal and that the summability
or integrability conditions that guarantee their existence are equivalent. The usefulness of these
identities will be illustrated in \Cref{sec:illustration}.

All the results of \Cref{sec:representation}, in addition to be of intrinsic interest, are important
to understand the extremal behaviour of regularly varying stochastic processes, and more
particularly so for max-stable and sum-stable processes whose distribution is entirely determined by
the tail measure, or equivalently the tail process. Thus they are also necessary preliminaries to
the proper investigation of regularly varying stochastic processes, done in
\Cref{sec:regvarinD}. However, \cref{sec:spectral-representation} can be skipped by readers more
interested in applications.

As already mentioned, in discrete time, the two definitions of regular variation of a time series,
either by means of finite dimensional distributions or by considering the time series as a random
element of the sequence space are equivalent. This is obviously not the case in continuous time.
Thus our first task is to relate finite dimensional convergence and convergence in $\spaceD$. This
is done in \Cref{theo:rv-in-D-equivalence}.  This result extends those of
\cite{hult:lindskog:2005,hult:lindskog:2006} which dealt only with $\spaceD([0,1])$. It states a
necessary and sufficient condition for regular variation in $\spaceD$ in terms of convergence of
finite dimensional distributions and a tightness criterion which extends the usual one in terms of
the $J_1$-modulus of continuity.  The proof of the direct implication is omitted since it is an
immediate adaptation of the proof of the corresponding result in
\cite[Theorem~10]{hult:lindskog:2005}. However, for the converse implication, we take advantage of
the results of \Cref{sec:representation} to obtain a more constructive proof.  Importantly, we also
obtain in \Cref{theo:rv-in-D-equivalence} that regular variation in $\spaceD$ implies the weak
convergence in $\spaceD$ of the process $\bsX$, conditioned on $|X_0|>x$ when $x\to \infty$, to its
tail process $\bsY$ which is thus a random element in $\spaceD$.

From there on, we are able to easily extend to continuous time processes the main results of the
extreme value theory for discrete time series previously developed by means of the tail process in a
series of papers ranging from \cite{basrak:segers:2009} to \cite{basrak:planinic:soulier:2018}. The
most important object that we consider is related to the point process of exceedances which measures
the time spent by the time series above a high threshold, introduced in discrete time as early as
\cite{resnick:1986}. In continuous time, it was studied under the name excursion random measure by
\cite{hsing:leadbetter:1998} which builds on the seminal paper \cite{davis:hsing:1995} dealing with
discrete time processes. Under the mixing condition \Cref{eq:laplace-blocks} which is related to the
well-known condition $D$ of \cite{leadbetter:1974}, and under condition \Cref{eq:anticlustering}
which yields the limit of excursions over a high level within a small portion of the path (first
used in \cite{davis:hsing:1995}), we obtain in \Cref{theo:equivalence-ppconv} the weak convergence
of a generalization of the excursion random measure to a Poisson point process on (a subspace of)
$\spaceD$.  {It must be noted that the anticlustering condition implies that the tail
  measure has a mixed-moving average representation as in \cref{sec:mma-representations}, so our
  results exclude processes with extremal long memory, that is such that $\candidate=0$.}

The convergence of the excursion random measure has many applications, of which we cite only one,
related to the convergence of the sample maxima. Recall that for an \iid\ sequence with regularly
varying marginal distribution, if $a_T$ is the quantile of order $1-T^{-1}$, then
$a_T\max(X_1,\dots,X_T)$ ($T$ being restricted to integer values) converges weakly to a Fr\'echet
distribution, say $F_\alpha$. For a stationary sequence, this convergence may hold to
$F_\alpha^\theta$, where $\theta\in[0,1]$ is called the extremal index. Exact computation of the
extremal index is not often easy or possible, but the tail process provides several convenient
representations of the extremal index. In continuous time, it is still possible to define the
extremal index as a real number $\theta$ such that $a_T^{-1} \sup_{0\leq s \leq T} X_s$ converges
weakly to $F_\alpha^\theta$. The essential difference is that in continuous time, the extremal
index, if it exists is not confined to $[0,1]$ but can take any value in $[0,\infty]$. The case
$\theta=1$ in discrete time or $\theta=\infty$ in continuous time corresponds to extremal
independence; the case $\theta=0$ is often called long range dependence in the extremes (different
from other notions of long range dependence).

The conditions of \Cref{theo:equivalence-ppconv} may be difficult to check. It is easily seen that
they hold for $m$-dependent processes, that is processes such that past and future separated by $m$
are independent. In \Cref{theo:ppconv-approx}, we show that if a process admits a suitable sequence
of $m$-dependent approximations, then the conclusions of \Cref{theo:equivalence-ppconv} hold.

As a consequence of \Cref{theo:equivalence-ppconv,theo:ppconv-approx}, we prove the existence of the
extremal index in $(0,\infty)$ and obtain representations in terms of the tail process. Some of
these representations had been obtained in the context of max-stable processes by
\cite{debicki:hashorva:2020}.

Here we must stress that we only consider regularly varying processes which are not extremally
independent, \ie\ whose extremal behaviour has kept some form of temporal dependence. In particular,
the assumptions of \Cref{theo:equivalence-ppconv} exclude both extremal independence, \ie\
$\theta=\infty$ and long range dependence, \ie\ $\theta=0$. Convergence of the point process of
clusters necessitates different normalization in both cases and much more sophisticated
techniques. See for instance \cite[Section~8]{roy:2017}, \cite{owada:samorodnitsky:2015} and
\cite{lacaux:samorodnitsky:2016} for examples of processes which exhibit extremal long memory.

We conclude this paper with several illustrative but relatively simple examples in
\Cref{sec:examples}.  We start with max-stable (\Cref{sec:maxstable}) and sum-stable
(\Cref{sec:sumstable}) processes, for which we recover the know results of the literature and also
prove the convergence of the point process of exceedances.  Then we study a general class of
functional moving averages in \Cref{sec:functional-ma}, the simplest example of which is the
well-known shot noise process (\Cref{sec:shot-noise}).

\subsection*{Notation}
We will use the usual letters $f$, $g$, etc.  to denote functions and also boldface letters such as
$\bsx$, $\bsy$, depending on the context.  We use indifferently $\bsy_t$ of $\bsy(t)$ for the value
of $\bsy$ at $t\in\Rset$. We use capitalized boldface ($\bsX$, $\bsY$ \dots) for stochastic
processes indexed by $\Rset$ (or any subset).

The space $\spaceD(\Rset,\Rset^d)$ is the space of \cadlag\ functions defined on $\Rset$ and when
there is no risk of confusion, we simply write $\spaceD$.  The null function is denoted by
$\bszero$.  Given an arbitrary norm $\norm{\cdot}$ on $\Rset^d$, we define
$\spaceD_0=\{\bsy\in\spaceD:\lim_{|t|\to\infty}\norm{\bsy_t}=0\}$ and
$\spaceD_\alpha=\{\bsy\in\spaceD:\int_{-\infty}^\infty \norm{\bsy_t}^\alpha \rmd t < \infty\}$. 

We will also use the following notation. For a function $\bsy$ defined on $\Rset$ and $a<b$, we
write~$\bsy_{a,b}$ for the restriction of $\bsy$ to the interval $[a,b)$. With an abuse of notation,
it will denote either a function defined on $[a,b)$ or the function defined on $\Rset$ which is
equal to $\bsy$ on $[a,b)$ and vanishes outside $[a,b)$. We further define
$\bsy^* = \supnorm[\infty]{\bsy} = \sup_{t\in\Rset} \norm{\bsy_t}$,
$\bsy_{a,b}^* = \sup_{a \leq t \leq b} \norm{\bsy_t}$.  For a measurable $\bsy$ and $p>0$, we set
$\lpnorm[p]{\bsy}[p] = \int_{-\infty}^\infty |\bsy(t)|^p \rmd t$.

The backshift operator is  defined by $\shift^t f= f(\cdot-t)$ for all functions
$f$ defined on $\Rset$.

\section{Representations of tail  measures on $\spaceD$}
\label{sec:representation}
Let the space $\spaceD(\Rset,\Rset^d)$, hereafter simply $\spaceD$, be endowed by the $J_1$ topology
and the related Borel $\sigma$-field. See \Cref{sec:J1} for its definition and basic properties.  We
say that a subset $A$ of $\spaceD$ is separated from $\bszero$ if it is included in the complement
of an open neighborhood of the null map~$\bszero$. This is equivalent to the existence of real
numbers $a\leq b$ (depending on $A$) such that
\begin{align}
  \label{eq:bafz-in-D}
  \inf_{\bsy\in A} \sup_{a \leq t \leq b} \norm{\bsy_t} >0 \; .
\end{align}
See \cref{sec:J1} for a proof. We will denote by $\mcb_0$ the class of sets separated from
$\bszero$. A Borel measure $\mu$ on $\spaceD$ will be said $\mcb_0$-boundedly finite if
$\mu(\{\bszero\})=0$ and $\mu(A)<\infty$ for all Borel sets $A$ in $\mcb_0$.  The measurable sets in
$\mcb_0$ are measure determining for $\mcb_0$-boundedly finite measures. Thus, a $\mcb_0$-boundedly
finite measure is determined by the values $\tailmeasure(H)$ for all bounded or \nonnegative\
measurable maps $H$ with support in $\mcb_0$.  See \cite[Theorem~4.11]{kallenberg:2017} or
\cite[Theorem~4.1]{basrak:planinic:2019}.
\begin{definition}
  \label{def:tailmeasureonD}
  A tail measure on $\spaceD$ endowed with its Borel $\sigma$-field is a $\mcb_0$-boundedly finite
  Borel measure $\tailmeasure$ such that
  \begin{enumerate}[(i)]
  \item $\tailmeasure(\{\bszero\})=0$;
  \item \label{item:standardization-tailmeasure}
    $\tailmeasure(\{\bsy\in\spaceD:\norm{\bsy_0}>1\})=1$;
  \item there exists $\alpha>0$ such that $\tailmeasure(tA) = t^{-\alpha} \tailmeasure(A)$ for all
    Borel subsets $A$ of $\spaceD$.
  \end{enumerate}
\end{definition}
The positive number $\alpha$ will be called the tail index of $\tailmeasure$.  Since a tail measure
$\tailmeasure$ is boundedly finite, it holds that
$\tailmeasure(\{\bsy\in\spaceD:\norm{\bsy_t}>1\}) < \infty$ for all $t\in\Rset$ and $\tailmeasure$
is $\sigma$-finite.

By assumption, the measure $\tailmeasure$ restricted to $\{\bsy\in\spaceD:\norm{\bsy_0}>1\}$ is a
probability measure, so we can consider a $\spaceD$-valued random element $\bsY$ defined on a
probability space $(\Omega,\mcf,\pr)$ with distribution
$\tailmeasure(\cdot\cap\{\bsy\in\spaceD:\norm{\bsy_0}>1\})$, which will be called the tail process
associated to $\tailmeasure$.  The homogeneity of $\tailmeasure$ implies that $\norm{\bsY_0}$ has a
Pareto distribution with tail index $\alpha$ and is independent of the process $\bsTheta$ defined by
$\bsTheta = \norm{\bsY_0}^{-1} \bsY$, called the spectral tail process associated to $\tailmeasure$.
The spectral tail process can be viewed as a spectral decomposition of the tail measure with respect
to the pseudo-norm on $\spaceD$ defined by $\bsy\mapsto\norm{\bsy_0}$. That is, for any
\nonnegative\ or bounded measurable map $\map$, 
\begin{align}
  \label{eq:spectral-spectral}
  \int_{\spaceD} \map(\bsy) \ind{\norm{\bsy_0}>0} \tailmeasure(\rmd\bsy) = \int_0^\infty \esp[\map(u\bsTheta)]  \alpha u^{-\alpha-1}\rmd u \; .
\end{align}
See \cite[Section~5.2.1]{kulik:soulier:2020}.  

When $\tailmeasure$ is shift-invariant, the tail process inherits  a very important property.
\begin{lemma}
  \label{lem:tcf}
  Let $\tailmeasure$ be a \shiftinvariant\ tail measure on $\spaceD$ with tail index $\alpha$ and
  associated tail and spectral tail processes $\bsY$ and $\bsTheta$.  For every \nonnegative\
  measurable map $\map$ on $\spaceD$ and $x>0$, 
  \begin{align}
    \label{eq:TCF-Y}
    \esp[\map(\bsY) \ind{\norm{\bsY_t}>x}] & = x^{-\alpha}\esp[\map(x\shift^t\bsY) \ind{\norm{x\bsY_{-t}}>1}] \; , \\
    \label{eq:TCF-Theta}
    \esp[\map(\norm{\bsTheta_t}^{-1}\bsTheta) \norm{\bsTheta_t}^\alpha] & = \esp[\map(\shift^t\bsTheta) \ind{\norm{\bsTheta_{-t}}\ne0}]  \; .
  \end{align}
\end{lemma}
In the left-hand side of \cref{eq:TCF-Theta}, the quantity inside the expectation is understood as
zero when \mbox{$\norm{\bsTheta_t}=0$}.  These two properties are equivalent and the form \eqref{eq:TCF-Theta}
was originally obtained in the context of discrete time stationary time series and called the time
change formula by \cite{basrak:segers:2009}. The version \Cref{eq:TCF-Y} was obtained by
\cite{planinic:soulier:2018}. The proof of the result in continuous time is exactly the same as in
discrete time but we give the three-lines proof of \Cref{eq:TCF-Y} for completeness.
\begin{proof}
  By the definition of $\bsY$, the shift-invariance and homogeneity of $\tailmeasure$, we have
  \begin{align*}
    \esp[\map(\bsY) \ind{\norm{\bsY_t}>x}] 
    & = \int_{\spaceD} \map(\bsy) \ind{\norm{\bsy_t}>x} \ind{\norm{\bsy_0}>1} \tailmeasure(\rmd \bsy) \\
    & = x^{-\alpha} \int_{\spaceD} \map(x\shift^t\bsy) \ind{\norm{\bsy_0}>1} \ind{\norm{x\bsy_{-t}}>1} \tailmeasure(\rmd \bsy) \\
    & = x^{-\alpha}\esp[\map(x\shift^t\bsY) \ind{\norm{x\bsY_{-t}}>1}] \; .
  \end{align*}
\end{proof}

\subsection{Spectral representation}
\label{sec:spectral-representation}
The main result of this section is a representation theorem for \shiftinvariant\ tail measures. It
extends or complements several results of the literature. It extends
\cite[Theorem~2.4]{dombry:hashorva:soulier:2018} to the continuous time case and provides a
constructive proof (in a restricted context) to \cite[Proposition~2.8]{evans:molchanov:2018} which
deals with homogeneous measures in abstract cones.  For $r\in\Rset$, a map $\homap$ on a $\spaceD$
is said to be $r$-homogeneous if $\homap(t\bsy) = t^r\homap(\bsy)$ for all $t>0$ and
$\bsy\in\spaceD$.  For any random element $X$, we will use the notation $\esp[\delta_X]$
for the measure defined by $\esp[\delta_X](f)=\esp[f(X)]$ for $f$ a measurable map in the relevant
space.
\begin{theorem}
  \label{theo:Y-determines-nu}
  A Borel measure $\tailmeasure$ on $\spaceD$ with tail index $\alpha$ is a \shiftinvariant\ tail
  measure \ifft\ there exists a $\spaceD$-valued process $\bsZ$, called a spectral process for
  $\tailmeasure$, such that $\pr(\bsZ=\bszero)=0$, $\esp[\norm{\bsZ_0}^\alpha]=1$, 
  \begin{gather}
    \label{eq:spectral-representation}
    \tailmeasure = \int_0^\infty \esp[\delta_{u\bsZ}] \alpha u^{-\alpha-1}\rmd u \; , 
  \end{gather}
  and for all $a<b$, $t\in\Rset$ and bounded measurable $0$-homogeneous maps $\homap$ on $\spaceD$,
  \begin{gather}
    \label{eq:local-boundedness-Z}
    0 < \esp\left[\sup_{a \leq s \leq b} \norm{\bsZ_s}^\alpha \right] < \infty \; ,   \\
    \label{eq:tilt-shift}
    \esp[\norm{\bsZ_t}^\alpha \homap(\bsZ)] =     \esp[\norm{\bsZ_0}^\alpha \homap(\shift^t\bsZ)] \; .
  \end{gather}  
  Furthermore, a \shiftinvariant\ tail measure is entirely determined by its tail process $\bsY$
  whose distribution $\pr_{\bsY}$ is related to any spectral process $\bsZ$ by
  \begin{align}
    \label{eq:Y-in-terms-of-Z}
    \pr_{\bsY} = \esp\left[\norm{\bsZ_0}^\alpha \delta_{\frac{Y\bsZ}{\norm{\bsZ_0}^\alpha}}\right] \; ,
  \end{align}
  where $Y$ is a Pareto random variable with tail index $\alpha$.
\end{theorem}

\begin{proof} 
  If $\tailmeasure$ is defined by \Cref{eq:spectral-representation} with $\bsZ$ satisfying the
  stated properties, then it is a \shiftinvariant\ tail measure.  We prove the converse. Let
  $\tailmeasure$ be a \shiftinvariant\ tail measure and let $f$ be a fixed bounded positive
  continuous function on $\Rset$ such that $\int_{-\infty}^\infty f(t) \rmd t = 1$.  For
  $\bsy\in\spaceD$, define $\mcj(\bsy) = \int_{-\infty}^\infty f(t) \norm{\bsy_t}^\alpha \rmd t$.
  By the time change formula \Cref{eq:TCF-Theta}, we have
  $\esp[\norm{\bsTheta_t}^\alpha] = \pr(\bsTheta_{-t}\ne\bszero)\leq1$ for all $t\in\Rset$, thus
  $\int_{-\infty}^\infty \esp[\norm{\bsTheta_{s-t}}^\alpha] f(s) \rmd s<\infty$ for all $t$ by
  assumption on $f$.

  Let $T$ be a real-valued random variable, independent of $\bsY$ with density $f$.  The previous
  property can be expressed as $\esp[\mcj(\shift^T\bsTheta)]<\infty$. Thus
  $\pr(\mcj(\shift^T\bsY)<\infty)=1$ and since $\bsY$ is \cadlag\ and $|\bsY_0|>1$,
  $\pr(\mcj(\bsY)>0)=1$.  Therefore, we can define a process $\bsZ$ by
  \begin{align}
    \label{eq:construction-spectralprocess}
    \bsZ = (\mcj(\shift^T\bsY))^{-1/\alpha} \shift^T\bsY =    (\mcj(\shift^T\bsTheta))^{-1/\alpha} \shift^T\bsTheta \; .
  \end{align} 
  Since $\pr(\norm{\bsTheta_0}=1)=1$, it holds that $\pr(\bsZ=\bszero)=0$.  Let the measure in the
  right-hand side of \Cref{eq:spectral-representation} be denoted by~$\nu_f$. By the first part of
  the proof, $\nu_f$ is a tail measure on $\spaceD$, hence is $\sigma$-finite. Let $\map$ be a
  \nonnegative\ measurable map and $\epsilon>0$.  Applying the time change formula \Cref{eq:TCF-Y} and
  the homogeneity of the functional $\mcj$, we obtain, for an arbitrary $t\in\Rset$,
  \begin{align*} 
    \int_{\spaceD} 
    & \map(\bsy)\ind{\norm{\bsy_t}>\epsilon}\nu_f(\rmd\bsy) \\ 
    & = \int_0^\infty \esp[\map(u\bsZ)\ind{u\norm{\bsZ_t}>\epsilon}] \alpha u^{-\alpha-1} \rmd u \\ 
    & = \int_{-\infty}^\infty
      \int_0^\infty \esp\left[\map\left(\frac{u\shift^s\bsY)} {\mcj^{1/\alpha}(\shift^s\bsY)}\right)
      \ind{\frac{u\norm{\bsY_{t-s}}} {\mcj^{1/\alpha}(\shift^s\bsY)}>\epsilon} \right] \alpha u^{-\alpha-1} \rmd u f(s) \rmd s
    \\
    & = \epsilon^{-\alpha} \int_{-\infty}^\infty \int_0^\infty \esp\left[ \frac{\map(\epsilon
      u\shift^{s-t}\shift^t\bsY) \ind{u\norm{\bsY_{t-s}}>1}} {{\mcj(\shift^{s-t}\shift^t\bsY)}} \right]
      \alpha u^{-\alpha-1} \rmd u f(s) \rmd s \\ 
    & = \epsilon^{-\alpha} \int_{-\infty}^\infty \int_0^\infty \esp\left[
      \frac{\map(\epsilon \shift^t\bsY) \ind{\norm{\bsY_{s-t}}>u}} {{\mcj(\shift^t\bsY)}} \right] \alpha
      u^{\alpha-1} \rmd u f(s) \rmd s \\ 
    & = \epsilon^{-\alpha} \int_{-\infty}^\infty \esp\left[
      \frac{\map(\epsilon \shift^t\bsY) \norm{\bsY_{s-t}}^\alpha} {{\mcj(\shift^t\bsY)}} \right] f(s) \rmd s 
      = \epsilon^{-\alpha} \esp[\map(\epsilon \shift^t\bsY)] \\
    &  = \epsilon^{-\alpha} \int_{\spaceD} \map(\epsilon \shift^t\bsy) \ind{\norm{\bsy_0}>1} \tailmeasure(\rmd\bsy)  
      = \int_{\spaceD} \map(\bsy) \ind{\norm{\bsy_t}>\epsilon} \tailmeasure(\rmd\bsy) \; .
  \end{align*} 
  Let $\epsilon>0$ and $\map$ be a non-negative measurable map on $\spaceD$ such that $\map(\bsy) = 0$ if
  $\bsy^*\leq\epsilon$. Then the previous identity yields 
  \begin{align*} 
    \nu_f(\map) & = \int_{\spaceD} \frac{\map(\bsy)\exc_f(\epsilon^{-1}\bsy)}{\exc_f(\epsilon^{-1}\bsy)} \nu_f(\rmd \bsy) 
              = \int_{-\infty}^\infty \int_{\spaceD} \frac{\map(\bsy)}{\exc_f(\epsilon^{-1}\bsy)}
               \ind{\norm{\bsy_t}>\epsilon} \nu_f(\rmd \bsy) \rmd t  \\
             & = \int_{-\infty}^\infty \int_{\spaceD} \frac{\map(\bsy)}{\exc_f(\epsilon^{-1}\bsy)}
               \ind{\norm{\bsy_t}>\epsilon} \tailmeasure(\rmd \bsy) \rmd t  
              = \int_{\spaceD} \map(\bsy) \frac{\exc_f(\epsilon^{-1}\bsy)}
               {\exc_f(\epsilon^{-1}\bsy)} \tailmeasure(\rmd \bsy) = \tailmeasure(H) \; .
  \end{align*} 
  As already noted, the class of such maps $\map$ is measure determining, thus we have proved
  \Cref{eq:spectral-representation}. Since $\bsZ$ depends only on the tail process $\bsY$, this
  shows that the tail measure is completely determined by its tail process.

  We now prove the other stated properties of $\bsZ$.  For all $a<b$,
  \begin{align*} 
    \esp\left[\sup_{a \leq s \leq b} \norm{\bsZ_s}^\alpha \right] 
    & = \int_0^\infty \pr(u \bsZ_{a,b}^*>1) \alpha u^{-\alpha-1} \rmd u \\ 
    & = \tailmeasure \left(\left\{\bsy\in\spaceD: \bsy_{a,b}^* >
      1\right\}\right) < \infty \; ,
  \end{align*} 
  since the set $\left\{\bsy\in\spaceD: \bsy_{a,b}^* > 1\right\}$ is separated from $\bszero$, hence
  has finite $\tailmeasure$-measure. This proves \Cref{eq:local-boundedness-Z}. 

  Finally, for any bounded measurable $0$-homogeneous map $\homap$ on $\spaceD$ and $t\in\Rset$, we have
  by \Cref{eq:spectral-representation},
  \begin{align*}
    \esp[\norm{\bsZ_t}^\alpha \homap(\bsZ)] 
    & = \int_0^\infty \esp[\homap(\bsZ)\ind{r\norm{\bsZ_t}>1} \alpha r^{-\alpha-1} \rmd r \\ 
    & =  \int_{\spaceD} \homap(\bsy)\ind{\norm{\bsy_t}>1} \tailmeasure(\rmd \bsy) \\
    & =  \int_{\spaceD} \homap(\shift^t\bsy)\ind{\norm{\bsy_0}>1} \tailmeasure(\rmd \bsy) = \esp[\homap(\shift^t\bsY)] \; .
  \end{align*}
  In the last line, we used the \shiftinvariance\ of $\tailmeasure$ and the definition of the tail
  process. For $t=0$, this yields \Cref{eq:Y-in-terms-of-Z}. Replacing $\homap$ by $\homap\circ \shift^t$
  yields \Cref{eq:tilt-shift}.
\end{proof}
One important difference between the tail process and spectral processes related to a tail measure
is that the former is unique in distribution and the latter is not.  The terminology is a bit
confusing since in general, a spectral tail process $\bsTheta$ is not a spectral process. The only
case where a spectral tail process is also a spectral process is when $\pr(\bsTheta_t=0)=0$ for all
$t$; in that case the time change formula \Cref{eq:TCF-Theta} is equivalent to \Cref{eq:tilt-shift}.

A tail process $\bsY$ satisfies $\pr(\norm{\bsY_0}>1)=1$ and the time change formula
\Cref{eq:TCF-Y}. As a consequence of \Cref{theo:Y-determines-nu}, we show that there is a
one-to-one correspondence between $\spaceD$-valued processes which satisfy these two properties and
an additional boundedness condition and \shiftinvariant\ tail measures on $\spaceD$.
\begin{corollary}
  \label{coro:equivalence-Y-nu}
  Let $\bsY$ be a random element in $\spaceD$ such that
  \begin{itemize}
  \item  $\pr(|\bsY_0|>1)=1$;
  \item the time change formula \Cref{eq:TCF-Y} holds;
  \item for all $a<b$, 
    \begin{align}
      \label{eq:condition-bnddlfnt}
      \int_a^b \esp \left[ \frac1{\int_a^b \ind{\norm{\bsY_{t-s}}>1} \rmd t} \right] \rmd s < \infty \; .
    \end{align}
  \end{itemize}
  Then there exists a unique \shiftinvariant\ tail measure $\tailmeasure$ such that the distribution
  of $\bsY$ is~$\tailmeasure$ restricted to the set $\{\bsy\in\spaceD:|\bsy_0|>1\}$.
\end{corollary}
Note that \cref{eq:condition-bnddlfnt} is also a necessary condition since if $\bsY$ is the tail
process associated to a \shiftinvariant\ tail measure $\tailmeasure$, then for $s\in[a,b)$,
$\pr(\int_a^b \ind{\norm{\bsY_{t-s}}>1}\rmd t >0)=1$ and
\begin{align*}
  \int_a^b \esp \left[ \frac1{\int_a^b \ind{\norm{\bsY_{t-s}}>1} \rmd t} \right] \rmd s
  & = \int_{\spaceD}   \int_a^b \frac{\ind{\norm{\bsy_0}>1}}{\int_a^b \ind{\norm{\bsy_{t-s}}>1} \rmd t}  \rmd s \,
     \ind{\int_a^b\ind{\norm{\bsy_{t-s}}>1}\rmd t > 0} \tailmeasure(\rmd\bsy)  \\
  &  = \int_{\spaceD} \ind{\int_a^b\ind{\norm{\bsy_s}>1}\rmd s>0} \tailmeasure(\rmd\bsy) 
    = \int_{\spaceD} \ind{\sup_{a\leq s \leq b}{\norm{\bsy_s}}>1} \tailmeasure(\rmd\bsy) \; .
\end{align*}
The last quantity is finite since a tail measure is boundedly finite by definition.
\begin{proof}
  The properties of $\bsY$ used to define the process $\bsZ$ in
  \Cref{eq:construction-spectralprocess} are those assumed here so we can define $\bsZ$ and a
  measure $\tailmeasure$ on $\spaceD$ by \Cref{eq:spectral-representation}. This measure is
  homogeneous and \shiftinvariant\ by construction and we must prove that $\bsY$ is the tail process
  associated to~$\tailmeasure$ and that $\tailmeasure$ is boundedly finite on
  $\spaceD\setminus\{\bszero\}$. By definition and by the time change formula, we have, for any
  bounded measurable map $\map$,
  \begin{align*}
    \tailmeasure(\map\ind{\norm{\bsy_0}>1})
    & = \int_0^\infty \esp\left[ \frac{\map(r\shift^T\bsY)\ind{r\norm{\bsY_{-T}}>1}}{\mcj(\shift^T\bsY)} \right] \alpha r^{-\alpha-1} \rmd r \\
    & = \int_0^\infty \esp\left[ \frac{\map(\bsY)\ind{\norm{\bsY_{T}}>r}}{\mcj(\bsY)} \right] \alpha r^{\alpha-1} \rmd r \\
    & = \esp\left[ \frac{\map(\bsY)\norm{\bsY_{T}}^\alpha}{\mcj(\bsY)} \right] 
      = \int_{-\infty}^\infty \esp\left[ \frac{\map(\bsY)\norm{\bsY_{t}}^\alpha}{\mcj(\bsY)} \right] f(t) \rmd t = \esp[\map(\bsY)] \; . 
  \end{align*}
  We now prove that $\tailmeasure$ is boundedly finite. As already noted, this is equivalent to
  proving that $\esp[(\bsZ_{a,b}^*)^\alpha] <\infty$ for all $a<b$.  Define the map $\exc_{a,b}$ on
  $\spaceD$ of exceedances between $a$ and~$b$ by
  $\exc_{a,b}(\bsy) = \int_a^b \ind{\norm{\bsy_s}>1} \rmd s$.  By definition of $\bsZ$, we have
  \begin{align*}
     \esp[(\bsZ_{a,b}^*)^\alpha] 
    &  = \esp\left[ \frac{(\shift^T\bsY_{a,b}^*)^\alpha}{\mcj(\shift^T\bsY)} \right]  
     = \int_0^\infty \esp\left[ \frac{\ind{r\shift^T\bsY_{a,b}^*>1}}{\mcj(\shift^T\bsY)} \right]  \alpha r^{-\alpha-1} \rmd r \\
    & = \int_a^b \int_0^\infty  \esp\left[ \frac{\ind{r\bsY_{s-T}>1}}{\mcj(\shift^T\bsY)\exc_{a,b}(r\shift^{T}\bsY)} \right]
      \alpha r^{-\alpha-1} \rmd r \rmd s \; .
  \end{align*}
  Applying now the time change formula \Cref{eq:TCF-Y} and the definition of $T$ yields 
  \begin{align*}       
     \esp[(\bsZ_{a,b}^*)^\alpha] 
    & =\int_a^b \int_0^\infty \esp\left[ \frac{\ind{\bsY_{T-s}>r}}{\mcj(\shift^s\bsY)\exc_{a,b}(\shift^{s}\bsY)} \right] 
      \alpha r^{\alpha-1} \rmd r \rmd s 
     = \int_a^b \esp \left[ \frac{1}{\exc_{a,b}(\shift^s\bsY)} \right] \rmd s \; .
  \end{align*}
  The last term is finite by assumption, thus $\tailmeasure$ is boundedly finite on $\spaceD_0$.
\end{proof}
Before turning to other representations of tail measures, we state and prove a simple lemma which
will nevertheless be very important.  Define the exceedance functional $\mce$ (or occupation time of
the interval~$(1,\infty)$) on $\spaceD$ by
\begin{align*}
  \exc(\bsy) = \int_{-\infty}^\infty  \ind{\norm{\bsy_t}>1} \rmd t  \; .
\end{align*}
This map is well defined and takes value in $[0,\infty]$. If $\bsY$ is a tail process in $\spaceD$,
then $\pr(\exc(\bsY)>0)=1$. In discrete time, the exceedance functional $\exc_d$ is defined with the
integral replaced by a series, and since a tail process $\bsY$ satisfies $|\bsY_0|>1$ almost surely,
it holds that $\exc_d(\bsY)\geq1$ almost surely, hence $\esp[\exc_d^{-1}(\bsY)]\leq1$. In continuous
time no such trivial bound holds. However, it still holds that the expectation of the inverse of the
exceedance functional is finite.
\begin{lemma}
  \label{lem:expinvexcfini}
  Let $\tailmeasure$ be a \shiftinvariant\ tail measure on $\spaceD$ and $\bsY$ be its tail process. Then
  \begin{align}
    \label{eq:eureka}
    \esp \left[ \frac1{\exc(\bsY)} \right] 
    = \lim_{T\to\infty}  \frac1T \tailmeasure \left( \left\{\bsy\in\spaceD: \sup_{0\leq t\leq T} |\bsy_t| > 1 \right\}\right) < \infty \; .
  \end{align}
\end{lemma}
\begin{proof}
  Since $\bsY$ is \cadlag\ and $\pr(\norm{\bsY_0}>1)=1$, for all $s\in[0,T)$,
  $\pr(\bsY_{-s,T-s}^*>1)=1$. Thus, by definition of the tail process, the shift-invariance of
  $\tailmeasure$ and by Fubini theorem, we have
  \begin{align*}
    \int_0^1 \esp 
    & \left[ \frac1{\int_{-Ts}^{T(1-s)} \ind{|\bsY_t|>1} \rmd t } \right] \rmd s \\
    &  = \frac1T\int_0^T \esp \left[ \frac1{\int_{-s}^{T-s} \ind{|\bsY_t|>1} \rmd t } \right] \rmd s  \\
    &  = \frac1T\int_0^T \esp \left[ \frac1{\int_{-s}^{T-s} \ind{|\bsY_t|>1} \rmd t } \ind{\bsY_{-s,T-s}^*>1}\right] \rmd s  \\
    &  =  \frac1T\int_0^T \int_{\spaceD} \frac{\ind{|\bsy_0|>1}} {\int_{-s}^{T-s} \ind{|\bsy_t|>1} \rmd t }
      \ind{\bsy_{-s,T-s}^*>1} \rmd s \,  \tailmeasure(\rmd\bsy) \\
    &  =  \frac1T\int_0^T \int_{\spaceD} \frac{\ind{|\bsy_s|>1}} {\int_{0}^{T} \ind{|\bsy_t|>1} \rmd t }
      \ind{\bsy_{0,T}^*>1}\rmd s \,  \tailmeasure(\rmd\bsy) \\
    & = \frac1T \tailmeasure \left( \left\{\bsy\in\spaceD: \sup_{0\leq t\leq T} |\bsy_t| > 1 \right\}\right)  < \infty \; .
  \end{align*}
  The last term is finite since $\tailmeasure$ is boundedly finite.  Since we have just proved that
  for $T=1$
  \begin{align*}
    \int_0^1 \esp \left[ \frac1{\int_{-s}^{(1-s)} \ind{|\bsY_t|>1} \rmd t } \right] \rmd s < \infty \; , 
  \end{align*}
  and since for each $s$ the map
  $T\mapsto \left(\int_{-Ts}^{T(1-s)} \ind{|\bsY_t|>1} \rmd t \right)^{-1}$ is decreasing \wrt\ $T$,
  we obtain by the dominated convergence theorem that
  \begin{align*}
    \esp \left[ \frac1{\exc(\bsY)} \right] 
    & = \int_0^1 \esp \left[ \lim_{T\to\infty}\frac1{\int_{-Ts}^{T(1-s)} \ind{|\bsY_t|>1} \rmd t } \right] \rmd s \\
    & = \lim_{T\to\infty} \int_0^1 \esp \left[ \frac1{\int_{-Ts}^{T(1-s)} \ind{|\bsY_t|>1} \rmd t } \right] \rmd s  
     \leq \int_0^1 \esp \left[ \frac1{\int_{-s}^{(1-s)} \ind{|\bsY_t|>1} \rmd t } \right] \rmd s < \infty \; .
  \end{align*}
\end{proof}

The quantity $\esp[\exc^{-1}(\bsY)]$ will be seen to be important enough to deserve a name.
\begin{definition}
    \label{def:def-candidate}
  Let $\tailmeasure$ be a \shiftinvariant\ tail measure on $\spaceD$. 
  The candidate extremal index $\candidate$ is defined by
  \begin{align*}
    \candidate = \esp\left[ \frac1{\exc(\bsY)} \right] \; .
  \end{align*}
\end{definition}
\Cref{lem:expinvexcfini} implies that $\candidate\in[0,\infty)$. As appear in the proof of
\cref{lem:expinvexcfini}, an expression of $\candidate$ in terms of $\tailmeasure$ only is given by
\begin{align*}
  \exc(\bsY) = \lim_{T\to\infty}    \frac1T \tailmeasure \left( \left\{\bsy\in\spaceD: \sup_{0\leq t\leq T} |y_t| > 1 \right\}\right)  \; .
\end{align*}

\subsection{Mixed moving average representations}
\label{sec:mma-representations}
We say that a tail measure $\tailmeasure$ has a mixed moving average representation if there exists
$\candidate\in(0,\infty)$ and a process $\bsQ\in\spaceD$ such that $\pr(\bsQ^*=1)=1$ and 
\begin{align}
  \label{eq:dissipativerepresentation}
  \tailmeasure = \candidate \int_{-\infty}^\infty \int_{0}^\infty \esp[\delta_{u\shift^t\bsQ}] \alpha u^{-\alpha-1} \rmd u \rmd t \; .
\end{align}
Since $\tailmeasure(\{\bsy\in\spaceD:\norm{\bsy_0}>1\})=1$, this representation entails
\begin{align}
 \label{eq:candidate-Q}
  \candidate \int_{-\infty}^\infty \esp[\norm{\bsQ_t}^\alpha]  \rmd t = 1 \; .
\end{align}
Since $\tailmeasure$ is boundedly finite, it also holds that  for all $a \leq b$,
\begin{align}
  \label{eq:Q-local-finite}
  0 <   \candidate \int_{-\infty}^\infty \esp[ (\bsQ_{a-t,b-t}^*)^\alpha] \rmd t 
  = \tailmeasure(\{\bsy\in\spaceD:\bsy_{a,b}^*>1\}) < \infty \; ,
\end{align}
The latter property implies that $\pr(\lim_{|t|\to\infty} \norm{\bsQ_t}=0)=1$. Since moreover $\bsQ$
is \cadlag\ and $\pr(\bsQ^*=1)=1$, we obtain that $\pr(0<\exc(r\bsQ)<\infty)=1$ for all
$r\geq1$. Then, defining the map $\map_e$ by $\map_e(\bsy)=\exc^{-1}(\bsy)\ind{\norm{\bsy_0}>1}$, we have
\begin{align*}
  \esp[\exc^{-1}(\bsY)] 
  & = \tailmeasure(\map_e)  
    = \candidate \int_{-\infty}^\infty \int_0^\infty \esp\left[ \frac{\ind{r\norm{\bsQ_{-t}}>1}} {\exc(r\bsQ)} \right] \alpha r^{-\alpha-1} \rmd r \rmd t \\
  & = \candidate  \int_1^\infty \esp\left[ \int_{-\infty}^\infty\frac{\ind{r\norm{\bsQ_{-t}}>1} \rmd t} {\exc(r\bsQ)} \right] \alpha r^{-\alpha-1} \rmd r
    = \candidate  \int_1^\infty \esp\left[ \frac{\exc(r\bsQ)} {\exc(r\bsQ)} \right] \alpha r^{-\alpha-1} \rmd r    = \candidate \; .
\end{align*}
Thus $\candidate$ is the candidate extremal index of \cref{def:def-candidate}, which justifies the
use of the same notation.

Conversely, if $\bsQ$ is a random element in $\spaceD$ such that $\pr(\bsQ^*=1)=1$ and
\cref{eq:Q-local-finite} holds for all $a \leq b$,
then (\ref{eq:dissipativerepresentation}), with
$\candidate = \left(\int_{-\infty}^\infty \esp[\norm{\bsQ_t}^\alpha]\rmd t\right)^{-1}$, defines a
\shiftinvariant\ tail measure which is supported on $\spaceD_0$.

The purpose of this section is to show that a tail measure supported on $\spaceD_0$ admits the
representation (\ref{eq:dissipativerepresentation}) with $\bsQ$ satisfy \Cref{eq:Q-local-finite} in
addition to $\pr(\bsQ^*=1)=1$.  We will need two preliminary lemmas.
\begin{lemma}
  \label{lem:independence-tilted}
  Assume that $\pr(\exc(\bsY)=\infty)=0$.  Let $\shinvhomap$ be a \nonnegative\ \shiftinvariant\ and
  $0$-homogeneous measurable map on $\spaceD(\Rset)$. Then, for all $x>1$,
  \begin{align}
    \label{eq:independence-tilted}
    \esp\left[\frac{\shinvhomap(\bsY)\ind{\bsY^*>x}}{\exc(\bsY)}\right] = x^{-\alpha}  \esp\left[\frac{\shinvhomap(\bsY)}{\exc(\bsY)}\right] \; .
  \end{align}
\end{lemma}

\begin{proof}
  If $\pr(\exc(\bsY)=\infty)=0$, by Fubini theorem and the time change formula \cref{eq:TCF-Y}, we
  have for all $x>0$,
  \begin{align*}
    \esp\left[\exc(x\bsY)\ind{\exc(x\bsY)=\infty} \right]
    & =    \int_{-\infty}^\infty  \esp \left[ \ind{\exc(x\bsY)=\infty}\ind{x\norm{\bsY_s}>1} \right] \rmd s \\
    & = x^\alpha \int_{-\infty}^\infty \esp \left[ \ind{\exc(\bsY)=\infty}\ind{\norm{\bsY_s}>x} \right] \rmd s = 0 \; .
  \end{align*}
  Thus $\pr(\exc(\bsY)=\infty)=0$ implies $\pr(\exc(x\bsY)=\infty)=0$ for all $x>0$.  For a \cadlag\
  function, $\bsy^*>x$ is equivalent to $\exc(x^{-1}\bsy)>0$, thus we can write
  \begin{align*}
    \ind{\bsy^*>x} = \frac{\exc(x^{-1}\bsy)}{\exc(x^{-1}\bsy)} \; .
  \end{align*}
  Thus, by Fubini theorem and the time change formula \Cref{eq:TCF-Y}, we obtain
  \begin{align*}
    \esp\left[\frac{\shinvhomap(\bsY)\ind{\bsY^*>x}}{\exc(\bsY)}\right]     
    & = \esp \left[ \frac{\shinvhomap(\bsY)} {\exc(\bsY)} \frac{\exc(x^{-1}\bsY)}{\exc(x^{-1}\bsY)} \right] \\
    &  =  \int_{-\infty}^\infty 
      \esp \left[ \frac{\shinvhomap(\bsY)} {\exc(\bsY)} \frac{\ind{\norm{\bsY_t}>x}}{\exc(x^{-1}\bsY)} \right] \rmd t 
     = x^{-\alpha} \int_{-\infty}^\infty 
      \esp \left[ \frac{\shinvhomap(\bsY) } {\exc(x\bsY)} \frac{\ind{x\norm{\bsY_{-t}}>1}}{\exc(\bsY)} \right] \rmd t \\
    & = x^{-\alpha}  \esp \left[ \frac{\shinvhomap(\bsY) } {\exc(x\bsY)} \frac{\exc(x\norm{\bsY})}{\exc(\bsY)} \right]
      = x^{-\alpha}  \esp \left[ \frac{\shinvhomap(\bsY) } {\exc(\bsY)} \ind{\exc(x\bsY)>0} \right]
      = x^{-\alpha} \esp\left[\frac{\shinvhomap(\bsY)}{\exc(\bsY)}\right] \; .
  \end{align*}
  The last equality holds since $\pr(\exc(x\bsY)>0)=1$ for all $x>1$.
\end{proof}
A set $A$ is said to be homogeneous if $tA=A$ for all $t>0$.
\begin{lemma}
  \label{lem:nunull}
  Let $\tailmeasure$ be a shift-invariant tail measure with tail process $\bsY$.  Let $A$ be a
  homogeneous and shift-invariant set. Then $\tailmeasure(A)=0\iff\pr(\bsY\in A)=0$.  If
  \Cref{eq:dissipativerepresentation} holds, then $\tailmeasure(A)=0\iff\pr(\bsQ\in A)=0$.
\end{lemma}

\begin{proof}
  For any set $A$,
  $\pr(\bsY\in A) = \int_{\spaceD} \1A(\bsy)\ind{\norm{\bsy_0}>1} \tailmeasure(\rmd\bsy)$. Thus
  $\tailmeasure(A)=0$ implies $\pr(\bsY\in A)=0$.  Conversely, if $A$ is shift-invariant and
  homogeneous, the shift-invariance and homogeneity properties of $\tailmeasure$ yield, for all
  $t\in\Rset$ and $x>0$,
  \begin{align*}
    \pr(\bsY\in A) = \int_{\spaceD} \1A(\bsy)  \ind{\norm{\bsy_0}>1} \tailmeasure(\rmd \bsy) 
    = x^{-\alpha} \int_{\spaceD} \1A(\bsy) \ind{\norm{x\bsy_t}>1} \tailmeasure(\rmd \bsy) \; .
  \end{align*}
    If $\pr(\bsY\in A)$=0, this yields by integration \wrt\ $t$, for all $x>0$, 
  \begin{align*}
    0 =  \int_{\spaceD} \1A(\bsy) \exc(x\bsy) \tailmeasure(\rmd \bsy) \; .
  \end{align*}
  As $x\to+\infty$, this yields by monotone convergence,
  \begin{align*}
    0 =  \int_{\spaceD} \1A(\bsy) \left(\int_{-\infty}^\infty \ind{\norm{\bsy_t}>0}\rmd t \right) \tailmeasure(\rmd \bsy) \; .
  \end{align*}
  Thus $\bsy\in A$ implies $\int_{-\infty}^\infty \ind{\norm{\bsy_t}>0}\rmd t=0$, for
  $\tailmeasure$-almost all $\bsy$. Since $\bsy$ is a \cadlag\ function, this implies that
  $\bsy=\bszero$. Thus $A$ is a set of $\tailmeasure$-measure 0 and possibly also contains
  $\bszero$. Since $\tailmeasure(\{\bszero\})=0$, this proves that $\tailmeasure(A)=0$.

  If \Cref{eq:dissipativerepresentation} holds and $A$ is homogeneous and shift-invariant, then
  \begin{align*}
    \tailmeasure(A) & = \int_{-\infty}^\infty \int_0^\infty \pr(u \shift^t\bsQ \in A) \alpha u^{-\alpha-1} \rmd u \rmd t =
                      \infty\times\pr(\bsQ\in A) \; .
  \end{align*}
  Thus $\tailmeasure(A)=0\iff\pr(\bsQ\in A)=0$.
\end{proof}
We now prove the mentioned result. It is a complement of \cite[Theorem~3]{dombry:kabluchko:2016}
which is only concerned with max-stable processes (and fields), with a more constructive proof. It
is also similar to \cite[Theorem~3]{dombry:hashorva:soulier:2018} which deals with spectral
processes. We state our result in terms of tail processes and tail measures. We also refer to
\cite{dombry:kabluchko:2016} for the link between these representations and the
dissipative/conservative decomposition of non-singular flows. Recall that $\spaceD_0$ is the set of
\cadlag\ function $\bsy$ such that $\lim_{|t|\to\infty} \norm{\bsy_t}=0$ and $\spaceD_\alpha$ is the
set of \cadlag\ function $\bsy$ such that $\int_{-\infty}^\infty \norm{\bsy_t}\rmd t<\infty$.
\begin{theorem}
  \label{theo:equivalences-dissipative}
  Let $\tailmeasure$ be a \shiftinvariant\ tail measure with tail process $\bsY$. The following statements are
  equivalent.
  \begin{enumerate}[(i)]
  \item \label{item:dissipative} There exists a random element $\bsQ$ in $\spaceD$ such that
    $\pr(\bsQ^*=1)=1$ and  \Cref{eq:dissipativerepresentation} holds;
  \item \label{item:Yto0} $\pr(\bsY\in\spaceD_0)=1$;
  \item \label{item:integrability-Y} $\pr(\bsY\in\spaceD_\alpha)=1$; 
  \item \label{item:exceedances-Y} $\pr(\exc(\bsY)<\infty)=1$; 
  \item \label{item:nu-supported-on-D0} $\tailmeasure$ is supported on $\spaceD_0$;
  \item \label{item:nu-supported-on-Dalpha}  $\tailmeasure$ is supported on $\spaceD_\alpha$.
  \end{enumerate}
  If these conditions hold, then $\candidate=\esp[\exc^{-1}(\bsY)]$, $\candidate>0$ and the
  distribution of $\bsQ$ is given by
  \begin{align}
    \label{eq:loi-Q}
    \pr_{\bsQ} = \candidate^{-1} \esp \left[ \frac{\delta_{\frac{\bsY}{\bsY^*}}}{\exc(\bsY)} \right] \; .
  \end{align}
\end{theorem}

\begin{proof}
  Since $\spaceD_0$ and $\spaceD_\alpha$ and their complements in $\spaceD$ are \shiftinvariant\ and
  homogeneous, \Cref{lem:nunull} implies that \Cref{item:nu-supported-on-D0}$\iff$\Cref{item:Yto0}
  and \Cref{item:nu-supported-on-Dalpha}$\iff$\Cref{item:integrability-Y}. Thus we will only prove
  the equivalence of \Cref{item:dissipative}, \Cref{item:Yto0}, \Cref{item:integrability-Y} and
  \Cref{item:exceedances-Y}.

  If (\ref{eq:dissipativerepresentation}) holds,  $\tailmeasure$ being
  boundedly finite by assumption, we have for all $a \leq b$,
  \begin{align*}
    \infty & >   \tailmeasure(\{\bsy\in\spaceD:\bsy_{a,b}^*>1\}) \\
           & = \candidate \int_{-\infty}^\infty \int_{0}^\infty \pr(u \bsQ_{a-t,b-t}^*>1) \alpha u^{-\alpha-1} \rmd u \rmd t \\
           & = \candidate \int_{-\infty}^\infty \esp[ (\bsQ_{a-t,b-t}^*)^\alpha] \rmd t \; .
  \end{align*}
  The finiteness of this integral implies that $\pr(\bsQ\in\spaceD_0\cap\spaceD_\alpha)=1$.  Thus
  \Cref{item:dissipative} implies both \Cref{item:Yto0} and \Cref{item:integrability-Y} by
  \cref{lem:nunull}. Obviously, \Cref{item:Yto0} implies \Cref{item:exceedances-Y} and we next prove
  that \Cref{item:exceedances-Y} implies \Cref{item:dissipative}.

\paragraph{Proof of \Cref{item:exceedances-Y}$\implies$\Cref{item:dissipative}}
Since $\pr(\exc(\bsY)<\infty)=1$ and $\bsY$ is \cadlag, $\candidate=\esp[\exc^{-1}(\bsY)]>0$ and
$\candidate<\infty$ by \Cref{lem:expinvexcfini}. Let~$\bsQ$ be a $\spaceD$-valued random process
whose distribution is given by \Cref{eq:loi-Q}.  Let $\map$ be a \nonnegative\ measurable map on
$\spaceD$ with support separated from $\bszero$. Then, there exists $\epsilon>0$ such that
$\map(\bsy)=0$ when $\bsy^*\leq\epsilon$.  Applying Fubini's theorem and the shift-invariance and
homogeneity of $\tailmeasure$, we obtain
  \begin{align*}
    \tailmeasure(\map) & = \epsilon^{-\alpha} \int_{\spaceD} \map(\epsilon \bsy) \tailmeasure(\rmd\bsy) 
                         = \epsilon^{-\alpha} \int_{\spaceD} \map(\epsilon \bsy) \ind{\exc(\bsy)>0} \tailmeasure(\rmd\bsy) \\
                       & = \epsilon^{-\alpha} \int_{\spaceD} \map(\epsilon\bsy)
                         \frac{\exc(\bsy)}{\exc(\bsy)} \ind{\exc(\bsy)>0} \tailmeasure(\rmd\bsy) 
                         = \epsilon^{-\alpha} \int_{-\infty}^\infty \int_{\spaceD} \map(\epsilon\bsy)
                         \frac{\ind{\norm{\bsy_t}>1}}{\exc(\bsy)} \tailmeasure(\rmd\bsy) \rmd t \\
                       & = \epsilon^{-\alpha} \int_{-\infty}^\infty \int_{\spaceD} \map(\epsilon\bsy)
                         \frac{\ind{\norm{\bsy_0}>1}}{\exc(\bsy)} \tailmeasure(\rmd\bsy) \rmd t 
                         =  \epsilon^{-\alpha}\int_{-\infty}^\infty \esp\left[\frac{\map(\epsilon B^t\bsY)}{\exc(\bsY)}\right] \rmd t \; .
  \end{align*}
  Define the map $\tilde{\map}$ on $\spaceD$ by
  $\tilde{\map}(\bsy)=\int_{-\infty}^\infty \map(\shift^t\bsy)\rmd t$. The map $\tilde{\map}$ is well
  defined but possibly infinite (since $\map$ is non-negative) and \shiftinvariant.  Applying
  \eqref{eq:independence-tilted} yields
  \begin{align*} 
    \tailmeasure(\map) 
    & = \epsilon^{-\alpha} \candidate \int_{-\infty}^\infty \esp\left[\frac{\map(\epsilon \bsY^*\frac{B^t\bsY}{\bsY^*})}{\bsY^*}\right] \rmd t \\
    & = \epsilon^{-\alpha} \candidate  \esp\left[\frac{\tilde{\map}(\epsilon \bsY^*\frac{\bsY}{\bsY^*})}{\bsY^*}\right] 
      =  \epsilon^{-\alpha}  \candidate  \int_1^\infty \esp\left[\frac{\tilde{\map}(\epsilon r\frac{\bsY}{\bsY^*})}{\exc(\bsY)}\right]
      \alpha r^{-\alpha-1} \rmd r \\
    & =  \epsilon^{-\alpha}  \candidate \int_{-\infty}^\infty \int_1^\infty \esp\left[\frac{\map(\epsilon r\frac{B^t\bsY}{\bsY^*})}{\exc(\bsY)}\right]
      \alpha r^{-\alpha-1} \rmd r \rmd t \; .
  \end{align*}
  Applying now \cref{eq:loi-Q}, we obtain
  \begin{align*}
    \tailmeasure(\map) 
    & =  \epsilon^{-\alpha}  \candidate \int_{-\infty}^\infty \int_1^\infty \esp\left[\map(\epsilon rB^t\bsQ)\right] \alpha r^{-\alpha-1} \rmd r \rmd t \\
    & =   \candidate \int_{-\infty}^\infty  \int_\epsilon^\infty \esp\left[\map( rB^t\bsQ)\right] \alpha r^{-\alpha-1} \rmd r \rmd t 
     =   \candidate \int_{-\infty}^\infty  \int_0^\infty \esp\left[\map(rB^t\bsQ)\right] \alpha r^{-\alpha-1} \rmd r \rmd t \; .
  \end{align*}
  In the last line, the lower bound of the integral is set equal to zero because $\bsQ^*=1$ by
  definition. This proves \Cref{item:dissipative}.

  We have proved that \Cref{item:dissipative}, \Cref{item:Yto0} and \Cref{item:exceedances-Y} are
  equivalent. Since \Cref{item:Yto0} implies \Cref{item:dissipative} and \Cref{item:dissipative}
  implies \Cref{item:integrability-Y}, we have also proved that \Cref{item:Yto0} implies
  \Cref{item:integrability-Y}. In discrete time, the converse is obvious but needs a proof in the
  present context. We will actually prove that \Cref{item:integrability-Y} implies
  \Cref{item:dissipative}.
  \paragraph{Proof of \Cref{item:integrability-Y}$\implies$\Cref{item:dissipative}}
  Assume that \Cref{item:integrability-Y} holds. Using the homogeneity and shift-invariance of
  $\tailmeasure$ and the fact that $\tailmeasure(\{\bszero\})=0$, we have, for any \nonnegative\
  measurable map $\map$,
  \begin{align*}
    \tailmeasure(\map) 
    & = \int_{\spaceD} \map(\bsy) \frac{\lpnorm[\alpha]{\bsy}[\alpha]}{\lpnorm[\alpha]{\bsy}[\alpha]}  \tailmeasure(\rmd \bsy)  
      = \int_{\spaceD}\int_0^\infty \int_{-\infty}^\infty \map(\bsy) \frac{\ind{u<|y_t|}}{\lpnorm[\alpha]{\bsy}[\alpha]} 
      \alpha u^{\alpha-1} \rmd u \,  \rmd t \, \tailmeasure(\rmd \bsy)   \\
    & = \int_{\spaceD}\int_0^\infty \int_{-\infty}^\infty \map(u\shift^t\bsy)
      \frac{\ind{1<|y_0|}}{S_\alpha(\bsy)}  \alpha u^{-\alpha-1} \rmd u \, \rmd t \, \tailmeasure(\rmd \bsy)    \\
    & = \int_0^\infty \int_{-\infty}^\infty \esp \left[ \frac{\map(u\shift^t\bsY)}  {S_\alpha(\bsY)} \right] \alpha u^{-\alpha-1} \rmd u \, \rmd t   \\
    & = \int_0^\infty \int_{-\infty}^\infty \esp \left[ \frac{(\bsY^*)^\alpha \map\left(\frac{u\shift^t\bsY}{\bsY^*}\right)}
      {\lpnorm[\alpha]{\bsY}[\alpha]}   \right] \alpha u^{-\alpha-1} \rmd u \, \rmd t  \; .
  \end{align*}
  Recall that $\candidate=\esp[\exc^{-1}(\bsY)]<\infty$. Thus the previous identity yields
  \begin{align*}
    \candidate 
    &  = \int_{-\infty}^\infty \int_0^\infty
      \esp \left[ \frac{(\bsY^*)^\alpha \ind{u\norm{\bsY_{-t}}>\bsY^*}}
      {\lpnorm[\alpha]{\bsY}[\alpha]\exc(u(\bsY^*)^{-1}\bsY)}   \right] \alpha u^{-\alpha-1} \rmd u \, \rmd t 
      =    \esp \left[ \frac{(\bsY^*)^\alpha}{\lpnorm[\alpha]{\bsY}[\alpha]} \right]\; .
  \end{align*}
  Thus \Cref{eq:dissipativerepresentation} holds with $\bsQ$ whose distribution is given by
  \begin{align*}
    \pr_{\bsQ} = \candidate^{-1} \esp \left[ \frac{(\bsY^*)^\alpha\delta_{\frac{\bsY}{\bsY^*}}}{\lpnorm[\alpha]{\bsY}[\alpha]} \right] \; .
  \end{align*}
  By construction, $\pr(\bsQ^*=1)=1$. Thus we have proved that \Cref{item:integrability-Y} implies
  \Cref{item:dissipative}.

This concludes the proof of \cref{theo:equivalences-dissipative}.
\end{proof}

An important consequence of \cref{theo:equivalences-dissipative} is the following equivalence:
\begin{align}
  \label{eq:conservative-nulextremalindex}
  \tailmeasure(\spaceD_0)=0 \iff \candidate = 0 \; .
\end{align}

In the course of the proof, we have obtained a representation of the tail measure in terms of the
spectral tail process.
\begin{corollary}
  \label{coro:Q-Theta}
  Let the equivalent conditions of \Cref{theo:equivalences-dissipative} hold.  Then
  \begin{align}
   \label{eq:candidate-theta}
    \candidate = \esp\left[\frac{(\bsTheta^*)^\alpha}{\lpnorm[\alpha]{\bsTheta}[\alpha]}\right] \; .
  \end{align}
  For all $\alpha$-homogeneous non-negative \shiftinvariant\ measurable maps $\shinvhomap$ on $\spaceD$,
  \begin{align}
    \label{eq:Q-Theta}
    \esp \left[ \frac{\shinvhomap(\bsTheta)} {\lpnorm[\alpha]{\bsTheta}[\alpha]} \right] = \candidate \esp[\shinvhomap(\bsQ)]  \; .
  \end{align}
  and
  \begin{align}
    \label{eq:nu-theta}
    \tailmeasure 
    &  = \int_{-\infty}^\infty \int_0^\infty
      \esp \left[ \frac{\delta_{u\shift^t\bsTheta}}{\lpnorm[\alpha]{\bsTheta}[\alpha]} \right] \alpha u^{-\alpha-1} \rmd u \rmd t  \; .
  \end{align}
\end{corollary}
Similarly to \Cref{coro:equivalence-Y-nu}, we obtain a one-to-one equivalence between
\shiftinvariant\ tail measures supported on $\spaceD_0$ and $\spaceD$-valued random elements $\bsQ$
which satisfy \Cref{eq:Q-local-finite}.
\begin{corollary}
  \label{coro:equivalence-nu-Q}
  Let $\bsQ$ be a $\spaceD$-valued random element such that $\pr(\bsQ^*=1)=1$ and
  \Cref{eq:Q-local-finite} holds for all real numbers $a \leq b$.  Then there exists a unique
  \shiftinvariant\ tail measure $\tailmeasure$ such that \Cref{eq:dissipativerepresentation}
  holds. The tail measure is supported on $\spaceD_0$ and the tail process associated to it can be
  defined in terms of $\bsQ$ by
  \begin{align}
    \label{eq:Y-in-terms-of-Q}
    \bsY = Y \frac{\shift^T\bsQ}{\norm{\bsQ_{-T}}} \; , 
  \end{align}
  where $\candidate = (\int_{-\infty}^\infty \esp[\norm{\bsQ_t}^\alpha]\rmd t)^{-1}$, $Y$ is a
  Pareto random variable with tail index $\alpha$, independent of $\bsQ$ and $T$ and $T$ is a real
  valued random variable whose joint distribution with $\bsQ$ is given by
  $\candidate \int_{-\infty}^\infty\esp[\delta_{\bsQ,t}\norm{\bsQ_{-t}}^\alpha] \rmd t$.
\end{corollary}

\begin{proof}
  Note that there is no issue with division by zero in \Cref{eq:Y-in-terms-of-Q} since
  \begin{align*}
    \pr(|\bsQ_{-T}|=0) = \candidate \int_{-\infty}^\infty \esp[\ind{|\bsQ_{-t}|=0} |\bsQ_{-t}|^\alpha] = 0 \; .
  \end{align*}
  Let $Y$ be a Pareto random variable with tail index $\alpha$. Let $\tailmeasure$ be defined by
  \cref{eq:dissipativerepresentation} and $\bsY$ its tail process. Then, for any \nonnegative\
  measurable map $\map$,
  \begin{align*}
    \esp[\map(\bsY)] & = \candidate\int_{-\infty}^\infty \int_0^\infty \esp[\map(u\shift^t\bsQ) \ind{\norm{u\bsQ_{-t}}>1}] \alpha u^{-\alpha-1} \rmd u \rmd t \\
                  & = \candidate\int_{-\infty}^\infty  \int_1^\infty \esp[\map(v\shift^t\bsQ/\norm{\bsQ_{-t}}) \norm{\bsQ_{-t}}^\alpha]
                    \alpha v^{-\alpha-1} \rmd v  \rmd t 
                    = \candidate   \esp[\map(Y\shift^T\bsQ/\norm{\bsQ_{-T}})] \; ,  
  \end{align*}
This proves the representation \cref{eq:Y-in-terms-of-Q} of the tail process.
\end{proof}

Before turning to the next issues, we note that the representations \Cref{eq:spectral-representation}
and \Cref{eq:dissipativerepresentation} can be interpreted in terms of Poisson point processes (PPP)
on $\spaceD$. A measure $\tailmeasure$ with the representation \Cref{eq:spectral-representation} is
the mean measure of a PPP $N$ which can be expressed as
\begin{align*}
  N = \sum_{i=1}^\infty \delta_{P_i\bsZ^{(i)}} \; , 
\end{align*}
with $\sum_{i=1}^\infty \delta_{P_i,\bsZ^{(i)}}$ a Poisson point process on $(0,\infty)\times\spaceD$  with mean measure
$\nualpha\otimes\pr_{\bsZ}$, $\nualpha(\rmd u) = \alpha u^{-\alpha-1} \rmd u$ and $\pr_{\bsZ}$ is the distribution of
the process $\bsZ$. If $\tailmeasure$ admits the representation \Cref{eq:dissipativerepresentation},
then $N$ can be expressed as
\begin{align*}
  N = \sum_{i=1}^\infty \delta_{P_i\shift^{T_i}\bsQ^{(i)}} \; , 
\end{align*}
where $\sum_{i=1}^\infty \delta_{T_i,P_i,\bsQ^{(i)}}$ is a PPP on
$\Rset\times(0,\infty)\times\spaceD$ with mean measure
$\candidate\leb\otimes\nualpha\otimes\pr_{\bsQ}$ (and $\pr_{\bsQ}$ denotes the distribution of
$\bsQ$).

\subsection{Anchoring maps}
\label{sec:anchoring-maps}
Let $\mci:\spaceD\to[-\infty,\infty]$ be a measurable map such that
\begin{align*}
  \mci(\shift^t\bsy) = \mci(\bsy)+t \; ,  \ \  \mbox{ for all } t\in\Rset \; .
\end{align*}
Such maps were introduced by \cite{basrak:planinic:2021} in the context of regularly varying
random fields indexed by a lattice and called anchoring maps, with the additional condition
$|\bsy_{\mci(\bsy)}|>1$ if $\mci(\bsy)\in\Rset$. We recall their properties in the case of time
series indexed by $\Zset$. Two examples are the infargmax functional denoted by $\mci_0$ and the
first exceedance over~1, denoted by $\mci_1$, defined on $(\Rset^d)^\Zset$ with values in
$\Zset\cup\{-\infty,+\infty\}$ by
\begin{align*}
  \mci_0(\bsy) & = \inf \{j\in\Zset: \norm{\bsy_j} = \sup_{i\in\Zset} \norm{\bsy_i}\} \; , \\
  \mci_1(\bsy) & = \inf \{j\in\Zset: \norm{\bsy_j} >1 \} \; , 
\end{align*}
with the convention that $\inf\emptyset=+\infty$.  For a discrete time tail process $\bsY$, that is
a random element in $(\Rset^d)^\Zset$ (endowed with the product topology) such that
$\pr(\norm{\bsY_0}>1)=1$ and which satisfies the time change formula \Cref{eq:TCF-Y}, if
$\pr(\lim_{|j|\to\infty} \norm{\bsY_j}=0)=1$, then for any anchoring map~$\mci$,
\begin{align}
  \label{eq:identities-discrete-time}
  \candidate  = \esp [\exc^{-1}(\bsY)] = \pr(\mci(\bsY)=0) = \frac{1}{\esp[\exc(\bsY)\mid \mci(\bsY)=0]} \; .
\end{align}
See \cite{basrak:planinic:2021}, \cite{planinic:soulier:2018} and
\cite[Chapter~5]{kulik:soulier:2020}. Also, the tail measure can be expressed as in
\Cref{eq:dissipativerepresentation} (with the integral over $\Rset$ replaced by a sum indexed by
$\Zset$) in terms of a sequence $\bsQ$ whose distribution is that of $(\bsY^*)^{-1}\bsY$
conditionally on $\mci(\bsY)=0$. The goal of this section is to investigate the role of anchoring
maps in continuous time, {to give a meaning to the conditioning in the two right-most
  terms of \cref{eq:identities-discrete-time} and to find if they can still be used as definitions
  of the candidate extremal index. It turns out that there are subtle differences and that anchoring
  maps are useful only when they fulfill the additional condition \cref{eq:condition-continuite}
  below, which is always true in discrete time, but not in continuous time as will be shown in 
  \cref{xmpl:subgaussian-maxstable}.}

We assume throughout this section that that the equivalent conditions of
\Cref{theo:equivalences-dissipative} hold.  Assume that there exists an anchoring map~$\mci$ such
that $\pr(\mci(\bsY)\in\Rset)=1$. Then, for all \nonnegative\ measurable maps $\map$ defined on
$\Rset \times\spaceD$, the representation (\ref{eq:dissipativerepresentation}) and the property
$\pr(\bsQ^*=1)=1$ yield
\begin{align*}
  \esp[\map(\mci(\bsY),\bsY)] 
  & = \candidate \int_{-\infty}^\infty \int_0^\infty \esp[ \map(\mci(u\bsQ)+t,u\shift^t\bsQ)\ind{u|\bsQ_{-t}|>1}] \alpha u^{-\alpha-1} \rmd u \, \rmd t \\
  & = \candidate \int_{-\infty}^\infty \int_1^\infty \esp[ \map(t,u\shift^{t-\mci(u\bsQ)} \bsQ)
    \ind{u|\bsQ_{\mci(u\bsQ)-t}|>1}] \alpha u^{-\alpha-1} \rmd u \, \rmd t   \\
  & = \candidate \int_{-\infty}^\infty \esp[ \map(t,Y\shift^{t-\mci(Y\bsQ)} \bsQ) \ind{Y|\bsQ_{\mci(Y\bsQ)-t}|>1}]  \, \rmd t  \; ,
\end{align*}
with $Y$ a random variable with a Pareto distribution with index $\alpha$, independent of
$\bsQ$. This shows that the distribution of $\mci(\bsY)$ is absolutely continuous \wrt\ Lebesgue's
measure with  density $f_\mci$ given by 
\begin{align}
  \label{eq:density-mci}
  f_{\mci}(t) = \candidate \pr(Y|\bsQ_{\mci(Y\bsQ)-t}|>1)] \; , 
\end{align}
and a regular version of the conditional distribution of $\bsY$ given $\mci(\bsY)$ is
\begin{align}
  \label{eq:conditionally=t}
  \esp[\map(\bsY)\mid \mci(\bsY)=t]  
  = \frac{\esp[ \map(Y\shift^{t-\mci(Y\bsQ)}\bsQ) \ind{Y|\bsQ_{\mci(Y\bsQ)-t}|>1}]}{\pr(Y|\bsQ_{\mci(Y\bsQ)-t}|>1)} \; , 
\end{align}
with the usual convention $\tfrac00=0$ and $\map$ a bounded or \nonnegative\ measurable map on
$\spaceD$.  Assume furthermore that 
\begin{align}
\label{eq:condition-continuite}
  \pr(Y|\bsQ_{\mci(Y\bsQ)}|>1)  = 1 \; .
\end{align}
Then {$f_\mci$ is left-continuous at 0 and} $\candidate = f_\mci(0)$ and
\Cref{eq:conditionally=t} yields, for $t=0$ and a \shiftinvariant\ measurable map $\shinvmap$ on
$\spaceD$,
\begin{align*}
  \esp[\shinvmap(\bsY)\mid \mci(\bsY)=0]    = \esp[ \shinvmap(Y\bsQ) ] \; .
\end{align*}
Thus we can formally extend the discrete time results to anchoring maps satisfying
\cref{eq:condition-continuite}, in particular \cref{eq:conditionally=t} yields
\begin{align}
  \label{eq:candidate-conditional}
  \candidate = \frac1{\esp[\exc(\bsY)\mid \mci(\bsY)=0]} \; .
\end{align}
{These formulae are useful when computing conditional distributions is easier than
  computing expectations under a changed of measure. See \Cref{sec:shot-noise} for such an example.}

\begin{example}
  \label{xmpl:infargmax-general}
  Let $\mci_0$ be the $\inf\arg\max$ functional, that is
  \begin{align*}
    \mci_0(\bsy) = \inf\{t\in\Rset: \bsy^*\in\{|y_t|,|y_{t^-}|\}\} \; .
  \end{align*}
  If $\pr(\bsY\in\spaceD_0)=1$, then $\tailmeasure(\spaceD_0^c)=0$ by
  \Cref{theo:equivalences-dissipative}, thus
  $\tailmeasure(\{\bsy\in\spaceD:\mci_0(\bsy)\notin\Rset\})=0$.  The map $\mci_0$ is an anchoring
  map and is $0$-homogeneous. The density $f_{\mci_0}$ of $\mci_0(\bsY)$ is then given by
  \begin{align*}
        f_{\mci_0}(t) = \candidate \pr(Y|\bsQ_{\mci_0(\bsQ)-t}|>1)] = \candidate \esp[|\bsQ_{\mci_0(\bsQ)-t}|^\alpha]   \; .
  \end{align*}
  Since for each $t\in\Rset$ the map $\bsy\mapsto\norm{\bsy_{\mci_0(\bsy)-t}}^\alpha$ is \shiftinvariant\ and
  $\alpha$-homogeneous,  \Cref{eq:Q-Theta} yields
  \begin{align}
    \label{eq:density-infargmax}
    f_{\mci_0}(t) = \esp\left[\frac{\norm{\bsTheta_{\mci_0(\bsTheta)-t}}^\alpha}{\int_{-\infty}^\infty \norm{\bsTheta_s}^\alpha \rmd s} \right] \; .
  \end{align}
  Condition \Cref{eq:condition-continuite} holds if $\bsY$ is almost surely continuous or if
  $\norm{\bsY}$ reaches its maximum by an upward jump.  Examples of tail processes with these
  properties will be given in \Cref{xmpl:subgaussian-maxstable} and
  \Cref{sec:shot-noise}. 
\end{example}

\begin{example}
  \label{xmpl:first-exceedance}
  Consider now the time of the  first exceedance over 1,  $\mci_1(\bsy) = \inf\{t\in\Rset:|y_t|>1\}$. If
  $\pr(\bsY\in\spaceD_0)=1$, then $\pr(\mci_1(\bsY)\in\Rset)=1$. For this map,
  \Cref{eq:condition-continuite} holds when $\norm{\bsY}$ exceeds 1 by an upward jump, but may not
  hold for continuous processes. See again \Cref{xmpl:subgaussian-maxstable}
  and \Cref{sec:shot-noise}.
\end{example}

\subsection{Identities}
\label{sec:identities}
In discrete time, the expectation of certain functionals of $\bsQ$ can be expressed in terms of the
forward spectral tail process only. This is important since the forward spectral tail process is
often easier to compute than the backward spectral tail process and the sequence $\bsQ$ which is
obtained by a change of measure or by conditioning.  For $\alpha$-homogeneous measurable functions
satisfying certain additional conditions, it may be proved that
\begin{align}
\label{eq:identite-Q-Theta-alpha-homogene}
  \candidate \esp[\homap(\bsQ)] = \esp[ \homap(\{\bsTheta_j,j\geq0\}) - \homap(\{\bsTheta_j,j\geq1\})] \; . 
\end{align}
See \cite[Section~3.3]{planinic:soulier:2018} for precise conditions. In particular, it always holds that
\begin{align}
    \label{eq:identiteforward-extremalindex}
  \candidate 
  & = \esp\left[\sup_{j\geq0}|\bsTheta|_j^\alpha-\sup_{j\geq1}|\bsTheta_{j}|^\alpha \right] \; , \\
  \candidate \esp\left[ \left(\sum_{j\in\Zset} |\bsQ_j| \right)^\alpha\right] 
  & = \esp\left[\left(\sum_{j=0}^\infty |\bsTheta_j| \right)^\alpha - \left(\sum_{j=1}^\infty |\bsTheta_j| \right)^\alpha \right] \; . 
    \label{eq:identite-l1alpha}
\end{align}
Both terms in \Cref{eq:identite-l1alpha} are finite when $\alpha\leq1$ and are simultaneously finite
or infinite if $\alpha>1$ (the difference inside the expectation in the right-hand side being set
equal to $\infty$ if the series is not summable in the latter case). Such identities are important
since the two different expressions appear as limits of the same statistics studied by different
tools.  In discrete time, these identities are obtained by a method of telescopic sums; in
continuous time this technique is not available and must be replaced by some form of
differentiation. Thus our only result is related to the identity \Cref{eq:identite-l1alpha}. Its
usefulness will be illustrated in \Cref{sec:illustration}.  It would be of interest to obtain a
formula extending \Cref{eq:identiteforward-extremalindex}, that is an expression of the candidate
extremal index in terms of the forward spectral tail process only.

\begin{lemma}
  \label{lem:identities}
  Let $\bsTheta$ be a $d$-dimensional spectral tail process such that $\pr(\bsTheta\in\spaceD_0)=1$
  and let $\bsQ$ be the sequence defined by \Cref{eq:loi-Q} or \Cref{eq:Q-Theta}. Then
    \begin{align}
      \label{eq:Q-Theta-forward}
      \candidate \esp \left[ \left( \int_{-\infty}^\infty |\bsQ_s| \rmd s \right)^\alpha \right] 
      & = \alpha \esp \left[ \left( \int_{0}^\infty |\bsTheta_s|\rmd s \right)^{\alpha-1} \right] \; .
    \end{align}
    If $\alpha\leq1$, both terms in \Cref{eq:Q-Theta-forward} are finite, and if $\alpha>1$ they are
    simultaneously finite or infinite. In the former case, for $d=1$, 
    \begin{align}
      \label{eq:Q-Theta-forward-noabs}
      \candidate \esp \left[ \left( \int_{-\infty}^\infty Q_s \rmd s \right)_+^\alpha
      \right] & = \alpha \esp \left[ \Theta_0 \left( \int_{0}^\infty \Theta_s\rmd s
                \right)_+^{\alpha-1} \right] \; .
    \end{align}
    In the case $0<\alpha\leq1$, we use the convention $x_+^{\alpha-1}=0$ if $x\leq0$.
\end{lemma}

\begin{proof}
  To prove \Cref{eq:Q-Theta-forward}, we can assume that $\bsTheta$ is a sequence of \nonnegative\
  random variables. Starting from \Cref{eq:Q-Theta}, and applying the monotone convergence theorem,
  we obtain
  \begin{align*}
    \candidate \esp \left[ \left( \int_{-\infty}^\infty Q_s \rmd s \right)^\alpha \right] 
    & =  \esp \left[ \frac{\left(\int_{-\infty}^\infty \Theta_t\right)_+^\alpha  }{\lpnorm[\alpha]{\bsTheta}[\alpha]} \right] 
      = \lim_{t\to\infty} \esp \left[ \frac{\left(\int_{-t}^\infty \Theta_t\right)_+^\alpha
      -\left(\int_{t}^\infty \Theta_t\right)_+^\alpha } {\lpnorm[\alpha]{\bsTheta}[\alpha]} \right]  \\
    & = \alpha \lim_{t\to\infty} \int_{-t}^t  \esp \left[ \frac{\left(\int_s^\infty \Theta_u\rmd u\right)^{\alpha-1}\Theta_s}
      {\lpnorm[\alpha]{\bsTheta}[\alpha]} \right]  \rmd s 
      = \alpha \int_{-\infty}^\infty  \esp \left[ \frac{\left(\int_s^\infty \Theta_u\rmd u\right)^{\alpha-1}\Theta_s  } 
      {\lpnorm[\alpha]{\bsTheta}[\alpha]} \right]  \rmd s \; .
  \end{align*}
  Applying the time change formula \Cref{eq:TCF-Theta} yields \Cref{eq:Q-Theta-forward} and that
  both terms are simultaneously finite or infinite. If $\alpha<1$, then \Cref{eq:candidate-Q} yields
  \begin{align*}
    \candidate \esp \left[ \left( \int_{-\infty}^\infty Q_s \rmd s \right)^\alpha \right]
    \leq \candidate \esp \left[ \int_{-\infty}^\infty Q_s^\alpha  \rmd s  \right] =1 \; .
  \end{align*}
  Thus both terms are finite.  If now $\bsTheta$ is not \nonnegative, assuming that both terms in
  \Cref{eq:identite-Q-Theta-alpha-homogene} are finite, we obtain \Cref{eq:Q-Theta-forward-noabs} by
  applying the dominated convergence theorem instead of monotone convergence.
\end{proof}

\section{Regular variation in $\spaceD(\Rset)$}
\label{sec:regvarinD}
Let ${\bsX}$ be a stationary process indexed by $\Rset$, with values in $\Rset^d$. We say that
$\bsX$ is finite dimensional regularly varying if there exists a sequence $a_n$ and for all $k\geq1$,
$s_1\leq \cdots \leq s_k\in\Rset$ there exists a non-zero measure $\nu_{s_1,\dots,s_k}$ on
$\Rset^{dk}\setminus\{\bszero\}$ such that
\begin{align}  
  \label{eq:fidi-rv}
  n\pr\left(\left(\frac{\bsX_{s_1}}{a_n},\dots,\frac{\bsX_{s_k}}{a_n}\right) \in \cdot \right)  \convvague 
 \nu_{s_1,\dots,s_k} \; , 
\end{align}
as $n\to\infty$, that is,
\begin{align*}
  \lim_{n\to\infty}
  n\pr\left(\left(\frac{\bsX_{s_1}}{a_n},\dots,\frac{\bsX_{s_k}}{a_n}\right) \in A \right) 
  = \nu_{s_1,\dots,s_k} (A) < \infty \; , 
\end{align*}
for all Borel sets $A$ separated from $\bszero$ in $\Rset^{dk}$ (\ie\ included in the complement of
a neighborhood (for the usual topology) of $\bszero$) which are continuity sets of
$\nu_{s_1,\dots,s_k}$.  The measure $\nu_{s_1,\dots,s_k}$ is called the exponent measure of
$(X_{s_1},\dots,X_{s_k})$ and there exists $\alpha$ such that $\nu_{s_1,\dots,s_k}$ is
$\alpha$-homogeneous, \ie\ $\nu_{s_1,\dots,s_k}(tA)=t^{-\alpha} \nu_{s_1,\dots,s_k}(A)$ for all
$t>0$, $k\geq1$ and $s_1,\dots,s_k$.

Although these measures are not finite and Kolmogorov extension theorem cannot be applied,
\cite{owada:samorodnitsky:2012} proved that there exists an $\alpha$-homogeneous, \shiftinvariant\
measure $\tailmeasure$ on $\Rset^{\Rset^d}$ endowed with the product topology, called the tail
measure, such that for all $k\geq1$, $s_1\leq \cdots \leq s_k\in\Rset$, $\nu_{s_1,\dots,s_k}$ is the
projection of $\tailmeasure$. The result of \cite{owada:samorodnitsky:2012} says nothing about the
support of $\tailmeasure$.
 
Let $\norm\cdot$ denote an arbitrary norm on $\Rset^d$.  We can and will henceforth assume that
the norming sequence $\{a_n\}$ is chosen such that
\begin{align*}
  \lim_{n\to\infty} n\pr(\norm{\bsX_0}>a_n) = 1 \; .
\end{align*}
With this choice, $\tailmeasure(\{\bsy\in\Rset^{\Rset^d}:\norm{\bsy_0}>1\})=1$.

According to \cite[Theorem~2.1]{basrak:segers:2009}, finite dimensional regular variation of the
process $\bsX$ is equivalent to the following conditions.
\begin{enumerate}[(i)]
\item $\norm{\bsX_0}$ is regularly varying with tail index $\alpha>0$;
\item there exists a process ${\bsY}$ such that for all $k\in\Nset^*$ and $s_1<\dots<s_k \in \Rset$, 
  \begin{align*}
    \lim_{x\to\infty} \law\left(\frac{\bsX_{s_1}}{x},\dots,\frac{\bsX_{s_k}}{x} \mid
      \norm{\bsX_0}>x\right)  = \law(\bsY_{s_1},\dots,\bsY_{s_k})  \; . 
  \end{align*}
\end{enumerate} 
The process $\bsY$ is called the tail process.  As shown in \cite{owada:samorodnitsky:2012}, the
distribution of the tail process (seen as a random element in $\Rset^{\Rset^d}$) is the tail measure
restricted to the set $\{\bsy\in\Rset^{\Rset^d}:\norm{\bsy_0}>1\}$, which is a probability measure
with the choice of norming constant $a_n$ defined above.  It follows from this definition that
$\norm{\bsY_0}$ is a Pareto random variable with tail index $\alpha$ and
$\bsTheta=\norm{\bsY_0}^{-1} \bsY$ is independent of $\bsY_0$.  These properties were proved in
\cite{basrak:segers:2009} in the case of processes indexed by $\Zset$, but so far as finite
dimensional distributions only are considered, they remain valid for processes indexed by $\Rset$.
However, this definition says nothing about the path properties of the process $\bsY$.

In order to obtain more information on the support of $\tailmeasure$ or the path properties of
$\bsY$, and make the link with the corresponding objects introduced in \Cref{sec:representation},
the mode of convergence must be strengthened.

Recall from \Cref{sec:representation} the definition of the boundedness $\mcb_0$, the class of
subsets $A$ separated from $\bszero$ in $\spaceD$ for the $J_1$ topology, \ie\ sets for which there
exist $a\leq b$ and $\epsilon>0$ such that $\inf_{\bsy\in A} \bsy_{a,b}^*>\epsilon$. Recall also that a
Borel measure $\mu$ on $\spaceD$ is $\mcb_0$-boundedly finite if $\mu(A)<\infty$ for all measurable
sets $A\in\mcb_0$.  A sequence $\{\mu_n,n\geq1\}$ of $\mcb_0$-boundedly finite Borel measures on
$\spaceD$ is said to converge $\mcb_0$-vaguely to a $\mcb_0$-boundedly finite Borel measure $\mu$,
denoted $\mu_n\convvague\mu$, if $\lim_{n\to\infty} \mu_n(A) = \mu(A)$ for all $\mu$-continuity
measurable sets $A\in\mcb_0$.

\cite{hult:lindskog:2005} introduced the notion of regular variation of stochastic processes indexed
by $[0,1]$, \ie\ random elements in $\spaceD([0,1],\Rset^d)$. Since $\spaceD(\Rset,\Rset^d)$ endowed
with the $J_1$ topology is a Polish space, the following extension is natural.
\begin{definition}
  \label{hypo:regvar-in-D}
  A $\spaceD$-valued stationary stochastic process $\bsX$ is said to be regularly varying in
  $\spaceD$ if there exists a non-zero $\mcb_0$-boundedly finite measure $\tailmeasure$ such that
  \begin{align}
    \label{eq:regvar-in-D}
    \frac{\pr(\bsX\in x\cdot)}{\pr(\norm{\bsX_0}>x)} \convvague \tailmeasure \; .
  \end{align}
\end{definition}
This definition entails that the limiting measure $\tailmeasure$ is necessarily
$\alpha$-homogeneous, \ie\ $\tailmeasure(t\cdot A) = t^{-\alpha} \tailmeasure(A)$ for all Borel sets
$A\subset\spaceD$. See \cite[Remark~3]{hult:lindskog:2005}.

Our first result is a necessary and sufficient condition for regular variation in $\spaceD$ which
adapts \cite[Theorem~10]{hult:lindskog:2005} to $\spaceD(\Rset)$. See \Cref{sec:J1} for the
definition of the moduli of continuity  $w'$ and $w''$.
\begin{theorem}
  \label{theo:rv-in-D-equivalence}
  Let $\bsX$ be a stationary, stochastically continuous $\spaceD$-valued process.  The following
  statements are equivalent.
  \begin{enumerate}[(i)]
  \item \label{item:Xrv-inD} $\bsX$ is regularly varying in $\spaceD$.
  \item \label{item:fidi+tightness} \Cref{eq:fidi-rv} holds for all $k\geq1$ and
    $(s_1,\dots,s_k)\in\Rset^k$ and for all $a<b$ and $\epsilon>0$, 
    \begin{align}
      \label{eq:tightness}
      \lim_{\delta\to0} \limsup_{x\to\infty} \frac{ \pr \left( w'(\bsX,a,b,\delta) > x \epsilon \right)}{\pr(\norm{\bsX_0}>x)} = 0 \; .
    \end{align}
  \end{enumerate}
  When these conditions hold, the tail measure of $\bsX$ is supported on $\spaceD$, its tail process
  has almost surely \cadlag\ paths and  conditionally on $\norm{\bsX_0}>x$, $x^{-1}\bsX$ converges
  weakly to~$\bsY$, as $x\to\infty$, on $\spaceD$ endowed with the $J_1$ topology.
\end{theorem}
In view of \Cref{theo:weak-convergence-in-DR}, \Cref{eq:tightness} can be equivalently replaced by 
\begin{align}
  \label{eq:tightness-wsecond}
  \lim_{\delta\to0} \limsup_{x\to\infty} \frac{ \pr \left( w''(\bsX,a,b,\delta) > x \epsilon \right)}{\pr(\norm{\bsX_0}>x)} = 0 \; .
\end{align}
Another way to state \cref{theo:rv-in-D-equivalence} is to say that a stochastically continuous,
stationary random element $\bsX$ in $\spaceD$ is regularly varying if and only if it is regularly
varying in $\spaceD(I)$ in the sense of \cite{hult:lindskog:2006}.

\begin{proof}[Proof of \Cref{item:Xrv-inD}$\implies$\Cref{item:fidi+tightness}]
  The proof is essentially the same as the proof of the implication
  (ii)$\implies$(i) of \cite[Theorem~10]{hult:lindskog:2005}, replacing
   $\mathcal{M}_0$-convergence  and $\spaceD([0,1])$ used therein by vague convergence and $\spaceD(\Rset)$.
\end{proof}

\begin{proof}[Proof of \Cref{item:fidi+tightness}$\implies$\Cref{item:Xrv-inD}]
  Define the measure $\mu_x$ on $\spaceD$ by 
  \begin{align*}
    \mu_{x} = \frac{\esp[ \delta_{x^{-1}\bsX} \ind{\bsX\ne\bszero}]} {\pr(\norm{\bsX_0}>x)} \; .
  \end{align*}
  For each $n,k\in\Nset^*$, let $\spaceD_{n,k}$ be the space of functions $f\in\spaceD$ such that
  $\sup_{-n\leq t < n} \norm{f(t)} > k^{-1}$.  We first prove that
  $\limsup_{x\to\infty}\mu_x(\spaceD_{n,k})<\infty$ for each fixed $n,k\in\Nset^*$. Indeed, for
  $\delta>0$,  
  \begin{align*}
    \mu_x(\spaceD_{n,k})  
    & = \frac{\pr(k\bsX_{-n,n}^*>x)}{\pr(\norm{\bsX_{0}}>x)} \\
    & \leq  \frac{\pr(2kw''(\bsX,-n,n,\delta)>x)}{\pr(\norm{\bsX_0}>x)} +
      \frac{\pr(k\bsX_{-n,n}^*>x,2kw''(\bsX,-n,n,\delta)\leq x)}{\pr(\norm{\bsX_0}> x)} \; .
  \end{align*}
  If $k\bsX_{-n,n}^*>x$ and $2kw''(\bsX,-n,n,\delta)\leq x$, then for every sequence
  $(t_0,\dots,t_k)$ such that $a=t_0<\cdots<t_k=b$ and $t_{i}-t_{i-1}\leq\delta$ for $i=1,\dots,k$,
  it necessarily holds that $2k\max_{0\leq i  \leq k} |X_{t_i}|>x$. Thus
  \begin{align*}
    \mu_x(\spaceD_{n,k}) & \leq  \frac{\pr(2kw''(\bsX,-n,n,\delta)>x)}{\pr(\norm{\bsX_0}>x)} 
                           + \frac{\pr(2k\max_{0 \leq i \leq  k} \norm{\bsX_{t_i}}>x)}{\pr(\norm{\bsX_0}>x)}
  \end{align*}
  Under the assumptions of the theorem, both terms tend to zero by letting $x\to\infty$, then
  $\delta\to0$. This proves our first claim, from which it ensues that the measure $\mu_x$ is
  $\mcb_0$-boundedly finite on $\spaceD$ and we can define the finite measure $\mu_{n,k,x}$ on
  $\spaceD_{n,k}$ by
  \begin{align*}
    \mu_{n,k,x} = \frac{\esp[ \delta_{x^{-1}\bsX} \ind{\bsX_{-n,n}^*>k^{-1}x}]} {\pr(\norm{\bsX_0}>x)} \; .
  \end{align*}
  Since
  $\pr \left( w'(\bsX,a,b,\delta) > x \epsilon ; \norm{\bsX_0}>x\right) \leq \pr \left(
    w'(\bsX,a,b,\delta) > x \epsilon \right)$, the assumption \Cref{eq:tightness} implies
  \begin{align}
      \label{eq:tightness-conditional}
      \lim_{\delta\to0} \limsup_{x\to\infty} \pr \left( w'(\bsX,a,b,\delta) > x \epsilon \mid \norm{\bsX_0}>x\right) = 0 \; .
  \end{align}
  By \Cref{theo:weak-convergence-in-DR}, the finite dimensional weak
  convergence of $x^{-1}\bsX$ conditionally on $\norm{\bsX_0}>x$ (at all but countably many points
  of $\Rset$) and \Cref{eq:tightness-conditional} yield the weak convergence in $\spaceD$ of
  $x^{-1}\bsX$ to $\bsY$, conditionally on $\norm{\bsX_0}>x$.

  Let $\map$ be a bounded continuous map on $\spaceD$ (\wrt\ the $J_1$ topology) with support separated
  from the null map in $\spaceD$, \ie\ there exists $a<b$ such that $\map(\bsy)=0$ if
  $\bsy^*_{a,b}\leq 2\epsilon$ (see \cref{eq:bafz-in-D}), or equivalently,
  $\map = \map\ind{\exc_{a,b}(2\epsilon^{-1}\cdot)>0}$.  Then, for $\eta>0$,
  \begin{align*} 
    \esp [\map(x^{-1} \bsX) 
    & \ind{\exc_{a,b}(\epsilon^{-1} x^{-1} \bsX)>x\eta}]] \\
    & = \int_a^b  \esp \left[ \frac{\map(x^{-1} \bsX)\ind{\norm{\bsX_t}>\epsilon x} 
      \ind{\exc_{a,b}(\epsilon^{-1}x^{-1}\bsX)>x\eta}} {\exc_{a,b}(\epsilon^{-1}x^{-1}\bsX)} \right] \rmd t \\
    & = \int_a^b  \esp \left[ \frac{\map(x^{-1} \shift^t\bsX)\ind{\norm{\bsX_0}>\epsilon x} 
      \ind{\exc_{a,b}(\shift^t\epsilon^{-1}x^{-1}\bsX)>x\eta}} {\exc_{a,b}(\epsilon^{-1}x^{-1}\shift^t\bsX)} \right] \rmd t \; .
  \end{align*}
  The last line was obtained by stationarity of $\bsX$. By the regular variation of $\bsX_0$, the
  weak convergence in $\spaceD$ stated above and  dominated convergence, we now obtain
  \begin{align*}
    \lim_{x\to\infty} \frac{ \esp [\map(x^{-1} \bsX) \ind{\exc_{a,b}(\epsilon^{-1} x^{-1} \bsX)>x\eta}]]}{\pr(\norm{\bsX_0}>x)}
    & = \epsilon^{-\alpha} \int_a^b \esp \left[ \frac{\map(\epsilon \shift^t\bsY)
      \ind{\exc_{a,b}(\shift^t\bsY)>\eta}} {\exc_{a,b}(\shift^t\bsY)} \right] \rmd t \; .
  \end{align*}  As $\eta\to0$, 
  Since $\map = \map\ind{\exc(\epsilon^{-1}\cdot)>0}$, we have by monotone convergence
  \begin{align}
    \label{eq:lim-mu}
    \lim_{\eta\to0}    \epsilon^{-\alpha} \int_a^b \esp \left[ \frac{\map(\epsilon \shift^t\bsY)
    \ind{\exc_{a,b}(\shift^t\bsY)>\eta}} {\exc_{a,b}(\shift^t\bsY)} \right] \rmd t 
    = \epsilon^{-\alpha} \int_a^b \esp \left[ \frac{\map(\epsilon \shift^t\bsY)} {\exc_{a,b}(\shift^t\bsY)} \right] \rmd t \; .
  \end{align}
  The latter quantity is finite by the first part of the proof and we denote it by $\mu(\map)$.

  If $f\in\spaceD$ is such that $w'(f,a,b,\eta) \leq \epsilon/2$ and
  $\exc_{a,b}((2\epsilon)^{-1}f) > 0$, then there exists $t\in[a,b)$ such that $f(t)>2\epsilon$ and
  an interval $[u,v)$ such that $t\in[u,v)$, $v-u \geq \eta$ and
  $\sup_{[u\leq s,s'<v)}|f(s)-f(s')|\leq \epsilon$. Consequently, $f(s)>\epsilon$ for all
  $s\in[u,v)$ and $\exc_{a,b}(\epsilon^{-1}f)\geq \eta$. This yields
  \begin{align*}
    \esp[\map(x^{-1} \bsX)
    & \ind{\exc_{a,b}(\epsilon^{-1}x^{-1}\bsX) \leq \eta}] \\
    & \leq \constant\ \pr(\bsX_{a,b}^* >2\epsilon x;\exc(\epsilon^{-1}x^{-1}\bsX) \leq \eta)  \\
    & \leq \constant\ \pr(w'(\bsX,a,b,\eta)>\epsilon x) \; .
  \end{align*}
  Here and throughout, $\constant$ denotes a numerical constant which depends on none of the
  variable parameters around it.  Thus \Cref{eq:tightness} implies that
  \begin{align}
    \label{eq:triangular-nu}
    \lim_{\eta\to0} \limsup_{x\to\infty} 
    \esp[\map(x^{-1} \bsX) \ind{\exc_{a,b}(\epsilon^{-1}x^{-1}\bsX) \leq \eta}] = 0 \; .
  \end{align}
  By the triangular argument \cite[Theorem~3.2]{billingsley:1999},
  \Cref{eq:lim-mu,eq:triangular-nu}, we obtain
  \begin{align*}
    \lim_{x\to\infty} \frac{\esp[\map(x^{-1} \bsX)}{\pr(\norm{\bsX_0}>x)} = \mu(\map) \; .
  \end{align*}
  This also proves that $\mu(\map)$ does not depend on the particular choice of $a,b,\epsilon$. Thus,
  taking $\epsilon=k^{-1}$, $a=-n$ and $b=n$, we  define a finite measure $\mu_{n,k}$ on
  $\spaceD_{n,k}$ by
  \begin{align*}
    \mu_{n,k}(\map) = \mu(\map) 
    = k^{\alpha} \int_{-n}^n \esp \left[ \frac{\map(k^{-1} \shift^t\bsY)} {\exc_{-n,n}(\shift^t\bsY)} \right] \rmd t \; .
  \end{align*}
  We have proved that $\mu_{n,k,x}\convweak\mu_{n,k}$ on $\spaceD_{n,k}$ for all $n,k\geq1$. This implies that
  $\mu_x\convvague\mu$ on $\spaceD\setminus\{\bszero\}$ by \cite[Lemma~4.6]{kallenberg:2017}.
\end{proof}

\subsection{The anticlustering condition}
\label{sec:anticlustering}
Let $\spaceD_0(\Rset,\Rset^d)$, hereafter abbreviated to $\spaceD_0$, be the space of
$\Rset^d$-valued \cadlag\ functions which tend to zero at $\pm\infty$. Let $\mch$ be the set of
strictly increasing continuous maps on $\Rset$ and $d_{\infty}$ be the distance defined on
$\spaceD_0$ by
\begin{align*}
  d_{\infty}(f,g) = \inf_{u \in \mch} \supnorm[\infty]{f\circ u -g} \vee \supnorm[\infty]{u-\id}\; .
\end{align*}
We denote the topology induced by the metric $d_\infty$ by $J_1^0$. Obviously,
$d_{J_1}(f,g) \leq d_{\infty}(f,g) \leq \supnorm[\infty]{f-g}$, thus a sequence converging \wrt\
$d_\infty$ converges \wrt\ $d_{J_1}$ and an open set for $d_{\infty}$ is also open for $d_{J_1}$.
{The topology $J_1^0$ induced by $d_{\infty}$ on $\spaceD_0$ is Polish and the associated Borel
  $\sigma$-field is the product $\sigma$-field.}

We now introduce an assumption which ensures that the tail process tends to zero at $\infty$. This
assumption is related to condition (2.8) of \cite{davis:hsing:1995}, see also
\cite[Condition~4.1]{basrak:segers:2009}. To avoid repetitions, we define a scaling function as a
\nondecreasing\ unbounded function defined on $[0,\infty)$. 
\begin{hypothesis}
  \label{hypo:anticlustering}
  There exist scaling functions  $a$ and $r:\Rset_+\to\Rset_+$ such that for all $x>0$, 
  \begin{align}
    \label{eq:anticlustering}
    \lim_{t\to\infty} \limsup_{T\to\infty} \pr\left(\sup_{t\leq |s| \leq r_T} \norm{\bsX_s} > a_T x \mid \norm{\bsX_0}>a_T\right) = 0 \; .
  \end{align}
\end{hypothesis}

For $m>0$, we say that a stochastic process $\bsX$ is $m$-dependent if for all $t\in\Rset$,
$\{\bsX_s,s \geq t+ m\}$ is independent of  $\{\bsX_s,s\leq t\}$.
\begin{lemma}
  \label{lem:AC-mdep}
  Let $\bsX$ be an $m$-dependent $\spaceD$-valued stationary stochastic process, regularly varying
  in $\spaceD$. Then \Cref{eq:anticlustering} holds for all scaling functions $a$ and $r$ such that
  $\lim_{T\to\infty} r_T \pr(\norm{\bsX_0}>a_T)=0$.
\end{lemma}
\begin{proof}
  For $t>m$, we have by $m$-dependence
  \begin{align*}
    \pr\left(\sup_{t\leq |s| \leq r_T} \norm{\bsX_s} > a_T x \mid \norm{\bsX_0}>a_T\right) 
    & = \pr\left(\sup_{t\leq |s| \leq r_T} \norm{\bsX_s} > a_T x \right) \\
    & \leq r_T \pr(\norm{\bsX_0}>a_T) \frac{\pr\left(\sup_{0\leq |s| \leq 1} \norm{\bsX_s} > a_T x \right)}{\pr(\norm{\bsX_0}>a_T)} \; .
  \end{align*}
  By regular variation in $\spaceD$, the fraction in the right-hand side converges to
  $x^{-\alpha} \tailmeasure(\{\bsy\in\spaceD:\bsy_{0,1}^*>1\})$ as $T$ tends to $\infty$. Thus
  \Cref{eq:anticlustering} holds for every sequence $r_T$ such that 
  $\lim_{T\to\infty} r_T \pr(\norm{\bsX_0}>a_T)=0$ as claimed.
\end{proof}

\begin{lemma}
  \label{lem:weakconv-anticlustering}
  Let $\bsX$ be a $\spaceD$-valued stationary process, regularly varying in the sense of
  \Cref{hypo:regvar-in-D}.  If \Cref{hypo:anticlustering} holds, then $\pr(\bsY\in\spaceD_0)=1$ and
  conditionally on $\norm{\bsX_0}>a_Tx$, $(xa_T)^{-1}\bsX\1{[-r_T,r_T)} \convweak \bsY$ in
  $\spaceD_0$ endowed with the $J_1^0$ topology.
\end{lemma}
\begin{proof}
  We first prove that $\pr(\bsY\in\spaceD_0)=1$. By
  \Cref{theo:rv-in-D-equivalence,hypo:anticlustering}, we have, for $\epsilon>0$ and large enough
  $t\leq t'$, 
  \begin{align*}
    \pr\left(\sup_{t\leq |s|\leq t'} \norm{\bsY_t}>\epsilon\right)
    & = \lim_{T\to\infty}
      \pr\left(\sup_{t\leq |s|\leq t'} \norm{\bsX_t}>\epsilon a_T\mid \norm{\bsX_0}>a_T\right) \\
    & \leq \limsup_{T\to\infty}
      \pr\left(\sup_{t\leq |s|\leq r_T} \norm{\bsX_t}>\epsilon a_T\mid \norm{\bsX_0}>a_T\right)  \leq \epsilon \; .
  \end{align*}
  This proves the first claim.

  To prove the stated weak convergence, we apply \cite[Theorem~11.3.3]{dudley:2002}.  Let $\map$ be a
  bounded Lipschitz function \wrt\ the metric $d_{\infty}$ on $\spaceD_0$. Since the metric
  $d_{\infty}$ is dominated by the uniform norm, we have
  \begin{multline*}
    \esp[ |\map(a_T^{-1}\bsX\1{[-t,t)}) - \map(a_T^{-1}\bsX\1{[-r_T,r_T)})|\mid \norm{\bsX_0}>xa_T] \\
     \leq \constant\ \epsilon + \pr\left( \sup_{t \leq |s| \leq r_T} \norm{\bsX_s}>\epsilon a_T\mid \norm{\bsX_0}>xa_T\right) \; .
  \end{multline*}
  Thus, \Cref{hypo:anticlustering} yields
  \begin{align}
    \label{eq:triangularargument1}
    \lim_{t\to\infty} \limsup_{T\to\infty} \esp[ |\map(a_T^{-1}\bsX\1{[-t,t)}) - \map(a_T^{-1}\bsX\1{[-r_T,r_T)})|\mid \norm{\bsX_0}>xa_T] = 0 \; .
  \end{align}
  By the convergence (\ref{eq:regvar-in-D}), we have, for each $t>0$, 
  \begin{align}
    \label{eq:triangularargument2}
    \lim_{T\to\infty} \esp[ \map(a_T^{-1}\bsX\1{[-t,t)}) |\mid \norm{\bsX_0}>xa_T] = \esp[\map(\bsY\1{[-t,t)})]  \; .
  \end{align}
  Since $\bsY$ tends to zero at $\infty$ and $\map$ is Lipschitz, we have
  \begin{align}
    \label{eq:triangularargument3}
    \lim_{t\to\infty} \esp[\map(\bsY\1{[-t,t)})] = \esp[\map(\bsY)] \; . 
  \end{align}
  The convergences (\ref{eq:triangularargument1}), (\ref{eq:triangularargument2}) and
  (\ref{eq:triangularargument3}) conclude the proof.
\end{proof}

{\Cref{lem:weakconv-anticlustering} shows that \cref{hypo:anticlustering} implies that the tail
measure is supported on $\spaceD_0$. Thus the tail measure admits a mixed moving average
representation and the candidate extremal index~$\candidate$ is positive. This rules out extremal
long memory in the sense of $\candidate=0$. This will be our working assumption from here on.}

\subsection{The cluster measure}
\label{sec:cluster-measure}
Recall that we have defined $\candidate =\esp[\exc^{-1}(\bsY)]<\infty$ by
\Cref{lem:expinvexcfini}. Throughout this section, we will assume that $\pr(\bsY\in\spaceD_0)=1$,
which implies $\candidate>0$ by \Cref{theo:equivalences-dissipative}.  Let $\bsQ$ be be a random
element on $\spaceD_0$ whose distribution is given by \Cref{eq:loi-Q} and define the boundedly
finite measure $\tailmeasurestar$ on $\spaceD_0$ by
\begin{align}
  \label{eq:tailmeasurestar-Q}
  \tailmeasurestar = \candidate \int_0^\infty  \esp [ \delta_{r \bsQ} ]  \alpha r^{-\alpha-1} \rmd r \; .
\end{align}
This and \Cref{eq:dissipativerepresentation} yield
\begin{align}
  \label{eq:tailmeasure-nu-nustar}
  \tailmeasure   = \int_{-\infty}^\infty  \tailmeasurestar\circ\shift^t \, \rmd t  \; .
\end{align}

\paragraph{The space $\spaceDtilde_0$} 
In order to state our result on the convergence of the point process of exceedances, we need to
introduce the space $\spaceDtilde_0$ which is the quotient of the space $\spaceD_0$ by the relation
of shift-equivalence. We say that two functions $g,f$ defined on $\Rset$ are shift-equivalent if
there exists $t\in\Rset$ such that $f=\shift^tg$, \ie\ $f(x) = g(x-t)$ for all $x\in\Rset$. This is
an equivalence relation and the space $\spaceDtilde_0$ is the set of equivalence classes for this
relation. We endow it with the quotient topology which is metrizable with the metric
$\dtilde_\infty$ defined by
\begin{align*}
  \dtilde_\infty(\tilde{f},\tilde{g}) = \inf_{f\in\tilde{f},g\in\tilde{g}} d_\infty (f,g) \; .
\end{align*}
The set $\spaceDtilde_0$ endowed with this topology inherits the Polishness property. A map $\tilde{\shinvmap}$
on $\spaceDtilde_0$ is uniquely associated to a shift-invariant map $\shinvmap$ on $\spaceD_0$ by the
relation $\shinvmap(f) = \tilde{\shinvmap}(\tilde{f})$ for all $\tilde{f}\in\spaceDtilde_0$ and $f\in\tilde{f}$.

We now define vague convergence of boundedly finite measures on $\spaceDtilde_0$. Let $\tilde\mcb_0$
be the class of subsets $\tilde{A}$ of $\spaceDtilde_0$ such that $\tilde\bsy\in\tilde{A}$ implies
that there exist $\bsy\in\tilde\bsy$ with $\bsy^*>\epsilon$. We say that $\nu$ is
$\tilde\mcb_0$-boundedly finite on $\spaceDtilde_0$ if $\nu(A)<\infty$ for all measurable sets
$A\in\tilde\mcb_0$ and $\nu(\{\tilde\bszero\})=0$. Vague convergence on $\spaceDtilde_0$ is defined
\wrt\ $\tilde\mcb_0$: we say that a sequence of $\tilde\mcb_0$-boundedly finite measures $\nu_n$
converges vaguely to $\nu$ in $(\spaceDtilde_0,\tilde\mcb_0)$, denoted $\nu_n\convvague\nu$ if
$\lim_{n\to\infty}\nu_n(A)=\nu(A)$ for all $\nu$-continuity sets  $A\in\mcb_0$.

Any \shiftinvariant\ map $\shinvmap$ on $\spaceD_0$ can be seen as a map on $\spaceDtilde_0$ and
conversely any map on $\spaceDtilde_0$ can be seen a \shiftinvariant\ map $\shinvmap$ on
$\spaceD_0$, and it is Lipschitz \wrt\ $d_\infty$ \ifft\ it is Lipschitz \wrt\
$\tilde{d}_\infty$. Thus a necessary and sufficient condition for vague convergence of a sequence of
boundedly finite measures $\{\nu_n,n\geq1\}$ on $\spaceDtilde_0$ to a measure $\nu$ is that
$\lim_{n\to\infty} \nu_n(\shinvmap) = \nu(\shinvmap)$, for all \shiftinvariant\ maps $\shinvmap$ in
$\spaceD_0$, with support in $\mcb_0$ and Lipschitz \wrt\ the metric $d_\infty$.

For functions $a_T$ and $r_T$, we define the measure $\tailmeasurestar_{T,r_T}$ on $\spaceD_0$ by
\begin{align}
  \label{eq:def-clustermeasure-rT}
  \tailmeasurestar_{T,r_T} = \frac{\esp \left[ \delta_{a_T^{-1}\bsX_{0,r_T}} \right]}{r_T\pr(\norm{\bsX_0}>a_T)} \; . 
\end{align}

\begin{lemma}
  \label{lem:convtailmeasurestar}
  Let $\bsX$ be a stationary $\spaceD$-valued stochastic process, regularly varying in
  $\spaceD$. Let $a_T$, $r_T$ be scaling functions such that \Cref{eq:anticlustering} holds. Then
  $\tailmeasurestar_{T,r_T} \convvague \tailmeasurestar$  in $(\spaceDtilde_0,\tilde\mcb_0)$.
\end{lemma}

\begin{proof}
  We must prove that for all $\epsilon>0$ and \shiftinvariant\ bounded maps $\shinvmap$ defined on
  $\spaceD_0$, continuous \wrt\ the $J_1^0$ topology an such that $\shinvmap(\bsy)=0$ if
  $\bsy^*\leq 2\epsilon$, it holds that
   \begin{align}
    \label{eq:conv-tailmeasurestar}
    \lim_{T\to\infty}    \tailmeasurestar_{T,r_T}(\shinvmap) = \tailmeasurestar(\shinvmap) \; . 
  \end{align}
  Write 
  \begin{align*}
    \tailmeasurestar_{T}(\shinvmap) 
     = \frac{\esp[\shinvmap(a_T^{-1} \bsX_{0,r_T})\ind{\exc(\epsilon^{-1}a_T\bsX_{0,r_T})>\eta}]}{r_T\pr(\norm{\bsX_0}>a_T)} 
       + \frac{\esp[\shinvmap(a_T^{-1} \bsX_{0,r_T})\ind{\exc(\epsilon^{-1}a_T^{-1}\bsX_{0,r_T}) \leq \eta}]}{r_T\pr(\norm{\bsX_0}>a_T)} \; .
  \end{align*}
  As argued in the proof of \Cref{theo:rv-in-D-equivalence}, if $f\in\spaceD$ is such that
  $w'(f,a,b,\eta) \leq \epsilon/2$ and $\exc_{a,b}((2\epsilon)^{-1}f) > 0$, then
  $\exc_{a,b}(\epsilon^{-1}f)\geq \eta$. Thus
  \begin{align*}
    \esp[\shinvmap(a_T^{-1} \bsX_{0,r_T}) 
    & \ind{\exc(\epsilon^{-1}a_T^{-1}\bsX_{0,r_T}) \leq \eta}] \\
    & \leq \constant\ \pr((\bsX_{0,r_T})^* >2\epsilon a_T;\exc(\epsilon^{-1}a_T^{-1}\bsX_{0,r_T}) \leq \eta)  \\
    & \leq \constant\ \pr(w'(\bsX,0,r_T,\eta)>\epsilon a_T) \\
    & \leq \constant\ r_T  \pr(w'(\bsX,0,1,\eta)>\epsilon a_T) \; .
  \end{align*}
  This yields, by \cite[Theorem~16.13]{billingsley:1999}, 
  \begin{multline}
    \label{eq:triangular-argument-nustar}
    \lim_{\eta\to0} \limsup_{T\to\infty} \frac{\esp[\shinvmap(a_T^{-1} \bsX_{0,r_T})\ind{\exc(\epsilon^{-1}a_T^{-1}\bsX_{0,r_T}) \leq \eta}]}
    {r_T\pr(\norm{\bsX_0}>a_T)} \\
    \leq     \lim_{\eta\to0} \limsup_{T\to\infty} \constant\ \frac{ \pr(w'(\bsX,0,1,\eta)>\epsilon a_T)} {\pr(\norm{\bsX_0}>a_T)} = 0  \; .  
  \end{multline}
  The other term is dealt with by dominated convergence arguments and the convergence in $\spaceD$
  of $x^{-}\bsX$ conditionally on $\norm{\bsX_0}>x$. We first write
  \begin{align*}
    & \frac{\esp[\shinvmap(a_T^{-1} \bsX_{0,r_T})\ind{\exc(\epsilon^{-1}a_T\bsX_{0,r_T})>\eta}]}{r_T\pr(\norm{\bsX_0}>a_T)} \\
    & = \frac{1}{r_T\pr(\norm{\bsX_0}>a_T)}
      \int_0^{r_T} \esp \left[ \frac{\shinvmap(a_T^{-1} \bsX_{0,r_T}) \ind{\norm{\bsX_s}>\epsilon a_T} \ind{\exc(\epsilon^{-1}a_T\bsX_{0,r_T})>\eta}}
      {\exc(\epsilon^{-1} a_T^{-1} \bsX_{0,r_T})}\right] \rmd t  \\
    & = \frac{\pr(\norm{\bsX_0}>a_T)}{\pr(\norm{\bsX_0}>a_T)} \int_0^1 g_T(s) \rmd s  \;, 
  \end{align*}
  with $g_T$ defined by 
  \begin{align*}
    g_T(s) & =     \esp \left[ \frac{\shinvmap(a_T^{-1} \bsX\ind{[-r_Ts,(1-s)r_T)}) \ind{\exc(\epsilon^{-1}a_T\bsX\1{[-r_Ts,r_Ts]})>\eta}}  
             {\exc(\epsilon^{-1} a_T^{-1} \bsX\ind{[-r_Ts,(1-s)r_T)})}  \mid \norm{\bsX_0}>\epsilon a_T\right] \; .
  \end{align*}  
  Since $\shinvmap$ is bounded and $g_T=O(\eta^{-1})$, we have by \Cref{lem:weakconv-anticlustering} and
  dominated convergence, for each $s\in[0,1]$,
  \begin{align*}
    \lim_{T\to\infty} g_T(s) 
    & \to \epsilon^{-\alpha}  \esp\left[\frac{\shinvmap(\epsilon\bsY)\ind{\exc(\bsY)>\eta}}{\exc(\bsY)} \right]
      = \tailmeasurestar(\shinvmap\ind{\exc>\eta}) \; .
  \end{align*}
  By dominated convergence again, we obtain
  \begin{align*}
    \lim_{T\to\infty} \frac{\esp[\shinvmap(a_T^{-1} \bsX_{0,r_T})\ind{\exc(\epsilon^{-1}a_T\bsX_{0,r_T})>\eta}]}
    {r_T\pr(\norm{\bsX_0}>a_T)} =  \tailmeasurestar(\shinvmap\ind{\exc>\eta}) \; .
  \end{align*}
  Furthermore, $\lim_{\eta\to0} \tailmeasurestar(\shinvmap\ind{\exc>\eta}) = \tailmeasurestar(\shinvmap)$. This
  convergence and \Cref{eq:triangular-argument-nustar} yield \Cref{eq:conv-tailmeasurestar}.
\end{proof}

The previous result states that if \Cref{eq:anticlustering} holds with some functions $a_T$ and
$r_T$, then $\tailmeasurestar_{T,r_T} \convvague \tailmeasurestar$. We also know that
\Cref{eq:anticlustering} implies that the tail process converges almost surely to zero. There is no
converse of this result. However, if we know that the tail process converges to zero, then for any
given function $a_T$, we can prove the existence of a function $r_T$ such that
$\tailmeasurestar_{T,r_T} \convvague \tailmeasurestar$.

\begin{lemma}
  \label{lem:tailtozero-implies-convergencetoclustermeasure}
  Assume that $\bsX$ is regularly varying in $\spaceD$ with tail process $\bsY$ such that
  $\pr(\bsY\in\spaceD_0)=1$. Then for each scaling function $a_T$, there
  exists a scaling function  $r_T$ such that
  \begin{align}
    \label{eq:rtprattozero}
    \lim_{T\to\infty} r_T\pr(\norm{\bsX_0}>a_T)=0 \ , 
  \end{align}
  and $\tailmeasurestar_{T,r_T}\convvague\tailmeasurestar$ in $(\spaceDtilde_0,\tilde\mcb_0)$.
\end{lemma}

\begin{proof} 
  Fix $m\geq1$.  For a map  $\map$ on $\spaceD$, let $\map_m$ be the map defined on $\spaceD$ by
  $\map_m(\bsy) = \map(\bsy_{0,m})$. Define a measure $\tildetailmeasurestar_m$ on $\spaceD_0$ by
  \begin{align}
    \label{eq:def-tildemeasurestar}
    \tildetailmeasurestar_m(\map) = \frac1m\tailmeasure(\map_m) \; .
  \end{align}
  Let $\epsilon>0$ and $\map$ be a bounded Lipschitz continuous map on $\spaceD_0$ such that
  $h(\bsy)=0$ if $\bsy^*\leq 2\epsilon$.  Regular variation in $\spaceD$ implies that
  \begin{align*}
    \lim_{T\to\infty} \frac{\esp[\map(a_T^{-1}\bsX_{0,m})]} {m\pr(\norm{\bsX_0}>a_T)} 
    & = \frac1m \tailmeasure(\map_m) = \tildetailmeasurestar_m(\map) \; .
  \end{align*}
  The latter quantity is finite since the support of $\map_m(\bsy) \leq \constant\ind{\bsy_{0,m}^*>1}$.
  Applying \Cref{eq:dissipativerepresentation}, we obtain
  \begin{align*}
    \tildetailmeasurestar_m(\map) 
    & = \frac1m \tailmeasure(\map_m)  
      = \frac1m \int_{-\infty}^\infty \int_0^\infty \esp[\map(r\bsQ_{-t,m-t})] \alpha r^{-\alpha-1}  \rmd r \, \rmd t \\
    & = \int_0^1 \int_0^\infty \esp[\map(r\bsQ_{-mt,m(1-t)})] \alpha r^{-\alpha-1}  \rmd r \, \rmd t + R_1 + R_2 \; , 
  \end{align*}
  with $R_1$ and $R_2$ the integrals over $(-\infty,0)$ and $(T,\infty)$, respectively.  Since $\map$
  is Lipschitz continuous and bounded and $\bsQ\in\spaceD_0$, we have by the dominated convergence
  theorem, 
  \begin{align*}
    \lim_{m\to\infty} \int_0^1  \int_0^\infty \esp[\map(r\bsQ_{-mt,m(1-t)})] \alpha r^{-\alpha-1}  \rmd r \rmd t 
    =   \int_0^\infty \esp[\map(r\bsQ)] \alpha r^{-\alpha-1}  \rmd r = \tailmeasurestar(\map) \; .
  \end{align*}
  We next prove that $R_1$ and $R_2$ tend to zero. Note that 
  \begin{align*}
    \int_0^\infty |\map(r\bsy)| \alpha r^{-\alpha-1} \rmd r \leq \constant\ \int_0^\infty \ind{u\bsy_{0,m}^*>\epsilon} \alpha u^{-\alpha-1} \rmd u 
    = \constant\ \epsilon^{-\alpha} (\bsy_{0,m}^*)^\alpha \; .
  \end{align*}
  This yields, by subadditivity of the maximum,
  \begin{align*}
    R_1 & \leq  \frac1m \int_{-\infty}^0 \esp[|\map_\alpha(\bsQ_{-s,m-s})|] \rmd s  
          \leq \frac{\constant}m \int_0^{\infty} \esp[(\bsQ_{s,m+s}^*)^\alpha] \rmd s  \\
        & \leq \frac{\constant}m \sum_{i=1}^m \int_0^{\infty} \esp[(\bsQ_{i-1+s,i+s}^*)^\alpha] \rmd s  
         = \frac{\constant}m \sum_{i=1}^m \int_{i-1}^\infty \esp[(\bsQ_{s,s+1}^*)^\alpha] \rmd s  \; .
  \end{align*}
  Note that
  $\int_{-\infty}^{\infty} \esp[(\bsQ_{s,s+1}^*)^\alpha] \rmd s =
  \tailmeasure(\{\bsy\in\spaceD:\bsy_{0,1}^*>1\})<\infty$, thus
  \begin{align*}
    \lim_{i\to\infty} \int_i^\infty \esp[(\bsQ_{s,s+1}^*)^\alpha] \rmd s =0 \; .
  \end{align*}
  By Cesaro's theorem, this yields $\lim_{m\to\infty} R_1=0$. The proof for $R_2$ is along the same
  lines.

  We have thus proved that $\tailmeasurestar_{T,m}\convvague\tildetailmeasurestar_m$ and
  $\tildetailmeasurestar_m\convvague\tailmeasurestar$. Since vague convergence is metrizable, this
  implies that there exists a sequence $r_T$ such that
  $\tailmeasurestar_{T,r_T}\convvague\tailmeasurestar$. (See for instance \cite[p.395, comment after
  Proposition~11.3.2]{dudley:2002}.)

  Assume that the function $r_T$ does not satisfy \Cref{eq:rtprattozero}. Then, along a subsequence,
  we would have $r_T\pr(\norm{\bsX_0}>a_T)\to c>0$ and thus 
  \begin{align*}
    \lim_{T\to\infty} \pr(\bsX_{0,r_T}^*>a_Tx) = c \tailmeasurestar(\{\bsy^*>w\}) = c \candidate x^{-\alpha} \; .
  \end{align*}
  This is a contradiction, since the left-hand side must be less than 1, and the right-hand side can
  be arbitrarily large. Thus the scaling function $r_T$ must satisfy \Cref{eq:rtprattozero}. 
\end{proof}

The previous  results allow to prove convergence of the measure $\tailmeasurestar_{T,r_T}$ when
the process $\bsX$ admits a suitable sequence of approximations.
\begin{lemma}
  \label{lem:mdep-approx-clustermeasure}
  Let $\bsX$ be a stationary process, regularly varying in $\spaceD$ with tail process $\bsY$ such
  that $\pr(\bsY\in\spaceD_0)=1$. Assume that there exists a sequence of
  stationary $m$-dependent processes $\bsX^{(m)}$, regularly varying in $\spaceD$, such that
  $(\bsX,\bsX^{(m)})$ is stationary and 
  \begin{align}
    \label{eq:condition-m-approx}
    \lim_{m\to\infty} \limsup_{x\to\infty}
    \frac {\pr\left(\sup_{0 \leq s \leq 1} \norm{\bsX_s-\bsX_s^{(m)}}>x\right)}{\pr(\norm{\bsX_0}>x)} = 0 \; .
  \end{align}  
  Let $\tailmeasurestar$ and $\tailmeasurestar_m$ be the cluster measures of $\bsX$ and
  $\bsX^{(m)}$, respectively.  Then $\tailmeasurestar_{m}\convvague\tailmeasurestar$ and
  $\tailmeasurestar_{T,r_T}\convvague\tailmeasurestar$ in $(\spaceDtilde_0,\tilde\mcb_0)$ for all
  scaling functions $a_T$, $r_T$ such that \Cref{eq:rtprattozero} holds. 
\end{lemma}
Because of stationarity, the interval $[0,1]$ in\cref{eq:condition-m-approx} can be equivalently
replaced by any compact interval $[a,b]$. Note that by
\Cref{lem:tailtozero-implies-convergencetoclustermeasure}, we already know that there exists at
least one sequence $r_T$ such that $\tailmeasurestar_{T,r_T}\convvague\tailmeasurestar$. The goal of
this lemma is twofold: to prove that $\tailmeasurestar_m\convvague\tailmeasurestar$ and that the
former convergence holds for all scaling functions $r_T$ such that
$ \lim_{T\to\infty} r_T\pr(|\bsX_0|>a_T)=0$ for any scaling function $a_T$.
\begin{proof}[Proof of \Cref{lem:mdep-approx-clustermeasure}]
  Note first that \Cref{eq:condition-m-approx} implies that
  \begin{align}
    \lim_{m\to\infty} \limsup_{T\to\infty} \left| \frac{\pr(|\bsX_0^{(m)}|>Tx)}{\pr(\norm{\bsX_0}>T)} - x^{-\alpha} \right| = 0 \; .
    \label{eq:equivalence-scaling}
  \end{align}
  See \cite[Proposition~5.2.5]{kulik:soulier:2020}. This implies that there exists $m_0\geq1$ such
  that for all $m\geq m$,
  \begin{align}
    \label{eq:uniform-bound}
    \frac12 \leq  \limsup_{T\to\infty}  \frac{\pr(|\bsX_0^{(m)}|>T)}{\pr(\norm{\bsX_0}>T)} \leq 2 \; .
  \end{align}
  As a consequence, if $r_T$ is a scaling function such that
  \begin{align}
    \label{rT-tail-aT}
    \lim_{T\to\infty} r_T\pr(|\bsX_0|>a_T)=0 \; , 
  \end{align}
  then it also holds that
  \begin{align}
    \label{rT-tail-aT-m}
    \lim_{T\to\infty} r_T\pr(|\bsX_0^{(m)}|>a_T)=0 \; , 
  \end{align}
  for all $m\geq m_0$.  Define now the measure $\tailmeasurestar_{T,r_T,m}$ by
  \begin{align*}
    \tailmeasurestar_{T,r_T,m} = \frac{\esp\left[\delta_{a_T^{-1}\bsX_{0,r_T}^{(m)}}\right]}{r_T\pr(|\bsX_0^{(m)}|>a_T)} \; .
  \end{align*}
  By \Cref{lem:AC-mdep} and the previous considerations, the process $\bsX^{(m)}$ being
  $m$-dependent, it satisfies condition \Cref{eq:anticlustering} for all sequences $r_T$ such that
  \Cref{rT-tail-aT} for $m\geq m_0$.  Thus $\tailmeasurestar_{T,r_T,m}\convvague\tailmeasurestar_m$
  for such sequences.
 
  Let $\shinvmap$ be a \shiftinvariant\ map on $\spaceD_0$, Lipschitz \wrt\ the metric $d_\infty$ and such
  that $\shinvmap(\bsy) = 0$ if $\bsy^*\leq 2\epsilon$ for $\epsilon>0$ depending on $\shinvmap$.  Fix
  $\eta<\epsilon$. Then $|\bsx-\bsy|\leq \eta$ and $\norm{\bsx}\wedge\norm{\bsy}\leq \epsilon$ imply
  $\shinvmap(\bsx)=\shinvmap(\bsy)=0$.  Thus,
  \begin{align}
    \left| \tailmeasurestar_{T,r_T}(\shinvmap) - \tailmeasurestar_{T,r_T,m}(\shinvmap) \right| 
    & \leq  \frac{\esp[|\shinvmap(\bsX_0/a_T)-\shinvmap(\bsX^{(m)}_0)|]}{r_T\pr(\norm{\bsX_0}>a_T)} 
      + \constant\ \left(\frac{\pr(\norm{\bsX_0}>a_T)}{\pr(|\bsX_0^{(m)}|>a_T)}-1\right) \nonumber \\
    & \leq  \constant\, \eta\, \tailmeasurestar_{T,r_T}(\{\bsy^*>\epsilon\}) 
      + \frac{\pr \left( \sup_{0 \leq s \leq r_T} \norm{\bsX_s-\bsX_s^{(m)}}>a_T\eta\right)} {r_T\pr(\norm{\bsX_0}>a_T)} 
      \nonumber  \\
    & \phantom{\leq  \constant\, \eta\, \tailmeasurestar_{T,r_T}(\{\bsy^*>\epsilon\}) }
      + \constant\ \left|\frac{\pr(\norm{\bsX_0}>a_T)}{\pr(|\bsX_0^{(m)}|>a_T)}-1\right| \nonumber \\
    & \leq \constant\, \eta\, \tailmeasurestar_{T,r_T}(\{\bsy^*>\epsilon\}) 
      + \frac{\pr \left( \sup_{0 \leq s \leq 1} \norm{\bsX_s-\bsX_s^{(m)}}>a_T\eta\right)} {\pr(\norm{\bsX_0}>a_T)}  \nonumber \\
    & \phantom{\leq  \constant\, \eta\, \tailmeasurestar_{T,r_T}(\{\bsy^*>\epsilon\}) }
      + \constant\ \left|\frac{\pr(\norm{\bsX_0}>a_T)}{\pr(|\bsX_0^{(m)}|>a_T)}-1\right| \; .
        \label{eq:borne-argutri}
  \end{align}
  By \Cref{lem:tailtozero-implies-convergencetoclustermeasure}, there exists a sequence $r_T^0$ such
  that \Cref{rT-tail-aT} holds and $\tailmeasurestar_{T,r_T^0}\convvague \tailmeasurestar$.  Thus,
  for this sequence applying \Cref{eq:condition-m-approx} and \Cref{eq:equivalence-scaling},
  we obtain
  \begin{align*}
    \lim_{m\to\infty}    \limsup_{T\to\infty}    \left| \tailmeasurestar_{T,r_T^0}(\shinvmap) - \tailmeasurestar_{T,r_T^0,m}(\shinvmap) \right|  
    \leq \constant\, \eta\, \tailmeasurestar(\{\bsy^*>\epsilon\})  \; .
  \end{align*}
  Since $\eta$ is arbitrary, the right-hand side is actually 0.  By \Cref{eq:equivalence-scaling},
  we we can also choose $r_T^0$ satisfying \Cref{rT-tail-aT-m} for $m$ large enough, thus
  $\tailmeasurestar_{T,r_T^0,m}\convvague\tailmeasurestar_m$. By
  \Cref{lem:triangular-argument-vague}, this proves that
  $\tailmeasurestar_m\convvague\tailmeasurestar$.  

  By inverting the roles of $\bsX$ and $\bsX^{(m)}$ in the derivations that lead to
  \Cref{eq:borne-argutri}, we obtain
  \begin{multline*}
    \left| \tailmeasurestar_{T,r_T}(\shinvmap) - \tailmeasurestar_{T,r_T,m}(\shinvmap) \right| \\
    \leq \constant\, \eta\, \tailmeasurestar_{T,r_T,m}(\{\bsy^*>\epsilon\}) + \frac{\pr \left(
        \sup_{0 \leq s \leq 1} \norm{\bsX_s-\bsX_s^{(m)}}>a_T\eta\right)}
    {\pr(|\bsX_0^{(m)}|>a_T)}  \\
    + \constant\ \left|\frac{\pr(|\bsX_0^{(m)}|>a_T)}{\pr(\norm{\bsX_0}>a_T)}-1\right| \; .
  \end{multline*}
  Since $\tailmeasurestar_{T,r_T,m}\convvague\tailmeasurestar_m$ for all sequences $r_T$ satisfying
  \Cref{rT-tail-aT} and $m$ large enoug, this yields
  \begin{align*}
    \lim_{m\to\infty}    \limsup_{T\to\infty}    \left| \tailmeasurestar_{T,r_T}(\shinvmap) - \tailmeasurestar_{T,r_T,m}(\shinvmap) \right|  
    \leq \constant\, \eta\, \tailmeasurestar_m(\{\bsy^*>\epsilon\})  \; .
  \end{align*}
  Since $\eta$ is arbitrary, the right-hand side is actually 0.  By
  \Cref{lem:triangular-argument-vague} again, this proves that
  $\tailmeasurestar_{T,r_T}\convvague\tailmeasurestar$ in $(\spaceDtilde_0,\tilde\mcb_0)$.
\end{proof}

\subsection{The point process of clusters}
\label{sec:convpp}
We now define the functional point process of clusters on $[0,\infty)\times \unzerospaceDtilde$. For
$i\in\Nset^*$, set $\bsX_{T,i}=a_T^{-1}\bsX\1{[(i-1)r_T,ir_T)}$, identified with its equivalence
class in $\spaceDtilde_0$. Set also $m_T = Tr_T^{-1}$ and 
\begin{align}
  \label{eq:def-ppcluster}
  N_T = \sum_{i=1}^\infty \delta_{\frac{i}{m_T},\bsX_{T,i}} \; .
\end{align}
This point process is related to the excursion random measure $\zeta_T$ of \cite{hsing:leadbetter:1998} defined by
\begin{align*}
  \zeta_T = \int_0^T \delta_{\frac{t}T,\frac{X_t}{a_T}} \rmd  t \; . 
\end{align*}
Indeed, for a function $f$ defined on $[0,\infty)\times\Rset$, let the map 
$K_f$ be (formally) defined on $[0,\infty)\times\spaceDtilde_0$ by
\begin{align*}
  K_f(t,\bsy) = \int_{-\infty}^\infty f(t,\bsy_s) \, \rmd s \; . 
\end{align*}
This yields 
\begin{align*}
  N_T(K_f) = \sum_{i=1}^\infty \int_{(i-1)r_T}^{ir_T} f\left(\frac{i}{m_T},\frac{X_s}{a_T}\right) \rmd s 
\end{align*}
If $f$ vanishes for large $t$ and for $|x|\leq\epsilon$ and is Lipschitz continuous \wrt\ $(t,x)$, then
$N_T(K_f) -\zeta_T(f) \convprob 0$.  Thus the convergence of $N_T$ implies that of $\zeta_T$. The
random measure $N_T$ contains more information than $\zeta_T$. See \cite{basrak:planinic:soulier:2018}
for the corresponding discussion in discrete time.

The convergence of $N_T$ will be established under the following unprimitive assumption which
validates a blocking method. It is a classical assumption in extreme value theory for stochastic
processes. See e.g. Condition $\mca(a_n)$ in \cite{davis:hsing:1995} for time series. It is implied
by condition $\Delta(u_T)$ of \cite{hsing:leadbetter:1998} in continuous time.
\begin{hypothesis}
  \label{hypo:mixing-laplace}
  There exist scaling functions $a$ and $r:\Rset_+\to\Rset_+$ such that for all bounded
  continuous maps $f:\Rset_+\times\unzerospaceDtilde\to\Rset_+$ such that $f(t,\bsy)=0$ if $|t|> A$
  or $\bsy^*\leq \epsilon$ for some $A$ and $\epsilon>0$ (depending on $f$),
  \begin{align}
   \label{eq:laplace-blocks}
    \limsup_{T\to\infty} \left| \esp\left[\rme^{-N_T(f)}\right] - \prod_{i=1}^\infty \esp\left[\rme^{-f(ir_T/T,\bsX_{T,i})}\right] \right| = 0 \; .
  \end{align}
\end{hypothesis}

Under this assumption, we extend \cite[Theorem~3.6]{basrak:planinic:soulier:2018} to the continuous
time framework and  \cite[Theorem~4.1]{hsing:leadbetter:1998}.
\begin{theorem}
  \label{theo:equivalence-ppconv}
  Let $\bsX$ be a stationary $\spaceD$-valued stochastic process, regularly varying in $\spaceD$,
  with cluster measure $\tailmeasurestar$. Assume that there exist scaling functions $a_T$ and $r_T$ such
  that $\lim_{T\to\infty} T\pr(\norm{\bsX_0}>a_T)=1$ and \Cref{eq:anticlustering} holds.  Let $N_T$
  be the point process of cluster defined in \Cref{eq:def-ppcluster} with the same functions $a$ and
  $r$. The following statements are equivalent:
  \begin{itemize}
  \item \Cref{hypo:mixing-laplace}  holds with the same functions $a$ and $r$;
  \item $N_T \convweak N$ with $N$ a PPP on $\Rset_+\times\unzerospaceDtilde$ with mean measure $\leb\otimes\tailmeasurestar$.
  \end{itemize}
\end{theorem}

A PPP $N$ with mean measure $\leb\otimes\tailmeasurestar$ has the representation
\begin{align}
  \label{eq:limit-PPE}
  N = \sum_{i=1}^\infty \delta_{T_i,P_i\bsQ^{(i)}} \; , 
\end{align}
where $\sum_{i=1}^\infty \delta_{T_i,P_i}$ is a PPP on $\Rset_+\times(0,\infty)$ with mean measure 
$\candidate\leb\otimes\nu_\alpha$ and $\bsQ^{(i)}$ are \iid\ copies of a process $\bsQ$ whose distribution is
given by \Cref{eq:loi-Q} and
\begin{align}
  \label{eq:rappel-candidate}
  \candidate & = \esp \left[ \frac1{\exc(\bsY)} \right] \; .
\end{align}
.

\begin{proof}[Proof of \Cref{theo:equivalence-ppconv}]
  The proof is essentially the same as the proof of
  \cite[Theorem~3.6]{basrak:planinic:soulier:2018}. See also
  \cite[Theorem~7.3.1]{kulik:soulier:2020}. We briefly sketch it. Let $\bsX_{T,i}^\dag$, $i\geq1$,
  be independent random elements in $\spaceD_0$ with the same distribution as $\bsX_{T,1}$. Define
  $\xi_{T,i}= \delta_{\frac{i}{m_T},\bsX_{T,i}}$, $i\geq1$. Then $\{\xi_{T,i},i\geq1\}$ is a null
  array of point processes in the sense of \cite[Section~4.3]{kallenberg:2017}. By stationarity, the
  mean measure of $\xi_{T,i}$ is $\delta_{\frac{i}{m_T}}\otimes \tailmeasurestar_{T}$. Denote
  $\nu_T = \sum_{i=1}^\infty \delta_{\frac{i}{m_T}}\otimes \tailmeasurestar_{T}$. By
  \cite[Theorem~7.1.6]{kulik:soulier:2020}, the stated convergence is equivalent to the convergence
  $\nu_T\convvague\leb\otimes\tailmeasurestar$. This boils down to proving that
  $\sum_{i=1}^\infty \delta_{\frac{i}{m_T}}\convvague\leb$ which is trivial.
\end{proof}

A real-valued stationary process $\bsX$ is said to have extremal index $\theta$ if for every
$\tau>0$ and sequence $u_T$ such that $\lim_{T\to\infty} T\pr(X_0>u_T) = \tau$, it holds that
$\lim_{T\to\infty} \pr(\sup_{0\leq s \leq T} X_s \leq u_T) = \rme^{-\theta\tau}$.  For a stochastic
process indexed by $\Zset$, the extremal index, if it exists, must be in $[0,1]$. For a continuous
time process it can be in $[0,\infty]$. Under \Cref{hypo:regvar-in-D,hypo:mixing-laplace}, we show
that it exists in $(0,\infty)$.

\begin{corollary}
  \label{coro:extremalindex}
  Let $\bsX$ be a stationary $\spaceD$-valued process, regularly varying in $\spaceD$ with tail
  index $\alpha>0$. If \Cref{hypo:anticlustering,hypo:mixing-laplace} hold with a scaling function
  $a_T$ such that $\lim_{T\to\infty} T\pr(\norm{\bsX_0}>a_T)=1$, then
  \begin{align*}
    \lim_{T\to\infty} \pr\left( \sup_{0\leq t \leq T} \norm{\bsX_s} \leq a_Tx \right) = \rme^{-\candidate x^{-\alpha}} \; .
  \end{align*}
\end{corollary}

\begin{proof}
  The point process convergence of \Cref{theo:equivalence-ppconv} yields:
  \begin{align*}
    \pr\left( \sup_{0\leq t \leq T} \norm{\bsX_s} \leq a_Tx \right) 
    & = \pr(N_T([0,1]\times\{\bsy^*>x\})=0) \\
    & \to  \pr(N([0,1]\times\{\bsy^*>x\})=0) \\
    & = \rme^{-\tailmeasurestar(\{\bsy^*>x\})} = \rme^{-\candidate x^{-\alpha}} \; .
  \end{align*}
\end{proof}

We now provide conditions which ensure \Cref{hypo:anticlustering} or \Cref{hypo:mixing-laplace}. For
the definition of $\beta$-mixing, see \cite{bradley:2005}.
\begin{lemma}
  \label{lem:beta-implies-laplace-block}
  Let $\bsX$ be stationary $\mathcal{D}$-valued stochastic process, $\beta$-mixing with rate
  $\beta_t$ and assume that there exist sequences $\{r_T\}$ and $\{\ell_T\}$ such that
  \begin{align*}
    \lim_{T\to\infty} \frac{r_T}{T} =    \lim_{T\to\infty} \frac{\ell_T}{r_T} =   \lim_{T\to\infty} \frac{T\beta_{\ell_T}}{r_T} = 0 \; .
  \end{align*}
  Then \Cref{hypo:mixing-laplace} holds with $a_T$ such that $T\pr(\norm{\bsX_0}>a_T)\to1$.
\end{lemma}
\begin{proof}
  The proof is similar to the proof in the case of discrete time processes, see
  \cite[Lemma~6.2]{basrak:planinic:soulier:2018}.  
\end{proof}

\begin{lemma}
  \label{lem:mdep}
  If $\bsX$ is an $m$-dependent regularly varying stationary $\spaceD$-valued stochastic process,
  then \Cref{hypo:anticlustering,hypo:mixing-laplace} hold for all scaling functions $a_T$ and $r_T$ such
  that \Cref{eq:rtprattozero} holds.
\end{lemma}

\begin{proof}
  We already know that \Cref{hypo:anticlustering} holds by \Cref{lem:AC-mdep}.  Since $m$-dependent
  sequences are $\beta$-mixing with arbitrarily fast rates, Assumption \Cref{hypo:mixing-laplace}
  holds by \Cref{lem:beta-implies-laplace-block}.
\end{proof}

We now consider processes which admit a sequence of tail equivalent approximations.
\begin{theorem}
  \label{theo:ppconv-approx}
  Let $\bsX$ be a stationary process, regularly varying in $\spaceD$ with tail process $\bsY$ such
  that $\pr(\bsY\in\spaceD_0)=1$ and cluster measure $\tailmeasurestar$.  Assume that there exists a
  sequence of $m$-dependent stationary processes $\bsX^{(m)}$, regularly varying in $\spaceD$, such
  that $(\bsX,\bsX^{(m)})$ is stationary for every $m$ and \Cref{eq:condition-m-approx} holds.  Then
  the point process of clusters $N_T$ converges weakly to a Poisson point process $N$ on
  $[0,\infty)\times \spaceDtilde\setminus\{\bszero\}$ with mean measure
  $\leb\otimes\tailmeasurestar$.
\end{theorem}

\begin{proof}
  Let $N_T^{(m)}$ be the point process of clusters of the process $\bsX^{(m)}$. Since $\bsX^{(m)}$
  is regularly varying in $\spaceD$ and $m$-dependent, $N_T^{(m)}\convweak N^{(m)}$ with $N^{(m)}$ a
  Poisson point process with mean measure $\leb\otimes\tailmeasurestar_m$. Since
  $\tailmeasurestar_m\convvague\tailmeasurestar$, it also holds that $N^{(m)}\convweak N$. If we
  moreover prove that for all $\epsilon>0$ and Lipschitz (\wrt\ the $d_{J_1}$ distance) continuous
  maps $\map$ such that $\map(\bsy)=0$ if $\bsy^*\leq\epsilon$ and for all $\eta>0$,
  \begin{align}
    \label{eq:triangularargument-pp}
    \lim_{m\to\infty} \limsup_{T\to\infty} \pr(|N_T(\map)-N_T^{(m)}(\map)| > \eta) = 0 \; , 
  \end{align}
  then by \Cref{prop:triangular-argument-weak-vague}, we will have proved that $N_T \convweak N$. 

  We now prove \Cref{eq:triangularargument-pp}.  {For simplicity, we drop the time component which
    brings no difficulty but only additional notational complexity. This means that we replace $N$
    by $N([0,1]\times\cdot)$.}  Let $\epsilon>0$ and $\shinvmap$ be a \shiftinvariant\ map on
  $\spaceD_0$, Lipschitz continuous \wrt\ the metric $d_{J_1}$ and such that $\shinvmap(\bsy)=0$ if
  $\bsy^*\leq2\epsilon$. Let $B_\epsilon^c$ denote the complement of the ball centered at 0 with
  radius $\epsilon$ \wrt\ the metric $d_{J_1}$. Then, for $\delta<\epsilon$,
  \begin{align*}
    \pr(|N_T(\shinvmap) - N_{T,m}(\shinvmap)| > \eta) 
    & \leq \pr \left(\sum_{i=1}^{[T/r_T]} |\shinvmap(\bsX_{T,i}) - \shinvmap(\bsX_{T,i}^{(m)})| >\eta \right)  \\
    & \leq \frac{T}{r_T} \pr(d_{J_1}(\bsX_{T,1},\bsX_{T,1}^{(m)}) > \delta) + 
      \pr\left( K \epsilon N_T^{(m)}(B_\epsilon^c) > \eta \right)   \\
    & \leq \frac{T}{r_T} \pr \left( \sup_{0 \leq s \leq r_T} \norm{\bsX_s-\bsX_s^{(m)}}  > a_T\delta \right) + 
      \pr\left( K \delta N_T^{(m)}(B_\epsilon^c) > \eta \right)   \\
    & \leq T \pr \left( \sup_{0 \leq s \leq 1} \norm{\bsX_s-\bsX_s^{(m)}}  > a_T\delta \right) + 
      \pr\left( K \delta N_T^{(m)}(B_\epsilon^c) > \eta \right) \; .
  \end{align*}
  By \Cref{eq:condition-m-approx}, the first term in the last equation vanishes when $n$, then $m$
  tend to $\infty$.  The weak convergence of $N_T^{(m)}$ and Markov inequality yield
  \begin{align*}
    \lim_{T\to\infty} \pr\left( K \delta N_T^{(m)}(B_\epsilon^c) > \eta \right) 
    =  \pr\left( K \delta N^{(m)}(B_\epsilon^c) > \eta \right) \leq K\delta\eta^{-1}\tailmeasurestar_m(B_\epsilon^c)   \; .
  \end{align*}
  Since $\tailmeasurestar_m\convvague\tailmeasurestar$, we obtain
  \begin{align*}
    \lim_{m\to\infty} \lim_{T\to\infty} \pr\left( K \delta N_T^{(m)}(B_\epsilon^c) > \eta \right) \leq \constant\ \delta \; .
  \end{align*}
  Since $\delta$ is arbitrary, this proves \Cref{eq:triangularargument-pp} and concludes the proof
  of \Cref{theo:ppconv-approx}.
\end{proof}

We conclude this section with a comparison between the candidate extremal index and the true one if
it exists.
\begin{lemma}
  \label{lem:nullextremalindex}
  Let $\bsX$ be a $\spaceD$-valued regularly varying process and $\candidate$ be defined as in
  \Cref{eq:rappel-candidate}.  If $\candidate=0$ then $\bsX$ admits an extremal index which is equal
  to $0$. If $\bsX$ admits an extremal index $\theta$, then $\theta\leq \candidate$.
\end{lemma}

\begin{proof}
  For $u<v$ and a measurable function $f$, write $\exc_{u,v}(f) = \int_{u}^v \ind{|f(s)|>1} \rmd s$.
  Fix $\eta>0$.  By stationarity, we have
  \begin{align*}
    \pr(\bsX_{0,T}^*>a_Tx;  & \ \exc_{0,T}(\bsX/(a_Tx))>\eta) \\
    & = \int_0^T \esp \left[ \frac{\ind{|X_0|>a_Tx}\ind{\exc_{-s,T-s}(\bsX/(a_Tx))>\eta}} {\exc_{-s,T-s}(\bsX/(a_Tx))}  \right] \rmd s \\
    & = T\pr(|X_0|>a_Tx)
      \int_0^1 \esp \left[ \frac{\ind{\exc_{-sT,T(1-s)}(\bsX/(a_Tx))>\eta}} {\exc_{-Ts,T(1-s)}(\bsX/(a_Tx))}  \mid |X_0|>a_Tx\right] \rmd s \; .
  \end{align*}
  Fix $A>0$. Then for $T>2A$, 
  \begin{multline*}
    \int_0^1 \esp \left[ \frac{\ind{\exc_{-sT,T(1-s)}(\bsX/(a_Tx))>\eta}}
      {\exc_{-Ts,T(1-s)}(\bsX/(a_Tx))}  \mid |X_0|>a_Tx\right] \rmd s \\
    \leq \frac{2A}{\eta T} + \left(1-\frac{2A}T\right) \esp \left[ \frac{1}
      {\exc_{-A,A}(\bsX/(a_Tx))} \wedge \frac1\eta \mid |X_0|>a_Tx\right] \; .
  \end{multline*}
  This yields by dominated convergence 
  \begin{align*}
    \limsup_{T\to\infty}   \int_0^1 \esp
    & \left[ \frac{\ind{\exc_{-sT,T(1-s)}(\bsX/(a_Tx))>\eta}} {\exc_{-Ts,T(1-s)}(\bsX/(a_Tx))}  \mid |X_0|>a_Tx\right] \rmd s
      \leq \esp \left[ \frac1{\exc_{-A,A}(\bsY)} \wedge\frac1\eta\right]
  \end{align*}
  Arguing as in the proof of \Cref{lem:convtailmeasurestar},  we have, by \Cref{theo:rv-in-D-equivalence}
  \begin{align*}
    \pr(\bsX_{0,T}^*>a_Tx;\exc_{0,T}(\bsX/(a_Tx))\leq \eta)  \leq T \pr(w'(\bsX,0,1,\eta)>a_Tx/2) \to 0 \; .
  \end{align*}
  Since $A$ and $\eta$ are  arbitrary, we obtain by definition of $\candidate$,
  \begin{align*}
   \limsup_{T\to\infty} \pr(\bsX_{0,T}^*>a_Tx) \leq \candidate x^{-\alpha} \; .
  \end{align*}
  If $\candidate=0$, then $\pr(\exc(\bsY)=\infty)$, whence
  \begin{align*}
    \lim_{T\to\infty}\pr(\bsX_{0,T}^*>a_Tx) = 0 \; .
  \end{align*}
  By definition, this proves that the extremal index $\theta$ exists and $\theta=0$. If the extremal
  index exists, then we have proved that
  \begin{align*}
    1 - \rme^{-\theta x^{-\alpha}} = \lim_{T\to\infty} \pr(\bsX_{0,T}^*>a_Tx) \leq  \candidate x^{-\alpha} \; .
  \end{align*}
  Multiplying both sides by $x^\alpha$ and letting $x\to0$ proves that $\theta\leq\candidate$. 
\end{proof}

\subsection{Illustrations}
\label{sec:illustration}
We now describe informally the type of results that can be obtained with the tools of
\Cref{sec:representation,sec:regvarinD}, which highlight the practical usefulness of the identities
of \Cref{sec:identities} and the measure $\tailmeasurestar$ introduced in
\Cref{sec:cluster-measure}.  A rigororous investigation of these problems is beyond the scope of
this paper.  See \cite[Chapters~8-10]{kulik:soulier:2020} for a review of the corresponding results
in discrete time. We consider univariate processes for simplicity. 

\subsubsection*{Convergence to $\alpha$-stable processes}
The convergence of the point process of clusters can be used to prove limit theorems such as
convergence of the partial sum process. For $\alpha\in(0,1)$, under
\Cref{hypo:anticlustering,hypo:mixing-laplace}, it can be proved by a truncation and continuous
mapping argument applied to the point process of clusters that
\begin{align*}
  a_T^{-1} \int_0^{Tt} X_s \rmd s - c_T \fidi \Lambda(t) \; , 
\end{align*}
(where $\fidi$ means weak convergence of finite dimensional distributions) with $c_T$ a suitable
centering, $\Lambda$ a L\'evy stable process such that,
\begin{align*}
  \log\esp[\rme^{\rmi z \Lambda(1)}] = -\sigma^\alpha |z|^\alpha \{1 - \rmi \beta \mathrm{sign}(z) \tan(\pi\alpha/2)\} \; , 
\end{align*}
with
\begin{align*}
  \sigma^\alpha & = \alpha \Gamma(1-\alpha) \cos(\pi\alpha/2)\esp \left[ \left|\int_0^\infty \Theta_s\rmd s\right|^{\alpha-1} \right] \; , \\
  \beta & = \frac{\esp \left[ \Theta_0\left(\int_0^\infty \Theta_s\rmd s\right)^{<\alpha-1>} \right]}
          {\esp \left[ \left|\int_0^\infty \Theta_s\rmd s\right|^{\alpha-1} \right]}
\end{align*}
with $x^{<a>} = \mathrm{sign}(x) |x|^a$ for all $x\ne0$ and $a\ne0$, and assuming that
$\sigma\ne0$. If $\alpha<1$, we have seen in \Cref{sec:identities} that the above quantities are
always finite.  If $1 < \alpha < 2$, the same convergence can be proved to hold under an extra
assumption which guarantees that $\sigma$ and $\beta$ are still well defined and a negligibility
assumption is needed to handle the small jumps of the L\'evy process. For $\alpha=1$, an additional
centering term appears in the limiting stable law. The convergence can be proved in the $J_1$ or
$M_1$ topology under some additional conditions.  Since this convergence is obtained by means of the
point process of clusters, the expressions of $\sigma^\alpha$ and $\beta$ are obtained first in
terms of the sequence $\bsQ$ and then translated in terms of the forward spectral tail process using
\Cref{lem:identities}.  See \cite{basrak:planinic:soulier:2018} for exhaustive results in discrete
time.

\subsubsection*{Estimation of cluster functionals}
Many extreme value statistical problems involved so-called cluster functionals, introduced in
discrete time by \cite{drees:rootzen:2010} without the formalism of the tail process. See
\cite[Chapter~10]{kulik:soulier:2020} for a complete presentation of such results in discrete time.

In the language of \Cref{sec:convpp}, quantities of interest are often of the form
$\tailmeasurestar(K)$ with $K$ a functional on $\spaceD_0$. For instance, the extreme value index
$\gamma=\alpha^{-1}$ can be expressed as $\gamma = \tailmeasurestar(K_{\log_+})$ with
$\log_+(x) = \log(x\vee1)$ and $K_{\log}(\bsy) = \int_{-\infty}^\infty \log_+(|y_s|) \rmd s$. For
$\bsy\in\spaceD_0$, the integral is over a finite interval hence well defined and finite. By
\Cref{eq:tailmeasurestar-Q} and \Cref{eq:candidate-Q}, we have
\begin{align*}
  \tailmeasurestar(K_{\log_+})
  & = \candidate \int_{-\infty}^\infty \int_0^\infty \esp[\log_+(r|Q_s|)]  \alpha r^{-\alpha-1} \rmd r \rmd s \\
  & = \candidate \int_{-\infty}^\infty \esp[|Q_s|^\alpha]  \rmd s \int_0^\infty \log_+(r)\alpha r^{-\alpha-1} \rmd r  = \gamma \; .
\end{align*}
Another example is given by the relation $\candidate = \tailmeasurestar(K_e)$ which is a direct
consequence of \Cref{eq:tailmeasurestar-Q} with $K_e(\bsy) = \ind{\bsy^*>1}$. The assumptions which
are needed to obtain results for estimators of $\candidate$ also imply that $\candidate=\theta$, the
true extremal index.

A natural estimator of $\tailmeasurestar(K)$ given the observation of a path $X_s,s\in[0,T]$ is
obtained by $\widehat{\tailmeasurestar}_{T}(K)$ with 
\begin{align*}
  \widetilde{\tailmeasure_T^*} = \frac1{T\pr(|X_0|>u_T)} \sum_{i=1}^{K_T} \delta_{\bsX_{T,i}} \; , 
\end{align*}
with $\bsX_{T,i} = (X_s/u_T)_{(i-1)r_T\leq s \leq ir_T}$, $r_T$ and $u_T$ are increasing functions such
that $T/r_T\to0$, $r_T\pr(|X_0|>u_T)\to0$, $T\pr(|X_0|>u_T)\to\infty$ and $K_T = [T/r_T]$. Such
estimators are called block estimators. The random measure $\widetilde{\tailmeasure_T^*} $ can be
called an empirical cluster measure, since it is expected to converge weakly to $\tailmeasurestar$.

Actually, $\widetilde{\tailmeasure_T^*}$ is not a feasible estimator, since the factor in the
denominator depends on the marginal distribution. There are different ways to deal with this issue
which we will not discuss here. Under mixing conditions such as $\beta$-mixing and under condition
which guarantee the existence of the limiting variance, we expect to prove a central limit theorem
of the form
\begin{align*}
  \sqrt{T\pr(|X_0|>u_T)} \{  \widetilde{\tailmeasure_T^*} - \tailmeasurestar(K)\} \convdistr N(0,\tailmeasurestar(K^2)) \; .
\end{align*}
We can compute $\tailmeasurestar(K^2)$ for the examples cited above. This is straightforward for the
estimator of the extremal index since $K_e^2=K_e$, thus $\tailmeasurestar(K_e^2)=\candidate$.  The
computations are harder for the estimator of the extreme value index.  For $\bsy\in\spaceD_0$,
define $\ell_s(\bsy) =\log_+(y_s)$.  Under assumptions which guarantee that the variance is finite,
we have, using \Cref{eq:tailmeasure-nu-nustar} and the shift-invariance of $\tailmeasure$,
\begin{align*}
  \tailmeasurestar(K_{\log}^2)
  & = \int_{-\infty}^\infty \int_{-\infty}^\infty \tailmeasurestar(\ell_s(\bsy)\ell_t(\bsy)) \rmd s \rmd t
    = \int_{-\infty}^\infty \tailmeasure(\ell_0(\bsy)\ell_t(\bsy)) \rmd t \; .
\end{align*}
Since $\ell_0(\bsy)=0$ if $|y_0|\leq1$, we obtain, using the definition of the tail process in terms
of the tail measure, 
\begin{align*}
  \tailmeasurestar(K_{\log}^2)
  & = \int_{-\infty}^\infty \esp[\log_+(|Y_0|) \log_+(|Y_t|)] \rmd t
    = \gamma \int_{-\infty}^\infty \esp[(|Y_{t}|\wedge1)^\alpha] \rmd t \; . 
\end{align*}
The last identity is obtained by a repeated application of the time change formula \Cref{eq:TCF-Y}
and the identity $\log_+(y) = \int_1^\infty \ind{y>u} u^{-1} \rmd u$. Conditions are needed to ensure that the integral is finite. 

\section{Examples}
\label{sec:examples}

In this section we will study several classes of stochastic processes. For each, we will prove
regular variation in $\spaceD$, check the conditions of \Cref{theo:ppconv-approx} and thus obtain
their extremal index. The novelty of our results is the regular variation in $\spaceD$ and the
convergence of the point process of clusters. We start with max- and sum-stable processes whose
distributions are entirely determined by the tail measure. {The link between max- and
  sum-stable processes and their spectral representations was brilliantly highlighted in
  \cite{kabluchko:2009}.}  Next, we consider functionally weighted sums of \iid\ regularly varying
random variables.

The list of models is limited to keep the paper at a reasonable length. Other interesting examples
include L\'evy driven (mixed) moving average processes: \cite[Section~4]{hult:lindskog:2005},
\cite[Section~4]{fasen:2006}, \cite{fasen:kluppelberg:2007} for which we expect that regular
variation in $\spaceD$ holds and \Cref{theo:ppconv-approx} can be applied. An important class of
models already studied in discrete time consists of regularly varying functions of Markov
processes. It is proved in \cite{kulik:soulier:wintenberger:2019} that \Cref{hypo:anticlustering}
holds for functions of geometrically ergodic Markov chains whose kernel satisfies a suitable drift
condition. Since these chains are also $\beta$-mixing with geometric rate, they fulfill the
conditions of \Cref{theo:equivalence-ppconv}. Extending these results for continuous time Markov
processes would provide a wealth of useful examples.

\subsection{Max-stable processes}
\label{sec:maxstable}

Let $\bseta$ be a real-valued max-stable process with $\alpha$-Fr\'echet marginal distributions,
\ie\ for all $k\geq1$, $t_1,\dots,t_k\in\Rset$ and $x_1,\dots,x_k>0$,
$\vee_{i=1}^kx_i^{-1}\eta(t_i)$ has an $\alpha$-Fr\'echet distribution.  The finite dimensional
distributions of $\bseta$ are regularly varying and are completely determined by their exponent
measures. Hence the tail measure $\tailmeasure$ in the sense of \cite{owada:samorodnitsky:2012}
entirely determines the distribution of $\bseta$. This means that there is a one-to-one
correspondance max-stable processes and tail measures.

If we assume moreover that $\tailmeasure$ is a tail measure in the sense of
\Cref{def:tailmeasureonD} with spectral process $\bsZ$, then we can define a max-stable process
$\tilde\bseta$ by
\begin{align}
  \label{eq:maxstable-associe-tailmeasure}
  \tilde\eta_t = \bigvee_{i=1}^\infty P_i Z^{(i)}_t \; , t \in \Rset \; .
\end{align}
with $\sum_{i=1}^\infty \delta_{P_i}$ a Poisson point process on $(0,\infty)$ with mean measure
$\nualpha$ and $\bsZ^{(i)}$, $i\geq1$ \iid\ copies of $\bsZ$. Then $\tilde\bseta$ has the same
distribution as $\bseta$, since it has the same finite dimensional distributions, expressed in terms
of the spectral process by
\begin{align}
  \label{eq:carac-maxstable-fidi}
  -\log \pr\left( \bigvee_{i=1}^k \frac{\eta_{t_i}}{x_i} \leq  1 \right)  = \esp \left[ \bigvee_{i=1}^k \frac{Z_{t_i}^\alpha}{x_i^\alpha} \right] \; , 
\end{align}
for $k,t_1,\dots,t_k,x_1,\dots,x_k$ as above.  There is an important literature on spectral
representations of max-stable processes, starting with \cite{vatan:1985} and \cite{dehaan:1984},
continuing more recently with \cite{kabluchko:2009} and \cite{wang:stoev:2010} among others.  The
functional context was studied in \cite{gine:hahn:vatan:1990} who considered max-stable (and max-id)
processes with continuous paths.

A spectral process is not unique in distribution. Any process $\bsZ'$ such that
\begin{align*}
  \esp[\homap(\bsZ')] = \esp[\homap(\bsZ)]
\end{align*}
for all $\alpha$-homogeneous maps $\homap$ is a valid spectral process.  The distribution of a
max-stable process can be equivalently characterized by finite dimensional distributions as in
\cref{eq:carac-maxstable-fidi} or by homogeneous functionals. For notational convenience, we will
prefer the latter hereafter.

We now prove that a max-stable process in $\spaceD$ with the representation
\cref{eq:maxstable-associe-tailmeasure} is regularly varying in~$\spaceD$.

\begin{theorem}
  \label{theo:maxstable}
  Let $\alpha>0$, $\tailmeasure$ be a tail measure on $\spaceD$ with tail index $\alpha$ and $\bsZ$
  be a spectral process for $\tailmeasure$. Let $\bseta$ be a $\spaceD$-valued process. The
  following statements are equivalent:
  \begin{enumerate}[(i)]
  \item \label{item:maxstableD} $\bseta$ is a max-stable process with $\alpha$-Fr\'echet
    marginal distributions, regularly varying in $\spaceD(\Rset,\Rset)$, with tail measure $\tailmeasure$;
  \item \label{item:poissonrepresentation} $\bseta$ admits the representation
    \Cref{eq:maxstable-associe-tailmeasure}.
  \end{enumerate}
\end{theorem}
\begin{proof}
  Since the tail measure determines the distribution of $\bseta$, the implication
  \cref{item:maxstableD}$\implies$\cref{item:poissonrepresentation} is straightforward, since the
  process in the representation \cref{eq:maxstable-associe-tailmeasure} has tail measure
  $\tailmeasure$, thus has the same distribution as $\bseta$.  We now prove the converse
  implication. If \cref{item:poissonrepresentation} holds, then $\eta_0$ has an $\alpha$-Fr\'echet
  distribution, and we must prove that for all $\epsilon>0$, $a>0$ and bounded Lipschitz continuous
  maps $\map$ on $\spaceD$ such that $\map(\bsy)=0$ if $\bsy_{-a,a}^*\leq \epsilon$,
  \begin{align}
    \label{eq:whatwemustprove}
    \lim_{T\to\infty}  T^\alpha \esp[\map(T^{-1} \bseta)] =  \tailmeasure(\map) \; . 
  \end{align}
  Let $\bsW$ be a process whose distribution is characterized by
  \begin{align*}
    \esp[\map(\bsW)] = \frac{\esp [ \map((\bsZ_{a,b}^*)^{-1}\bsZ)(\bsZ_{a,b}^*)^\alpha]}{\esp[(\bsZ_{a,b}^*)^\alpha]} \; .
  \end{align*}
  for all \nonnegative\ measurable maps $\map$.  Let $\bsW^{(i)}$ be \iid\ copies of $\bsW$.  Then,
  $\bseta$ has the same distribution as
  \begin{align*}
    c_\alpha^{1/\alpha} \sum_{i=1}^\infty P_i \bsW^{(i)} \; , 
  \end{align*} 
  with $c_\alpha=\esp[\sup_{a\leq s \leq b} |Z(s)|^\alpha]$.  Since $P_1$ has a Fr\'echet distribution and $\sup_{-a\leq t\leq a} |W_t| = 1$,
  we have
  \begin{align*}
    \lim_{T\to\infty}    T^\alpha  \esp[\map(c_\alpha^{1/\alpha}T^{-1}P_1\bsW)] 
    & = c_\alpha \int_0^\infty \esp[\map(u\bsW)] \alpha u^{-\alpha-1} \rmd u \\
    & = \int_0^\infty \esp[\map(u\bsZ)] \alpha u^{-\alpha-1} \rmd u = \tailmeasure(\map) \; .
  \end{align*}
  By definition of the metric $d_{J_1}$, for every $f,g\in\spaceD(\Rset,\Rset_+)$ and $a>0$, we have
  \begin{align*}
    d_{J_1}(f,f\vee g) \leq \sup_{-a\leq s \leq a} |f(s)-f\vee g(s)| + \rme^{-a}
    \leq \sup_{-a\leq s \leq a} g(s) + \rme^{-a} \; .   
  \end{align*}
  Since $\map$ is bounded, Lipschitz continuous \wrt\ the $J_1$ metric and has support separated from
  $\bszero$, and $\bsW_{a,b}^*=1$, we have, for all $\beta<\epsilon/2$,
  \begin{align*}
    T^\alpha |\esp[\map(T^{-1}P_1\bsW^{(1)})] & - \esp[\map(T^{-1}\bseta)]| \\
    & \leq \constant\ T^\alpha \pr(d_{J_1} (P_1\bsW^{(1)},\vee_{i=1}^\infty P_i\bsW^{(i)})>T\beta) + \constant\ T^\alpha \pr(P_1>T\epsilon/2)  \\
    & \leq \constant\ T^\alpha \pr(P_2+\rme^{-a} > T\beta) + \constant\ \beta T^\alpha \pr(P_1>T\epsilon/2)  \; .
  \end{align*}
  This yields, for all $\beta< \epsilon/2$,
  \begin{align*}
    \limsup_{T\to\infty} T^\alpha |\esp[\map(T^{-1}P_1\bsW^{(1)})] & - \esp[\map(T^{-1}\bseta)]| \leq \constant \beta \; .
  \end{align*}
  Since $\beta$ is arbitrary, the $\limsup$ is actually zero. Altogether, we have proved \Cref{eq:whatwemustprove}.
\end{proof}
Applying the results of \Cref{sec:representation}, we obtain an explicit mixed moving
maximum representation of max-stable processes generated by a dissipative flow. We also recover the
result of \cite[Theorem~2.1 and~2.3]{debicki:hashorva:2020} and give a new expression for the
extremal index.

\begin{corollary}
  \label{theo:M3representation-maxstable}
  Let $\tailmeasure$ be a \shiftinvariant\ tail measure on $\spaceD$ with tail process $\bsY$.  Let
  $\bseta$ be the associated max-stable process defined by \Cref{eq:maxstable-associe-tailmeasure}.
  Then $\bseta$ admits an extremal index $\theta$ equal to $\candidate$,~\ie\
  \begin{align*}
    \theta =  \candidate = \esp \left[ \frac1{\exc(\bsY)} \right] <  \infty \; .
  \end{align*}
  The extremal index is positive \ifft\ $\tailmeasure(\spaceD_0)>0$.  The process $\bseta$ admits a
  mixed moving maxima representation \ifft\ $\tailmeasure(\spaceD_0^c)=0$. In the latter case,
  $\candidate>0$ and
  \begin{align}
    \label{eq:M3representation}
    \eta_t = \bigvee_{i=1}^\infty P_i Q_{t-T_i}^{(i)}
  \end{align}
  with $\{T_i,P_i,Q_i\}$ the points of a Poisson point process on
  $\Rset\times(0,\infty)\times \spaceD$ with mean measure $\leb\otimes\nualpha\otimes\pr_{\bsQ}$ and
  with $\bsQ$ a random element in $\spaceD$ whose distribution is given in \Cref{eq:loi-Q}. 
\end{corollary}

\begin{proof}
  We only prove the statement on the extremal index for completeness, slightly simplifying the
  argument in the proof of \cite[Theorem~2.1]{debicki:hashorva:2020}. Without loss of generality, we
  can assume that $\tailmeasure_C=0$ since $\bseta=\bseta_D\vee\bseta_C$ and $\bseta_C$ has zero
  extremal index. Then,
  \begin{align*}
    -\log   \pr\left(\sup_{0 \leq s \leq T} \eta_s \leq T^{1/\alpha} x\right) 
    & = \tailmeasure(\{\bsy\in\spaceD:\bsy_{0,T}^*>T^{1/\alpha}x\}) \\
    & = \candidate x^{-\alpha} T^{-1} \int_{-\infty}^\infty \esp \left[ \sup_{-s\leq u \leq T-s} Q_u^\alpha \right] \rmd s 
      = x^{-\alpha} \tildetailmeasurestar_T(\shinvmap)  \; ,
  \end{align*}
  with $\tildetailmeasurestar_T$ defined in \Cref{eq:def-tildemeasurestar} and
  $\shinvmap(\bsy) = \ind{\bsy^*>1}$. Since we have obtained in the proof of
  \Cref{lem:tailtozero-implies-convergencetoclustermeasure} that
  $\tildetailmeasurestar_T\convvague\tailmeasurestar$ and $\tailmeasurestar(\shinvmap) = \candidate$, we
  obtain
  \begin{align*}
    \lim_{T\to\infty} 
    -\log   \pr\left(\sup_{0 \leq s \leq T} \eta_s \leq T^{1/\alpha} x\right) = \candidate x^{-\alpha} \; .
  \end{align*}
  This proves that $\candidate$ is the true extremal index of the process $\bseta$.
\end{proof}

The representation \cref{eq:M3representation} is a probabilistic interpretation of
\cref{eq:dissipativerepresentation} in the case of \nonnegative\ max-stable processes. The class of
processes generated by a conservative flow is much wider and much more complex, see for instance
\cite{dombry:kabluchko:2016}, \cite{kabluchko:schlather:2010}, \cite{wang:roy:stoev:2013}.

\begin{example}
  \label{xmpl:subgaussian-maxstable}
  Let $\bsW$ be a Gaussian process with continuous paths, stationary increments and such that
  $\pr(W_0=0)=1$. Let $\sigma_t^2 = \var(W_t)$. For $\alpha>0$ define the process $\bsZ$ by
  $Z_t=\rme^{W_t-\alpha\sigma_t^2/2}$, $t\in\Rset$. Let $\tailmeasure$ be the tail measure on
  $\spaceD$ defined by $\tailmeasure = \int_0^\infty \esp[\delta_{u\bsZ}]\alpha u^{-\alpha-1}\rmd u$
  and $\bseta$ the associated max-stable process. Then $\bseta$ is stationary and has almost surely
  continuous paths. Cf. \cite{kabluchko:schlather:dehaan:2009}. Assume that
  \begin{align*}
    \pr\left( \lim_{|t|\to\infty} W_t-\alpha\sigma_t^2/2 = -\infty \right) = 1 \; .
  \end{align*}
  Then the extremal index $\theta$ of $\bseta$  is positive and given by
  \begin{align*}
    \theta & = \candidate = \esp \left[ \frac{\sup_{t\in\Rset} \rme^{\alpha W_t-\alpha^2\sigma^2_t/2}}
             {\int_{-\infty}^\infty \rme^{\alpha W_t-\alpha^2\sigma^2_t/2} \rmd t}\right] 
            = \esp \left[ \frac1 {\int_{-\infty}^\infty \ind{\alpha W_t-\alpha^2\sigma^2_t/2>-E}  \rmd t}\right] \; ,
  \end{align*}
  with $E$ a random variable with a standard exponential distribution, independent of $\bsW$.
  The first expression for $\theta$ was obtained by  \cite{debicki:hashorva:2020}. 

  Since $\bsW$ has continuous paths, we can also apply the results of \Cref{sec:anchoring-maps} and
  \Cref{xmpl:infargmax-general}. Let $\widehat{\bsW}$ be the process defined by
  $\widehat{W}_t=W_t-\alpha\sigma_t^2/2$ and $\mci_0$ be the infargmax functional. By
  homogeneity of $\mci_0$ and the monotonicity of the exponential,
  $\mci_0(\bsY) = \mci_0(\widehat{\bsW})$ and by \Cref{eq:density-infargmax}, $\mci_0(\widehat{\bsW})$
  admits a continuous density $q_0$ \wrt\ Lebesgue's measure on $\Rset$ given by
  \begin{align*}
    q_0(t) = \esp \left[ \frac{\rme^{\alpha \widehat{W}_{\mci_0(\widehat\bsW)-t}}}
    {\int_{-\infty}^\infty \rme^{\alpha\widehat{W}_s}\rmd t } \right] \; , \ \ t\in\Rset \; .
  \end{align*}
  In this case, condition \Cref{eq:condition-continuite} holds by continuity and the representation
  \Cref{eq:candidate-conditional} hold.

  Consider now the first exceedance map $\mci_1$. Then the continuity of the sample paths imply that
  $\pr(Y|\bsQ_{\mci_1(Y\bsQ)}|>1)=0$, thus \Cref{eq:condition-continuite} does not hold.
\end{example}

\subsection{Sum-stable processes}
\label{sec:sumstable}
Let $\bszeta$ be an $\alpha$-stable process ($0<\alpha<2$). This means that its finite dimensional
distributions are $\alpha$-stable, hence regularly varying with tail index $\alpha$ and
characterized by their exponent measures. Consequently, the distribution of $\bszeta$ is
characterized by its tail measure (in the sense of \cite{owada:samorodnitsky:2012}) on $\Rset^\Rset$
endowed with the product topology.

If we assume further that $\bszeta$ is regularly varying on $\spaceD$ then its tail measure is a
tail measure on $\spaceD(\Rset,\Rset)$ in the sense of \Cref{def:tailmeasureonD}.  If $\bszeta$ is
stationary, by \Cref{theo:Y-determines-nu}, it admits a spectral process $\bsZ$ which satisfies
$\esp[ \sup_{a \leq s \leq b}|Z(s)|^\alpha]<\infty$ for all $a \leq b$. Then,
$\sum_{i=1}^\infty \delta_{P_i\bsZ^{(i)}}$ is a PPP on $\spaceD$ with mean measure
$\tailmeasure$. If $\alpha<1$, the series $\sum_{i=1}^\infty P_i$ is summable, thus the series
\begin{align}
  \label{eq:series-representation}
   \sum_{i=1}^\infty P_i \bsZ^{(i)} \; , 
\end{align}
is almost surely locally uniformly convergent.  If $\alpha\in[1,2)$, we will assume furthermore that
$\tailmeasure$, or equivalently $\bsZ$, is symmetric (\ie\ $\tailmeasure(A) = \tailmeasure(-A)$ for
all Borel measurable subset $A$ of~$\spaceD$.  This allows to avoid the issue of centering and to
use the maximal inequality recalled in \Cref{sec:maximal}. Then the series in
\Cref{eq:series-representation} is pointwise almost surely convergent, and defines a symmetric
$\alpha$-stable process. In both cases, the series in \Cref{eq:series-representation} defines an
$\alpha$-stable process with the same finite-dimensional distributions as $\bszeta$. Cf.
\cite[Theorem~1.4.2]{samorodnitsky:taqqu:1994}.

\begin{theorem}
  \label{theo:sum-stable-regvarD}
  Let $\alpha\in(0,2)$, $\tailmeasure$ be a \shiftinvariant\ tail measure on $\spaceD$ with tail
  index $\alpha$, $\bsY$ be the associated tail process and $\bsZ$ be a spectral process for
  $\tailmeasure$. If $\alpha\in[1,2)$, assume furthermore that $\tailmeasure$ (or equivalently
  $\bsY$ or $\bsZ$) is symmetric. Let $\bszeta$ be a $\spaceD$-valued process. The following
  statements are equivalent:
  \begin{enumerate}[(i)]
  \item \label{item:rvimpliesseries} $\bszeta$ is a stationary $\alpha$-stable process, regularly
    varying in $\spaceD(\Rset)$, with tail measure $\tailmeasure$;
  \item \label{item:seriesimpliesrv} $\bszeta$ admits the representation \Cref{eq:series-representation} with $\bsZ^{(i)}$,
    $i\geq1$, \iid\ copies of $\bsZ$.
  \end{enumerate}
\end{theorem}

\begin{proof}
  We have already proved the implication \Cref{item:rvimpliesseries} $\implies$
  \Cref{item:seriesimpliesrv}. We prove the converse. We must prove that for all $\epsilon>0$,
  $a>0$ and bounded Lipschitz continuous maps $\map$ on $\spaceD$ such that $\map(\bsy)=0$ if
  $\bsy_{-a,a}^*\leq \epsilon$,
  \begin{align}
    \lim_{x\to\infty}  \frac{\esp[\map(x^{-1} \bszeta)]}{\pr(|\zeta_0|>x)} =  \tailmeasure(\map) \; . 
  \end{align}
  As in the proof of \Cref{theo:maxstable}, let $\bsW$ be a process whose distribution is given by
  \begin{align*}
    \esp[\map(\bsW)] = \frac{\esp [ \map((\bsZ_{a,b}^*)^{-1}\bsZ)(\bsZ_{a,b}^*)^\alpha]}{\esp[(\bsZ_{a,b}^*)^\alpha]} \; .
  \end{align*}
  Let $\bsW^{(i)}$ be \iid\ copies of $\bsW$.  Then $\bszeta$ has the same distribution as
  \begin{align*}
    c_\alpha^{1/\alpha} \sum_{i=1}^\infty P_i \bsW^{(i)} \; , 
  \end{align*} 
  with $c_\alpha=\esp[\sup_{a\leq s \leq b} |Z(s)|^\alpha]$.  Since $P_1$ has a Fr\'echet distribution and $\sup_{-a\leq t\leq a} |W_t| = 1$,
  we have
  \begin{align*}
    \lim_{x\to\infty}    x^\alpha  \esp[\map(c_\alpha^{1/\alpha}x^{-1}P_1\bsW)] 
    & = c_\alpha \int_0^\infty \esp[\map(u\bsW)] \alpha u^{-\alpha-1} \rmd u \\
    & = \int_0^\infty \esp[\map(u\bsZ)] \alpha u^{-\alpha-1} \rmd u = \tailmeasure(\map) \; .
  \end{align*}
  Write $\bszeta^{(1)} = c_\alpha^{1/\alpha} P_1 \bsW^{(1)}$.
  Since $\map$ is Lipshitz-continuous \wrt\ $d_{J_1}$ with support separated from $\bszero$, and
  $\bsW_{-a,a}^*=1$ almost surely, we have, for every $\eta\in(0,\epsilon/2)$,
  \begin{align*}
    \left| \esp[\map(x^{-1}\bszeta)] - \esp[\map(x^{-1}\bszeta^{(1)})] \right| 
    & \leq \constant\ \eta \pr(P_1>a_T\epsilon/2) + \constant \pr(d_{J_1}(\bszeta,\bszeta^{(1)})>x\eta) \; .
  \end{align*}
  Since the $J_1$ metric is bounded by the uniform metric on any compact interval, we obtain, for
  any two functions $f,g\in\spaceD$, 
  \begin{align*}
    d_{J_1}(f,g)  \leq \sup_{-a\leq s \leq a} |f(s)-g(s)| + \rme^{-a} \; . 
  \end{align*}
  Thus, applying the bound \Cref{eq:borne-P2} in the proof of \Cref{lem:inegalite-maximale-02}, we
  obtain
  \begin{align*}
    \limsup_{x\to\infty} x^\alpha \pr(d_{J_1}(\bszeta,\bszeta^{(1)})>x\eta) 
    \leq \limsup_{x\to\infty} x^\alpha \pr \left(\sup_{a\leq s\leq b} \left|\sum_{j=2}^\infty P_j W^{(j)}(s)\right| +\rme^{-a} > x\eta \right) = 0 \; . 
  \end{align*}
  This proves that
  \begin{align*}
    \limsup_{x\to\infty} x^\alpha \left| \esp[\map(x^{-1}\bszeta)] - \esp[\map(x^{-1} \bszeta^{(1)})] \right|  \leq \constant\ \eta \; .
  \end{align*}
  Since $\eta $ is arbitrary, we have proved that \Cref{item:seriesimpliesrv} $\implies$ \Cref{item:rvimpliesseries}.
\end{proof}

If $\tailmeasure(\spaceD_0^c)=0$, we can apply \Cref{theo:ppconv-approx} and obtain the convergence
of the point process of clusters $N_T$ defined in \Cref{eq:def-ppcluster}.
\begin{theorem}
  \label{thep:convpp-mma}
  Let $\bszeta$ be $\spaceD$ valued $\alpha$-stable process, regularly varying in $\spaceD$ with
  tail measure $\tailmeasure$ such that $\tailmeasure(\spaceD_0^c)=0$. Then $\candidate>0$ and
  $\bszeta$ can be expressed as
  \begin{align}
    \label{eq:mma-sum-stable}
    \zeta_t = \sum_{i=1}^\infty P_i Q^{(i)}(t-T_i) \; , 
  \end{align}
  where $\sum_{i=1}^\infty \delta_{T_i,P_i,\bsQ^{(i)}}$ is a PPP on
  $\Rset\times(0,\infty)\times\spaceD_0$ with mean measure
  $\candidate\leb\otimes\nualpha\otimes\pr_{\bsQ}$. Furthermore, the process $\bsQ$
  satisfies~\Cref{eq:Q-local-finite} and the point process of clusters $N_T$ converges to a Poisson
  point process with mean measure $\leb\otimes\tailmeasurestar$ for all sequences $r_T$ such that
  $r_T/T\to0$.
\end{theorem}

\begin{proof}
  Define 
  \begin{align*}
    Q_t^{(m)} = Q_t \ind{|t|\leq m}  \; .
  \end{align*}
  Let $\bsQ^{(m,i)}$, $i\geq1$, be \iid\ copies of $\bsQ^{(m)}$ and  define $\bszeta^{(m)}$ by
  \begin{align*}
    \zeta_t^{(m)} & = \sum_{i=1}^\infty P_i Q^{(m,i)}_{t-T_i} \; .    
  \end{align*}
  Then $(\bszeta,\bszeta^{(m)})$ is stationary, $\bszeta^{(m)}$ is $m$-dependent and regularly
  varying in $\spaceD$. We prove that the condition \Cref{eq:condition-m-approx} of
  \Cref{theo:ppconv-approx} holds. Applying \Cref{eq:inegalite-maximale-01} of
  \Cref{lem:inegalite-maximale-02} (with $\bsZ$ expressed in terms of $\bsQ$ as in
  \Cref{coro:equivalence-nu-Q}), we obtain
  \begin{align*}
    \lim_{m\to\infty} \limsup_{T\to\infty}
    &   T\pr\left(\sup_{0 \leq s \leq 1} \norm{\zeta_s-\zeta_s^{(m)}}>a_T\epsilon\right) \\
    & \leq \constant\ \lim_{m\to\infty} \limsup_{T\to\infty}
      T^{-1}  \int_{-\infty}^\infty \esp[\sup_{0\leq s \leq T} |Q(t+s)|^\alpha\ind{|t+s|>m}] \rmd t \\
    & \leq \constant\ \lim_{m\to\infty} \int_{-\infty}^\infty \esp[\sup_{0\leq s \leq 1} |Q(t+s)|^\alpha\ind{|t+s|>m}] \rmd t = 0 \; .
  \end{align*}
  Thus \Cref{eq:condition-m-approx} holds.
\end{proof}

Thus we recover part of \cite[Theorem~2.2]{samorodnitsky:2004:maxima} which obtained the extremal
index of a general stable process and \cite[Corollary~5.3]{rootzen:1978} which obtained the extremal
index of moving average \wrt\ a L\'evy stable process.
\begin{corollary}
  \label{coro:extremal-index-max-stable}
  Let $\bszeta$ be $\spaceD$ valued $\alpha$-stable process, regularly varying in $\spaceD$ with
  tail measure~$\tailmeasure$.  Then $\bszeta$ admits an extremal index $\theta = \candidate$, \ie\
  \begin{align*}
    \lim_{T\to\infty} T\pr\left(\max_{0\leq s \leq T} |\zeta_s| \leq a_Tx\right) 
    =  \lim_{T\to\infty} T\pr\left(\max_{0\leq s \leq T} \eta_s \leq a_Tx\right) = \rme^{-\candidate x^{-\alpha}} \; .
  \end{align*}
\end{corollary}
\begin{proof}
  As noted in \Cref{eq:conservative-nulextremalindex}, if $\nu(\spaceD_0)=0$, then $\candidate=0$,
  hence $\theta=0$ by \Cref{lem:nullextremalindex}.  Assume that $\candidate>0$, which is equivalent
  to $\tailmeasure(\spaceD_0)>0$. If $\tailmeasure(\spaceD_0^c)=0$, then we can apply
  \Cref{thep:convpp-mma} to conclude. Otherwise, define
  $p=\esp[|\bsZ(0)|^\alpha\ind{\bsZ\in\spaceD_0}]$, $\bsZ_0=p^{-1/\alpha}\bsZ\ind{\bsZ\in\spaceD_0}$
  and $\bsZ_1=(1-p)^{-1/\alpha}\bsZ\ind{\bsZ\notin\spaceD_0}$. Let $\bsZ_0^{(i)}$, $\bsZ_1^{(i)}$,
  $i\geq1$ be \iid\ copies of $\bsZ_0$ and $\bsZ_1$. Define $\bszeta_0$ and $\bszeta_1$ by
  \Cref{eq:series-representation} with $\bsZ_0$ and $\bsZ_1$ respectively. By the previous part of
  the proof we have $T^{-1/\alpha} \sup_{0 \leq s \leq T} |\zeta_1(s)| \convprob 0$ and by
  \Cref{thep:convpp-mma}, we have
  \begin{align*}
    \pr \left( \sup_{0 \leq s \leq T} |\zeta(s)| \leq  T^{1/\alpha}x \right) \sim
    \pr \left( p\sup_{0 \leq s \leq T} |\zeta_0(s)| > T^{1/\alpha}x \right) \to \rme^{-p^{-1} \candidate_0x^{-\alpha}} \; ,
  \end{align*}
  with
  $\candidate_0 = \lim_{T\to\infty}\esp[\sup_{0\leq s \leq T}|Z_0(s)|^\alpha] =
  p\lim_{T\to\infty}\esp[\sup_{0\leq s \leq T}|Z(s)|^\alpha] = p\candidate$.
\end{proof}

\subsubsection*{Integral representations}
Stable processes can also be defined by their integral representations. For completeness, we briefly
rewrite (without proof) our results using these representations. Let $M$ be a $\alpha$-stable random
measure with independent increments and control measure $m$ on a measurable space
$(\Eset,\mce)$. For simplicity, following \cite{samorodnitsky:2004:maxima}, we assume that $M$
(hence $m$) is symmetric. This means that for every measurable function $f$
such that $\int_{\Eset}|f(x|^\alpha m(\rmd x)<\infty$,
\begin{align*}
  \log \esp[  \rme^{\rmi z M(f)}] = -C_\alpha |z|^\alpha \int_{\Eset} |f(x)|^\alpha m(\rmd x) \;  \; .
\end{align*}
with $C_\alpha = \int_0^\infty \sin(x) x^{-\alpha} \rmd x$.  Let $f:\Eset\times\Rset\to\Rset$ be a
measurable function, such that $t\mapsto f(x,t)$ is \cadlag\ for all $x\in\Eset$ and for all
$a\leq b$,
\begin{align*}
  \int_E  \sup_{a \leq s  \leq b} |f(x,s)|^\alpha m(\rmd x) < \infty \; .
\end{align*}
By \cite[Theorem~10.2.3]{samorodnitsky:taqqu:1994}, this is a necessary condition for local
boundedness of the stable process $\bszeta$ defined by
\begin{align*}
  \zeta_t = \int_{\Eset} f(x,t) M(\rmd x) \; ,  \ \ t\in\Rset \; .
\end{align*}
The tail measure $\tailmeasure$ is given by
\begin{align}
  \label{eq:tailmeasure-stable-general}
  \tailmeasure = \frac{1} {\int_{\Eset} |f(x,0)|^\alpha m(\rmd x)}
  \int_{\Eset} \int_{0}^\infty \delta_{uf(x,\cdot)} \alpha u^{-\alpha-1} \rmd  u \; m(\rmd x) \; .
\end{align}
See \cite[Section~3]{owada:samorodnitsky:2012}. If $\tailmeasure(\spaceD_0^c)=0$, then $\bsX$ admits
a mixed moving average representation. This means that there exist a measured space $(\Fset,\mcf,\mu)$
and a measurable function $g:\Fset\times\Rset\to\Rset$ such that
\begin{align}
  \label{eq:condition-f-mma}
  \int_{\Fset}  \int_{-\infty}^\infty \sup_{a \leq s \leq b} |g(w,t+s)|^\alpha \mu(\rmd w) \rmd t < \infty \; ,
\end{align}
for all $a\leq b$. The latter condition implies that $\lim_{|t|\to\infty} g(w,t)=0$ for $\mu$-almost
all $w$. Then $\bszeta$ can be defined as 
\begin{align*}
  \zeta_t = \int_{\Fset}\int_{-\infty}^\infty g(w,t-s) \Lambda (\rmd w\rmd s) \; .
\end{align*}
where $\Lambda$ is an $\alpha$-stable random measure with independent increments and control
measure $\mu\otimes\leb$ on $\Fset\times\Rset$.

Define $g^* = \sup_{t\in\Rset} |g(\cdot,t)|$. Then, the extremal index is given by
\begin{align*}
  \candidate
    = \frac{ \int_{\Eset} (g^*(x))^\alpha \mu(\rmd x)} {\int_{\Eset}\int_{-\infty}^\infty |g(x,t)|^\alpha \mu(\rmd x)\rmd t} \; .
\end{align*}
See \cite[Theorem~22(i)]{samorodnitsky:2004:maxima}. A mixed-moving average representation of the
tail measure is given by
\begin{align*}
  \tailmeasure
  & = \candidate \int_{-\infty}^\infty \int_0^\infty \esp\left[ \delta_{uh(W,\cdot-t)} \right] \alpha u^{-\alpha-1} \rmd u  \, \rmd t \; ,
\end{align*}
with $W$ a random variable whose distribution admits a density proportional to $(g^*)^\alpha$ \wrt\
Lebesgue's measure and $h(w,t) = (g^*)^{-1}(w)g(w,t)$.

\subsection{Functional weighted sums}
\label{sec:functional-ma}
Let $\{f_k\in\Zset\}$ be a sequence of random element in $\spaceD(\Rset,\Rset)$ and let
$\{V_k,k\in\Zset\}$ be a sequence of \iid\ random variables, regularly varying with index $\alpha$
and extremal skewness~$p_Z$, independent of the previous sequence. We (formally) define the process
$\bsX$ by
\begin{align}
  \label{eq:functional-linear}
  X_t = \sum_{k\in\Zset} f_k(t) V_k \; , \ \ t \in\Rset \; .
\end{align}
If $\alpha\leq1$ or $\alpha>1$ and $\esp[V_0]=0$ and if there exists $\beta \in (0,\min(\alpha,1))$
such that for every $t\in\Rset$,
\begin{align*}
  \pr \left(  \sum_{k\in\Zset} |f_k(t)|^\beta < \infty \right) = 1 \; , 
\end{align*}
then the series $\sum_{k\in\Zset} f_k(t) V_k$ is almost surely convergent; see
\cite[Section~3]{hult:samorodnitsky:2008} and \cite[Lemma~A.3]{mikosch:samorodnitsky:2000}. If furthermore
\begin{align}
  \label{eq:moment-beta}
  \sum_{k\in\Zset} \esp[ |f_k(t)|^\beta] < \infty  \; , 
\end{align}
then $\bsX$ is finite dimensional regularly varying.  Before stating and proving rigorous results,
we give some heuristics. The main argument to obtain the extremal behavior of the process $\bsX$ is
the so-called ``single large jump principle'': $X_t$ is large \ifft\ there is a single jump $V_k$
which is extremely large and it is chosen ``at random'' among the sequence $\{V_j,j\in\Zset\}$. Thus
we expect that for a continuous map $\map$ on $\spaceD$ endowed with the $J_1$ topology with support
separated from~$\bszero$,
\begin{align}
  \frac{\esp[\map(x^{-1}\bsX)]}{\pr(|V_0|>x)}
  & \sim \sum_{k\in\Zset} \frac{\esp[\map(x^{-1}f_k V_k)] }{\pr(|V_0|>x)} \nonumber \\
  & \to \sum_{k\in\Zset} \int_0^\infty \esp[\map(u\epsilon_0f_k)] \alpha u^{-\alpha-1} \rmd u \; ,
    \label{eq:convtotailmeasure-H}
\end{align}
with $\epsilon_0$ independent of $\{f_k\}$ and $\pr(\epsilon_0=1)=1-\pr(\epsilon_0=-1)=p_Z$.  Since
the tail measure is normalized by the condition $\tailmeasure(\{\bsy\in\spaceD:\norm{y_0}>1\})=1$,
\cref{eq:convtotailmeasure-H} yields
\begin{align*}
  \tailmeasure(\map) & = \frac{ \sum_{k\in\Zset} \int_0^\infty \esp[\map(u\epsilon_0f_k)] \alpha u^{-\alpha-1} \rmd u}
                    {\sum_{k\in\Zset} \esp[|f_k(0)|^\alpha] } \\
                  & = \frac{ \sum_{k\in\Zset} \int_0^\infty \esp[|f_k(0)|^\alpha \map(u|f_k(0)|^{-1}\epsilon_0f_k)] \alpha u^{-\alpha-1} \rmd u}
                    {\sum_{k\in\Zset} \esp[|f_k(0)|^\alpha] } \; .
\end{align*}
Thus we expect the tail measure to be given by
\begin{align}
  \label{eq:def-tailmeasure-functional-weighted-sum}
  \tailmeasure = \int_0^\infty \esp[\delta_{u|f_N(0)|^{-1}\epsilon_0f_N}]  \alpha u^{-\alpha-1} \rmd u \; .
\end{align}
with an integer valued random variable $N$, independent of $\epsilon_0$ and such that
\begin{align*}
  \esp[\map(f_j,j\in\Zset) \ind{ N=k}] = \frac{  \esp[\map(f_j,j\in\Zset) |f_k(0)|^\alpha]}{\sum_{j\in\Zset} \esp[|f_j(0)|^\alpha]} \; ,  \ \ k \in \Zset \; .
\end{align*}
In particular, 
\begin{align*}
  \pr(N=k) = \frac{\esp[|f_k(0)|^\alpha]}{\sum_{j\in\Zset} \esp[|f_j(0)|^\alpha]} \; ,  \ \ k \in \Zset \; .
\end{align*}
We derive from \cref{eq:def-tailmeasure-functional-weighted-sum} the tail process
\begin{align}
   \label{eq:tailprocess-with-N}
  Y_t = \frac{f_N(t)}{|f_N(0)|}Y\epsilon_0 \; , \ \ t \in \Rset \; , 
\end{align}
with $Y$ a random variable with a Pareto distribution with tail index $\alpha>0$, independent of
$\epsilon_0$, $N$ and $\{f_k,k\in\Zset\}$. In discrete time, the tail process $\bsY$ of a time
series of the form \Cref{eq:functional-linear} was obtained in
\cite[Section~8]{meinguet:segers:2010}:

In order to prove rigorously \Cref{eq:def-tailmeasure-functional-weighted-sum} we must make
assumptions that allow the use of a truncation argument.

\begin{proposition}
  \label{prop:functionalconvergencewithadhoctruncationcondition}
  Let $\{V_k,k\in\Zset\}$ be a sequence of \iid\ random variables, regularly varying with index
  $\alpha>0$ and extremal skewness~$p_Z$ and such that $\esp[V_0]=0$ if $\alpha>1$.  Assume that
  $\{f_k,k\in\Zset\}$ is a sequence of random functions, independent of $\{V_k,k\in\Zset\}$, such
  that $\pr(f_k\in\spaceD)=1$, $f_k$ is stochastically continuous for all $k\in\Zset$ and there
  exists $\beta\in(0,\min(\alpha,1))$ such that \Cref{eq:moment-beta} holds.  Assume that the
  process $\bsX$ defined in \Cref{eq:functional-linear} is stochastically continuous and that for
  all $a<b$ and $x>0$,
  \begin{gather}
    \label{eq:summability-supf_k}
    \sum_{k\in\Zset} \esp \left[ \sup_{a \leq s  \leq b} |f_k(s)|^\alpha \right]  < \infty \; , \\
    \label{eq:truncation-argument-functional-weighted-sum}
    \lim_{m\to\infty} \limsup_{T\to\infty} \left( \sup_{a\leq s  \leq b} \left| \sum_{|k|>m}f_k(s)V_k\right| > a_T x \right)  = 0 \; ,
  \end{gather}
  with $a_T$ such that $\lim_{T\to\infty} T\pr(|X_0|>a_T)=1$.  Then $\pr(\bsX\in\spaceD)=1$, $\bsX$
  is regularly varying in $\spaceD$ with tail measure given by
  \Cref{eq:def-tailmeasure-functional-weighted-sum}.

  Assume moreover that $\{\shift^tf_k,k\in\Zset\} \eqdistr \{f_k,k\in\Zset\}$ for all $t\in\Rset$
  and $\pr(f_k\in\spaceD_0)=1$ for all $k\in\Zset$ and let $N_T$ be the point process of clusters of
  $\bsX$ as defined in \Cref{eq:def-ppcluster}. Then  $\bsX$ is stationary, $N_T\convweak N$, where $N$ is a Poisson point
  process on $\Rset\times\spaceDtilde_0$ with mean measure $\leb\otimes\tailmeasurestar$ and
  $\tailmeasurestar$ defined by
  \begin{align*}
    \tailmeasurestar = \int_0^\infty \esp \left[ \frac{\delta_{uf_N}}
    {\int_{-\infty}^\infty |f_N(t)|^\alpha\rmd t} \right] \alpha u ^{-\alpha-1} \rmd u  \; .
  \end{align*}
  The extremal index $\candidate$ of $\bsX$ exists and is given by
  \begin{align*}
    \candidate & = \esp \left[ \frac{\sup_{t\in\Rset} |f_N(t)|^\alpha}{\int_{-\infty}^\infty |f_N(t)|^\alpha \rmd t } \right] \; .
  \end{align*}
\end{proposition}

\begin{proof}
  For a positive integer $m$, define
  \begin{align}
    \label{eq:functional-linear-truncated}
    X_t^{(m)} = \sum_{|k|\leq m} f_k(t) V_k \;  , \ \ \hat{X}_t^{(m)} = \sum_{|k|> m} f_k(t) V_k \; , \ \ t \in\Rset \; .
  \end{align}
  Then $\bsX=\bsX^{(m)}+\hat\bsX^{(m)}$. The process $\bsX^{(m)}$ is almost surely \cadlag\ as a
  finite sum of almost surely \cadlag\ functions and stochastically continuous.  Condition
  \Cref{eq:truncation-argument-functional-weighted-sum} implies that $\bsX^{(m)}$ converges in
  probability (hence almost surely along a subsequence) locally uniformly to $\bsX$, thus $\bsX$ is
  also almost surely \cadlag. Define $c_m$ as
  \begin{align*}
    c_m = \lim_{T\to\infty} T \pr(|X_0^{(m)}|>a_T) = \frac{\sum_{|k\leq m}\esp[|f_k(0)|^\alpha]}{\sum_{k\in\Zset} \esp[|f_k(0)|^\alpha]} \; .
  \end{align*}
  We will prove that for all $a<b\in\Rset$, $\epsilon>0$ and Lipschitz continuous maps $\map$ \wrt\
  $d_{J_1}$ (defined in (\ref{eq:def-local-J1})) such that $\map(\bsy)=0$ if
  $\bsy_{a,b}^*\leq\epsilon$, it holds that
  \begin{align}
    \label{eq:convergence-num}
    &    \lim_{T\to\infty}   T\esp[\map(a_T^{-1}\bsX^{(m)})] = c_m \sum_{|k|\leq m} \int_0^\infty \esp[\map(u\epsilon_0f_k)] \alpha u^{-\alpha-1} \rmd u \; , \\
    &    \lim_{m\to\infty} \limsup_{T\to\infty} T\esp[|\map(a_T^{-1} \bsX)-\map(a_T^{-1} \bsX^{(m)})|] = 0 \; .
      \label{eq:triangular-argument-map-H}
  \end{align}
  In view of \Cref{eq:summability-supf_k}, the series
  $ \sum_{k\in\Zset} \int_0^\infty \esp[|\map(u\epsilon_0f_k)|] \alpha u^{-\alpha-1} \rmd u$ is summable
  and
  \begin{align*}
    \lim_{m\to\infty}    \sum_{|k|\leq m} \int_0^\infty \esp[\map(u\epsilon_0f_k)] \alpha u^{-\alpha-1} \rmd u = 
    \sum_{k\in\Zset} \int_0^\infty \esp[\map(u\epsilon_0f_k)] \alpha u^{-\alpha-1} \rmd u \; .
  \end{align*}
  Thus \Cref{eq:convergence-num} and \Cref{eq:triangular-argument-map-H} imply the convergence
  \Cref{eq:convtotailmeasure-H}.  Because of the Lipschitz property of $\map$, we have, for every
  $\eta\in(0,\epsilon/2)$,
  \begin{align*}
    \esp[|\map(a_T^{-1} \bsX)-\map(a_T^{-1} \bsX^{(m)})|] \leq \constant\ \eta \pr(\bsX_{a,b}^*>a_T\epsilon/2) 
    + \constant\ \pr(d_{J_1}(\bsX,\bsX^{(m)})>a_T\eta) \; .
  \end{align*}
  By definition of the metric $d_{J_1}$ and since the $J_1$ metric on an interval is bounded by the
  uniform metric, we have, for all functions $f,g\in\spaceD$ and $t>0$,
  \begin{align*}
    d_{J_1}(f,g)  \leq \sup_{-t\leq s\leq t} |f(s)-g(s)| + \rme^{-t} \; . 
  \end{align*}
  Therefore,  
  \begin{multline*}
    \limsup_{T\to\infty} T \esp[|\map(a_T^{-1} \bsX) - \map(a_T^{-1} \bsX^{(m)})|] \\
    \leq \constant\ \eta + \constant\ \limsup_{T\to\infty} T \pr\left( \sup_{-t\leq s \leq t} |X_s -
      X_s^{(m)}| + \rme^{-t} > a_T\eta\right) \; .
  \end{multline*}
  Thus \Cref{eq:truncation-argument-functional-weighted-sum} implies that
  \Cref{eq:triangular-argument-map-H} holds and there only remains to prove
  \Cref{eq:convergence-num}.

  \begin{enumerate}[$\bullet$,wide=0pt]
  \item The regular variation of $V_0$ implies that for all $\eta>0$ and bounded continuous
    functions $g$ on $\Rset$ such that $g(x)=0$ if $|x|\leq \eta$,
    \begin{align*}
      \lim_{x\to\infty} \frac{ \esp[g(x^{-1}V_0)]}{\pr(|V_0|>x)} = \int_{0}^\infty \esp[g(u\epsilon_0)] \alpha u^{-\alpha-1} \rmd u \; ,
    \end{align*}
    with $\epsilon_0$ as above.
  \item For every $f\in\spaceD$, the map $u\mapsto uf$ is continuous on $\Rset$. Thus, for every map
    $\map$ on $\spaceD$, continuous  \wrt\ the $J_1$ topology, the 
      map $\map:u\mapsto \map(uf)$ is continuous on $\Rset$.
    \item If $\map$ is moreover bounded and there exist $a<b$ and $\eta>0$ such that $\map(\bsy)=0$
      if~$\bsy^*_{a,b}\leq\eta$, then for every function $f\in\spaceD$ (which is necessarily locally
      bounded), the map $u\mapsto \map(uf)$ is bounded, continuous with support separated from zero: if
      $|u|\leq \epsilon (f_{a,b}^*)^{-1}$, then $\map(uf)=0$. Consequently,
    \begin{align*}
      \lim_{x\to\infty} \frac{\esp[\map(x^{-1}fZ)]}{\pr(|V_0|>x)} = \int_{0}^\infty \esp[\map(uf\epsilon_0)] \alpha u^{-\alpha-1} \rmd u \; ,
    \end{align*}
  \item If $f$ is a $\spaceD$-valued random map, independent of $Z$, then, conditionally on $f$, we have almost
    surely,
    \begin{align*}
      \lim_{x\to\infty}      \frac{\esp[\map(x^{-1}fZ)\mid f]}{\pr(|V_0|>x)} 
      & = \int_0^\infty \esp[\map(uf\epsilon_0)\mid f] \alpha u^{-\alpha-1} \rmd u \; .
    \end{align*}
    Since $\map(x^{-1}fZ)=0$ if $f_{a,b}^*Z\leq x\epsilon$, By Potter's bound
    (cf. \cite[Proposition~1.4.2]{kulik:soulier:2020}), we have, for $x\geq1$, 
    \begin{align*}
      \frac{\esp[\map(x^{-1}fV_0)\mid f]}{\pr(|V_0|>x)} \leq \constant \frac{\pr(f_{a,b}^*|V_0|>x\epsilon)}{\pr(|V_0|>x)}
      \leq \constant (f_{a,b}^*\vee1)^{\alpha+\epsilon} \; .
    \end{align*}
    Thus, by the dominated convergence theorem,  if $\esp[(f_{a,b}^*)^{\alpha+\epsilon}]<\infty$, we obtain
    \begin{align*}
      \lim_{x\to\infty}      \frac{\esp[\map(x^{-1}fZ)]}{\pr(|V_0|>x)} 
      & = \int_0^\infty \esp[\map(uf\epsilon_0)] \alpha u^{-\alpha-1} \rmd u \; .
    \end{align*}
  \item Consider now \iid\ random variables $V_1,\dots,V_k$. If $g:\Rset^k\to\Rset$ is continuous and
    bounded with support separated from zero, regular variation yields (cf. \cite[Proposition~2.1.1]{kulik:soulier:2020})
    \begin{align}
      \label{eq:mutlrv-indep}
      \lim_{x\to\infty} \frac{\esp[g(x^{-1}(V_1,\dots,V_k))}{\pr(|V_0|>x)} 
      =  \sum_{i=1}^k \int_0^\infty \esp[g_i(u\epsilon_0)] \alpha u^{-\alpha-1} \rmd u \; ,
    \end{align}
    with $g_i(u) = g(0,\dots,u,\dots,0)$ with the only nonzero component in the $i$-th
    position. Each function $g_i$ is bounded, continuous with support separated from zero so each
    integral in \Cref{eq:mutlrv-indep} is well defined and finite.
  \item Since the functions $f_i$ have no common discontinuities, the map
    $(u_1,\dots,u_k)\mapsto \sum_{i=1}^k u_i f_i$ is continuous \wrt\ the $J_1$ topology. Thus,
    defining $\bsX = \sum_{i=1}^k f_i V_i$, we have, for a bounded continuous (\wrt\ the $J_1$
    topology) map $\map$ and $a<b$, $\epsilon>0$ such that $\map(\bsy)=0$ if $\bsy_{a,b}^*\leq\epsilon$,
    applying \Cref{eq:mutlrv-indep} with $g(u_1,\dots,u_k) = \map(u_1f_1+\cdots+u_kf_k)$, we obtain by
    the same arguments as in the case $k=1$,
    \begin{align*}
      \lim_{x\to\infty} \frac{\esp[\map(x^{-1}(f_1V_1+\cdots+f_kV_k))}{\pr(|V_0|>x)} 
      = \sum_{i=1}^k  \int_0^\infty\esp[\map(f_is\epsilon_0)] \alpha s^{-\alpha-1} \rmd s \; .
    \end{align*}
    Thus \Cref{eq:convergence-num} holds and this proves that $\bsX$ is regularly varying in
    $\spaceD$ with tail measure given by \Cref{eq:def-tailmeasure-functional-weighted-sum}. 
  \end{enumerate}
  The congergence of the point process of clusters and the expression of the extremal index follow
  from \Cref{theo:ppconv-approx}. 
\end{proof}
Obtaining the bound \Cref{eq:truncation-argument-functional-weighted-sum} may be a hard task. We
will pursue the investigation on two examples.

\subsubsection{Functional moving average}
We now consider  the case 
\begin{align*}
  f_k(t) = f(t-T_k) \; , \ \ k \in\Zset \; ,  \ \ t \in \Rset 
\end{align*}
where $\{T_k,k \in\Zset\}$ are the points of a unit rate homogeneous Poisson point process on
$\Rset$ and $f\in\spaceD_0$ is a deterministic function such that
\begin{align*}
  \int_{-\infty}^\infty |f(t)|^\beta \rmd t < \infty \; . 
\end{align*}
with $\beta\in(0,\min(\alpha,1))$. Since $f$ is bounded, this implies that
$ \int_{-\infty}^\infty |f(t)|^q \rmd t < \infty$ for all $q\geq\beta$. The tail process $\bsY$ is
given by
\begin{align*}
  Y_t = \frac{f(t-T)}{|f(T)|}  Y\epsilon_0
\end{align*}
with $T$ a random variable with density $\lpnorm[\alpha]{f}[-\alpha] |\tilde{f}|^\alpha$ \wrt\
Lebesgue's measure on~$\Rset$, with $\tilde{f}(t)=f(-t)$. The condition
\Cref{eq:truncation-argument-functional-weighted-sum} becomes
\begin{align}
  \label{eq:truncation-argument-functional-weighted-sum-deterministic-shape}
  \lim_{m\to\infty} \limsup_{T\to\infty}    T \pr \left( \sup_{a\leq s  \leq b} \left| \sum_{|k|>m} f(s-T_k) V_k \right| > a_T x\right) = 0 \; , 
\end{align}
for all $x>0$.  Instead of
\Cref{eq:truncation-argument-functional-weighted-sum-deterministic-shape}, a different truncation
may used. If for all $x>0$, 
\begin{align}
  \label{eq:truncation-argument-functional-weighted-sum-deterministic-shape-mdep}
  \lim_{m\to\infty} \limsup_{T\to\infty}    T \pr \left( \sup_{a\leq s  \leq b}
  \left| \sum_{k\in\Zset} f(s-T_k)\ind{|s-T_k|>m} V_k \right| > a_T x\right) = 0 \; , 
\end{align}
then the process $\bsX$  can be approximated by the sequences of processes $\tilde{\bsX}^{(m)}$ defined by
\begin{align*}
  \tilde{X}^{(m)}_t = \sum_{|k| \in \Zset} f(s-T_k)\ind{|s-T_k|\leq m} V_k \; .
\end{align*}
The process $\tilde{\bsX}^{(m)}$ is $m$-dependent by the independent increment property of the Poisson
process. If either \Cref{eq:truncation-argument-functional-weighted-sum-deterministic-shape} or
\Cref{eq:truncation-argument-functional-weighted-sum-deterministic-shape-mdep} holds, then $\bsX$ is
regularly varying in $\spaceD$ and its tail measure is given by
\begin{align*}
  \tailmeasure 
  & = \lpnorm[\alpha]{f}[-\alpha] \int_{-\infty}^\infty \int_0^\infty \esp[\delta_{\epsilon_0u \shift^{-t} f}] \rmd t  
    \alpha u^{-\alpha-1} \rmd u \; .
\end{align*}
By \Cref{coro:extremalindex}, 
the extremal index is given by
\begin{align*}
  \candidate = \esp\left[ \frac{\sup_{t\in\Rset}|f(t-T)|^\alpha}  {\int_{-\infty}^\infty f(t-T)|^\alpha \rmd t } \right] 
  = \frac{\sup_{t\in\Rset}|f(t)|^\alpha}{\int_{-\infty}^\infty |f(t)|^\alpha \rmd t} \;. 
\end{align*}
Proving \Cref{eq:truncation-argument-functional-weighted-sum-deterministic-shape-mdep} is easy in
the case $\alpha<1$. Indeed, by   \cite[Theorem~3.1]{hult:samorodnitsky:2008}, we obtain 
\begin{align*}
  T \pr & \left( \sup_{a\leq s  \leq b} \left| \sum_{k\in\Zset} f(s-T_k) V_k \right| > a_T x\right)  \\
        & \leq   T \pr  \left(\sum_{k\in\Zset} \sup_{a\leq s  \leq b} |f(s-T_k)| |V_k| > a_T x\right)  \\
        & \to x^{-\alpha} \sum_{k\in\Zset} \esp\left[ \sup_{a\leq s  \leq b} |f(s-T_k)|^\alpha \right]  
          = x^{-\alpha} \int_{-\infty}^\infty  \sup_{a+t\leq s  \leq b+t} |f(s)|^\alpha  \rmd t \; .
\end{align*}
Thus both methods of truncation are suitable. We leave the case $\alpha\geq1$ for future work.

\subsubsection{Shot noise process}
\label{sec:shot-noise}
Let $\{T_k,V_k,\eta_k,k \in\Zset\}$ are the points of a Poisson point process on
$\Rset\times\Rset\times(0,\infty)$ with mean measure $\leb\otimes \pr_Z\otimes\pr_\eta$ where
$\pr_Z$ and $\pr_\eta$ denote the distribution of $V_0$ and $\eta_0$, respectively. We assume that
$\esp[\eta_0]<\infty$. Consider the model defined in \Cref{eq:functional-linear} with
\begin{align*}
  f_k = \1{[T_k,T_k+\eta_k)} \; , \ \  k\in\Zset \; , 
\end{align*}
that is 
\begin{align*}
  X_t = \sum_{j\in\Zset} V_j \ind{T_j\leq t < T_j+\eta_j}  \; .
\end{align*}
This process is also know as the infinite source Poisson process; see
\cite{roueff:samorodnitsky:soulier:2012}. The number of non-zero terms in the sum is almost surely
finite with a Poisson distribution with mean $\esp[\eta_0]$. Thus the assumption $\esp[V_0]=0$ is
not needed in the case $\alpha\geq1$, the sample paths are piecewise constant and \cadlag\ and the
process is stationary.  Furthermore, by standard results on random sums of regularly varying random
variables (see \eg\ \cite[Corollary~3.2]{hult:samorodnitsky:2008}), we have
\begin{align}
\label{eq:tail-IPS}
  \pr(|X_0|>x) \sim \esp[\eta_0] \pr(|V_0|>x)  \; .
\end{align}

For $m>0$, define
\begin{align*}
  X_t^{(m)} = \sum_{j\in\Zset} V_j \ind{T_j\leq t < T_j+\eta_j\wedge m}  \; .
\end{align*}
Then
\begin{align*}
  X_t-X_t^{(m)} =  \sum_{j\in\Zset} V_j \ind{T_j+\eta_j \wedge m \leq t < T_j+\eta_j}  \; .
\end{align*}
Since  $\{(T_j,\eta_j),j\in\Zset\}$ are the point of a marked Poisson point process with
independent \iid\ marks, the process $\bsX-\bsX^{(m)}$ has the same distribution as the process
$\tilde{X}^{(m)}$ defined by 
\begin{align*}
  \tilde{X}_t^{(m)} = \sum_{j\in\Zset} V_j \ind{T_j \leq t < T_j+(\eta_j-m)_+}  \; .
\end{align*}
Thus, for $a_T$ such that $\lim_{T\to\infty} T\pr(|X_0|>a_T)=1$ and $a < b$, we have 
\begin{align*}
  \lim_{T\to\infty} T \pr \left( \sup_{a\leq s \leq b} |X_s-X_s^{(m)}| > a_Tx \right) 
 & = \lim_{T\to\infty} T \pr \left( \sup_{a \leq s \leq b} |\tilde{X}_s^{(m)}| > a_Tx \right)  \; .
\end{align*}
Forgetting the truncation for the moment, we have 
\begin{align*}
  \pr   \left( \sup_{a\leq s \leq b} |X_s| > x \right)  
  & \leq \pr   \left(  \sup_{a\leq s \leq b} \sum_{k\in\Zset} |V_k| \ind{T_k \leq s < T_k+\eta_k} > x \right)  \; .
\end{align*}
For $s\in\Rset$, let $A_s$ be the subset of $\Rset\times\Rset_+$ defined by
$A_s=\{(t,u)\in\Rset\times\Rset_+:t \leq s < t+u\}$. Then
\begin{align*}
  \sum_{k\in\Zset} |V_k| \ind{T_k \leq s < T_k+\eta_k} = \sum_{k\in\Zset} |V_k| \1{A_s}(T_k,\eta_k) 
\end{align*}
Let $\{\tilde{T}_k,k\in\Zset\}$ be the reordering of the points $\{T_k,T_k+\eta_k,k\in\Zset\}$ with
the usual convention $\tilde{T}_0 \leq 0 < \tilde{T}_1$. Since $\{T_k\}$ is a homogoneneous Poisson
point process, the mean measure of the point process $\tilde{N}$ with points $\{\tilde{T}_k\}$ is
$2\mathrm{Leb}$. The map $s\mapsto A_s$ is piecewise constant and the changes happen at the points
$\tilde{T}_k$ (with one point added or removed). Thus
\begin{align*}
  \sup_{a\leq s \leq b} \sum_{k\in\Zset} |V_k| \ind{T_k \leq s < T_k+\eta_k}
  & \leq \sum_{j\in\Zset} \ind{a \leq \tilde{T}_j\leq b}  \sum_{k\in\Zset} \1{A_{\tilde{T}_j}}(T_k,\eta_k) \; |V_k| \\
  & = \sum_{k\in\Zset} \left(\sum_{j\in\Zset} \ind{a \leq \tilde{T}_j\leq b}  \1{A_{\tilde{T}_j}}(T_k,\eta_k) \right) \; |V_k| \\
\end{align*}
Applying the bounds for sums with random coefficients of regularly varying random variables of
\cite{hult:samorodnitsky:2008} (see also \cite[Chapter~4]{kulik:soulier:2020}) yields 
\begin{align*}
  \limsup   \pr   \left( \sup_{a\leq s \leq b} |X_s| > x \right)  \leq C \esp[\eta_0]\pr(|V_0|>x) \; .
\end{align*}
for some constant $C$. Returning to the truncated sum, we obtain 
\begin{align*}
  \lim_{m\to\infty}  \limsup_{T\to\infty} T \pr \left( \sup_{0\leq s \leq 1} |\tilde{X}_s^{(m)}| > a_Tx \right)
  \leq C \lim_{m\to\infty} \esp[(\eta_0-m)_+]=0\; .
\end{align*}
Thus we can apply the truncation argument and this proves that $\bsX$ is regularly varying in
$\spaceD$ and its tail measure is given by
\begin{align*}
  \tailmeasure(\map) & = \frac1{\esp[\eta_0]} \sum_{k\in\Zset}  \int_0^\infty \esp[\ind{T_k\leq0<T_k+\eta_0}
                    \map(u\epsilon_0\1{[T_k,T_k+\eta_0)})] \alpha u^{-\alpha-1} \rmd u \\
                  & = \frac1{\esp[\eta_0]} \int_{-\infty}^\infty   \int_0^\infty \esp[\ind{t\leq0<t+\eta_0}
                    \map(u\epsilon_0\1{[t,t+\eta_0)})] \alpha u^{-\alpha-1} \rmd u \\
                  & = \frac1{\esp[\eta_0]} \int_0^\infty \esp \left[  \int_0^{\eta_0}
                    \map(u\epsilon_0\1{[-t,\eta_0-t)}) \rmd t  \right] \alpha u^{-\alpha-1} \rmd u \\
                  & = \frac1{\esp[\eta_0]}  \int_0^\infty \esp \left[  
                    \map(u\epsilon_0\1{[-\zeta',\zeta)} \right] \alpha u^{-\alpha-1} \rmd u \; .
\end{align*}
with $\zeta,\zeta'$ such that
\begin{align*}
  \pr(\zeta'>s,\zeta>t) = \frac{ \esp[(\eta_0-s-t)_+] }{\esp[\eta_0]}\; .
\end{align*}
\begin{remark}
  \label{rem:palm}
  The law of $(-\zeta',\zeta)$ is the law of the points $(T_0,T_1)$ of a stationary renewal process
  with interarrival times distributed as $\eta_0$ under the Palm
  measure. Cf.~\cite[Section~1.4.1]{baccelli:bremaud:2003}.
\end{remark}
The tail process is given by
\begin{align*}
  \bsY = Y \epsilon_0 \1{[-\zeta',\zeta)}  \; .
\end{align*}
By \Cref{coro:extremalindex}, the extremal index of $\bsX$ is given by
\begin{align*}
  \candidate =  \esp \left[ \frac1{\int_{-\infty}^\infty \ind{Y_s>1} \rmd s } \right] = \esp[ (\zeta+\zeta')^{-1}] = \frac1{\esp[\eta_0]} \; .  
\end{align*}
For this simple process, we can also confirm the findings of \Cref{sec:anchoring-maps}. Since $\bsY$
reaches its maximum by an upward jump, condition \cref{eq:condition-continuite} holds.  Let $\mci_0$
and $\mci_1$ be the infargmax functional and the time of the first exceedance over 1 as in
\Cref{xmpl:infargmax-general,xmpl:first-exceedance}. Here, $\mci_0(\bsY)=\mci_1(\bsY)$ and in view
of \Cref{rem:palm}, the law of $\bsY$ given $\mci_0(\bsY)=\mci_1(\bsY)=0$ is
$Y\epsilon_0\1{[0,\eta)}$. Thus
\begin{align*}
  \esp [ \exc(\bsY) \mid \mci_1(\bsY)=0] = \esp[\eta_0] = \candidate^{-1} \; .
\end{align*}

\begin{appendices}
\appendix
\section{Vague convergence}
\label{app:vague}
Let $\Eset$ be a non-empty set. A boundedness on $\Eset$ is a subset $\mcb$ of $\mathcal{P}(E)$ with the
following properties:
\begin{itemize}
\item a finite union of elements of $\mcb$ is in $\mcb$;
\item a subset of an element of $\mcb$ is in $\mcb$.
\end{itemize}
The elements of $\mcb$ are called bounded sets. In a metric space, the class of metrically bounded
sets is a boundedness. Let now $\Eset$ be a Polish space, endowed with its Borel $\sigma$-field. We
will also need the following class of sets.
\begin{itemize}
\item A sequence $\{U_n,n\in\Nset\}$ of open sets if called a localizing sequence if for all
  $n\geq0$, $U_n\in \mcb$, $\overline{U}_n\subset U_{n+1}$, $\cup_{n\geq0} U_n=\Eset$ and every
  bounded set is included in one of the $U_n$.
\end{itemize}
Such a sequence $\{U_n\}$ is called a localizing sequence for $\Eset$.

A Borel measure $\mu$ is said to be $\mcb$-boundedly finite if $\mu(B)<\infty$ for all Borel sets
$B\in\mcb$. A sequence of $\mcb$-boundedly finite measures $\{\mu_n,n\in\Nset\}$ is said to converge
vaguely to a boundedly finite measure $\mu$, denoted $\mu_n\convvague\mu$, if
$\lim_{n\to\infty}\mu_n(A)=\mu(A)$ for all bound Borel sets $B$ such that $\mu(\partial B)=0$. A
version of the Portmanteau theorem is available, \cite[Theorem~2.7]{basrak:planinic:2019}. Let
$\mcm_\mcb$ be the set of boundedly finite Borel measures on $\Eset$. The topology of vague
convergence is the smallest topology on $\mcm_\mcb$ which makes the maps $\mu\mapsto\mu(f)$
continuous for all continuous functions $f$ with bounded support. Endowed with this topology, the
space $\mcm_\mcb$ is Polish, \cite[Theorem~3.1]{basrak:planinic:2019}.

If there exists a localizing sequence, then vague convergence can be related to weak
convergence. This is a consequence of \cite[Lemma~4.6]{kallenberg:2017}. 
\begin{proposition}
  \label{prop:vague-localized-is-weak}
  Let $\{\nu_n,n\in\Nset\}$ be a sequence of boundedly finite Borel measures on $\Eset$. Let
  $\{U_n,n\in\Nset\}$ be a localizing sequence and $\nu$ a boundedly finite measure such that $U_k$
  is a continuity set of $\nu$ for all $k\geq1$. Then $\nu_n\convvague\nu$ if and only if, for each
  $k\in\Nset$, the restrictions $\mu_n$ to $U_k$ converge weakly to the restriction of $\mu$ to
  $U_k$. 
\end{proposition}
Another characterization of vague convergence is by means of Lipschitz functions. Let $\Eset$ be a
Polish space endowed with a boundedness $\mcb$. We say that a metric $d$ on a Polish space is
compatible with $\mcb$ if $d$ induces the topology of $\Eset$ and for every $B\in\mcb$ there exists
$\epsilon>0$ such that the $\epsilon$-enlargement of $B$ \wrt\ $d$ is still bounded.  (The
$\epsilon$-enlargement \wrt\ $d$ of a subset $B$ is the set
$\{x\in \Eset:\exists y\ni B, d(x,y)\leq \epsilon\}$.) A real-valued function $f$ on $\Eset$ is said
to be $d$-Lipschitz if there exists a constant $K$ such that $|f(x)-f(y)|\leq K d(x,y)$ for all
$x,y\in\Eset$. The following result is essentially \cite[Lemma~4.1]{kallenberg:2017}. See also
\cite[Theorem~B.1.17]{kulik:soulier:2020}.
\begin{proposition}
  \label{prop:vague-lipschitz}
  Let $\Eset$ be a Polish space endowed with a boundedness $\mcb$ and $d$ be a compatible
  metric. Let $\{\nu,\nu_n,n\geq1\}$ be $\mcb$-boundedly finite Borel measures. Then
  $\nu_n\convvague\nu$ \ifft\ $\lim_{n\to\infty} \nu_n(f) = \nu(f)$ for all bounded $d$-Lipschitz
  functions $f$ with support in~$\mcb$.
\end{proposition}
As a consequence of metrizability and the characterization of vague convergence by Lipschitz
functions, we obtain the following triangular argument.
\begin{lemma}
  \label{lem:triangular-argument-vague}
  Let $\{\nu_n,\nu_{m,n},m\geq1,n\geq1\}$ be boundedly finite Borel measures on a Polish space
  $\Eset$ endowed with a boundedness $\mcb$.  Assume that for each $m\geq1$,
  $\nu_{n,m}\convvague\nu^{(m)}$ as $n\to\infty$, $\nu^{(m)}\convvague \nu$ as $m\to\infty$ and for
  every non-negative bounded measurable map $f$ with bounded support and Lipshitz \wrt\ an arbitrary
  compatible metric,
  \begin{align*}
    \lim_{m\to\infty} \limsup_{n\to\infty} |\nu_n(f)-\nu_{m,n}(f)| = 0 \; .
  \end{align*}
  Then $\nu_{n}\convvague\nu$ as $n\to\infty$.
\end{lemma}

A random measure is a random element of $\mcm_\mcb$ endowed with the topology of vague
convergence. A sequence of random measures $N_n,n\in\Nset$ is said to converge weakly to a random
measure $N$, denoted $N_n\convweak N$, if $N_n(f)\convdistr N(f)$ for all bounded continuous functions $f$
with bounded support, and $\convdistr$ denotes weak convergence of real valued random variables.

The following result, \cite[Proposition~4.6]{basrak:planinic:2019}, provides a useful
characterization of vague convergence of weak convergence of random measures. 
\begin{theorem}
  \label{theo:carac-vague}
  Let $\Eset$ be a Polish space endowed with a boundedness $\mcb$ and let $d$ be a compatible
  metric.  Let $\{N_n,n\in\Nset\}$ be random measures in $\mcm_\mcb$. Then the
  following statements are equivalent:
  \begin{enumerate}[(i)]
  \item $N_n\convweak N$ as $n\to\infty$;
  \item $N_n(f)\convdistr N(f)$ as $n\to\infty$ for all bounded, \nonnegative\ $d$-Lipschitz continuous functions with bounded support;
  \item $\lim_{n\to\infty} \esp[\rme^{-N_n(f)}] = \esp[\rme^{-N(f)}]$ for all bounded, \nonnegative\
    $d$-Lipschitz continuous functions with bounded support.
  \end{enumerate}
\end{theorem}
Similarly to \Cref{lem:triangular-argument-vague}, we obtain a triangular argument for weak convergence of
random measures.
\begin{proposition}
  \label{prop:triangular-argument-weak-vague}
  Let $\{N_n,N_{m,n},n\geq1,m\geq1\}$ be random measures in $\mcm_\mcb$.

  Assume that $N_{m,n}\convweak N^{(m)}$ as $n\to\infty$, $N^{(m)}\convweak N$ as $m\to\infty$ and
  for all $\eta>0$, 
  \begin{align*}
    \lim_{m\to\infty} \limsup_{n\to\infty} \pr(|N_{n,m}(f)-N_n(f)|>\eta) = 0 \; .
  \end{align*}
  Then $N_n\convweak N$.
\end{proposition}

The fundamental example which covers all the situations of this paper is investigated in
\cite{hult:lindskog:2006}.
\begin{example}
  \label{xmpl:M0-convergence}
  Let $\Eset$ be a Polish space and let $\bszero$ be an element of $ \Eset$. Let
  $\Eset_0=\Eset\setminus\{0\}$. The boundedness $\mcb_0$ on $\Eset_0$ is the class of sets
  separated from $\bszero$: $B\in\mcb_0$ if and only if there exists an open set $U$ of $\Eset$ such
  that $B\subset U$ (the complement of $U$). If $d$ is any metric which induces the topology of
  $\Eset$, then $B\in\mcb_0$ if and only if there exists $\epsilon>0$ such that $x\in B$ implies
  $d(x,\bszero)>\epsilon$. The sequence $U_n = \{x\in \Eset:d(x,\bszero)>n^{-1}\}$, $n\geq1$ is a
  localizing sequence. Also, every bounded set has an $\epsilon$-enlargement \wrt\ $d$ which is
  still bounded. Thus, vague convergence on $\Eset_0$ is characterized by \nonnegative\ bounded
  Lipschitz functions \wrt\ any metric which induces the topology of $\Eset$.
\end{example}

\section{The $J_1$ topology}
\label{sec:J1}
For an $I\subset\Rset$, we define the $J_1$ metric on the space $\spaceD(I)$ of \cadlag\ functions
of $I$, denoted $d_{I}$, as follows. Let $\mcb_I$ be the set of one-to-one strictly increasing
continuous maps on $I$.  Then, for $f,g\in\spaceD(I)$,
\begin{align*}
  d_{I}(f,g) = \inf_{u \in \mcb_I} \supnorm[I]{f\circ u -g} \vee \supnorm[I]{u-\id}\; .
\end{align*}
Oviously, $d_{I}(f,g) \leq \supnorm[I]{f-g}$ and $d_I(0,f)=\supnorm[I]{f}$.  For $I=\Rset$, we write
  $d_\infty$ and $\supnorm[\infty]{\cdot}$.

For fixed $f,g\in\spaceD$, the map $t\mapsto d_{[-t,t]}(f,g)$ is \cadlag\ and continuous at every
$t$ such that $t$ and $-t$ are  continuity points of both $f$ and $g$.

The $J_1$ topology on $\spaceD(\Rset)$ is the topology of $J_1$ convergence on compact subsets of
$\Rset$, induced  by the metric
\begin{align}
  \label{eq:def-local-J1}
  d_{J_1}(f,g) = \int_{0}^\infty \{d_{[-t,t]}(f,g)\wedge1\} \rme^{-t} \rmd t  \; .
\end{align}
The space $\spaceD$ endowed with the $J_1$ topology is Polish and the Borel $\sigma$-field
associated to the $J_1$ topology on $\spaceD$ is the product $\sigma$-field. See
\cite[Section~2]{whitt:1980}. Note that for all $f,g\in\spaceD$,
\begin{align*}
  d_{J_1}(f,g)  \leq \supnorm[\infty]{f-g} \; .
\end{align*}
For any $a>0$, 
\begin{align*}
  d_{J_1}(0,f)=  \int_0^\infty (\supnorm[{[-t,t]}]{f} \wedge1) \rme^{-t}\rmd t \leq \supnorm[{[-a,a]}]{f}+\rme^{-a}
\end{align*}
Thus if $\epsilon>0$, $a>-\log(\epsilon/2)$ and $d_{J_1}(0,f)>\epsilon$, then
$\supnorm[{[-a,a]}]{f}>\epsilon/2$. Since $a$ does not depend on $f$, this proves that a subset $A$
of $\spaceD$ is separated from the null map $\bszero$ if and only if there exists $\epsilon>0$ and
$a>0$ such that
\begin{align*}
  \inf_{f\in A} \supnorm[{[-a,a]}]{f} > \epsilon \; .  
\end{align*}
For $a<b$ and $\eta>0$, $\eta<b-a$, let $\mcp(a,b,\eta)$ be the set of finite increasing sequences
$(t_0,\dots,t_k)$ with $k\geq1$, $t_0=a$, $t_k=b$ and
$\inf_{1 \leq i \leq k} (t_i-t_{i-1}) \geq \eta$.  Define for a function $f\in\spaceD$,
\begin{align*}
  w'(f,a,b,\eta) & = \inf_{(t_0,\dots,t_k)\in \mcp(a,b,\eta)} \sup_{1 \leq i \leq k} \sup_{t_{i-1}\leq s,t < t_i} |f(s)-f(t)|  \; , \\
  w''(f,a,b,\delta) & = \sup_{a \leq s\leq t \leq u \leq b \atop |u-t|\leq\delta} |f(t)-f(s)|\wedge |f(u)-f(t)| \; .
\end{align*}
It generally holds that $w''(f,a,b,\delta)\leq w'(f,a,b,\delta)$
(\cite[Eq.~(12.28)]{billingsley:1999}) but both quantities can be used to characterize relative
compactness in $\spaceD$.
\begin{theorem}
  \label{theo:weak-convergence-in-DR}
  Let $\{\bsX,\bsX_n,n\in\Nset\}$ be a sequence of $\spaceD(R)$-valued stochastic processes. Then
  $\bsX_n\convweak \bsX$ in $\spaceD(\Rset)$ endowed with the $J_1$ topology \ifft\ for all $a<b$
  such that $\pr(\bsX \mbox{ discontinuous at } a)=\pr(\bsX \mbox{ discontinuous at } b)=0$,
  $\bsX_n\fidi\bsX$ (in a dense subset of $[a,b]$) and for all $\epsilon>0$, either (hence both) of
  the following conditions hold:
  \begin{align}
    \label{eq:conditionwprime}
    \lim_{\delta\to0} \limsup_{n\to\infty} \pr(w'(\bsX_n,a,b,\delta) > \epsilon) & = 0 \; , \\
    \label{eq:conditionwseconde}
    \lim_{\delta\to0} \limsup_{n\to\infty} \pr(w''(\bsX_n,a,b,\delta) > \epsilon) & = 0 \; .
  \end{align}
\end{theorem}
\begin{proof}
  For $a\leq b$ and $f\in\spaceD$, let $R_{a,b}f$ be the restriction of $f$ to $[a,b]$.  By
  \cite[Theorem~2.8]{whitt:1980}, a sequence of probability measures $P_n$ on $\spaceD(\Rset)$
  converges weakly to $P$ if and only if
  $P_n\circ R_{a_k,b_k}^{-1} \convweak P_n\circ R_{a_k,b_k}^{-1}$ for all $k\in\Nset$ and a sequence
  $\{(a_k,b_k),k\in\Nset\}$ such that $\cup_{k\geq0} [a_k,b_k]=\Rset$. The sequence $a_k$ can be
  chosen \nonincreasing, the sequence $b_k$ can be chosen \nondecreasing\ and the points $a_k,b_k$
  can be chosen as continuity points of $P$, \ie\ $\pr(\bsy \mbox{ not continuous at } t)=0$ for all
  $t\in\{a_k,b_k,k\in\Nset\}$.

  Thus it suffices to prove that $R_{a,b}\bsX_n\convweak R_{a,b} \bsX$ for all continuity points
  $a$, $b$ of $\bsX$. By \cite[Theorem~13.2 and
  Theorem~13.3]{billingsley:1999} 
  this follows from the stated finite dimensional weak convergence and~\Cref{eq:conditionwseconde}
  or \Cref{eq:conditionwprime}.
\end{proof}

\section{A lemma for stable processes}
\label{sec:maximal}
We summarize here and give a self-contained proof of certain arguments used in
\cite{samorodnitsky:2004:maxima} that are needed in \Cref{sec:sumstable}.

\begin{lemma}
  \label{lem:inegalite-maximale-02}
  Let $\alpha\in(0,2)$ and $\{P_i,i\geq1\}$ be the points of a Poisson point process on $(0,\infty)$
  with mean measure $\nualpha$. Let $\{\bsZ,\bsZ_j,j\geq1\}$ be \iid\ separable locally bounded
  stochastic processes, independent of $\{P_i,i\geq1\}$, such that
  $0<\esp[\sup_{a\leq s \leq b}|Z(t)|^\alpha]<\infty$ for all real numbers $a\leq b$.  If
  $\alpha\in[1,2)$, assume furthermore that the distribution of $\bsZ$ is symmetric. Then it is
  possible to define an $\alpha$-stable process $\bsX = \sum_{j=1}^\infty P_j \bsZ_j$ and assume
  that $\bsX$ is separable and locally bounded. Then,
  \begin{align}
    \label{eq:inegalite-maximale-01}
    \lim_{x\to\infty} x^\alpha \pr \left(\sup_{a\leq s \leq b} \left| X(s)\right| > x\right)
    = \esp \left[ \sup_{a\leq s \leq b} |Z(s)|^{\alpha}\right] \; .
  \end{align}
\end{lemma}

\begin{proof}
  The proof in the case $\alpha\in(0,1)$ is straightforward since the sum $\sum_{j=1}^\infty P_j$ is
  almost surely convergent. We only prove the case $1 \leq \alpha <2$.

  Recall that  $\bsZ_{a,b}^*= \sup_{a\leq s\leq b} |Z(s)|$ and write 
  $c_\alpha(a,b)=\esp[(\bsZ_{a,b}^*)^\alpha]$.  Let $\bsW$ be a stochastic process whose distribution is given
  by
  \begin{align*}
    \pr_{\bsW} = \frac{\esp[(\bsZ_{a,b}^*)^\alpha\delta_{(\bsZ_{a,b}^*)^{-1}\bsZ}]}{\esp[(\bsZ_{a,b}^*)^\alpha]} \; .
  \end{align*}
  Let $\bsW^{(i)}$, $i\geq1$ be \iid\ copies of $\bsW$. Then
  $\bsX\eqdistr c_\alpha^{1/\alpha} \sum_{i=1}^\infty P_i W_i$.  See \cite[Section
  3.10]{samorodnitsky:taqqu:1994}.  The interest of replacing the process $\bsZ$ by $\bsW$ is that
  the latter satisfies $\pr(\sup_{a \leq s \leq b} |W(s)|=1)=1$.

  Let the points $P_i,i\geq1$ be numbered in decreasing order. Then $P_i = \Gamma_i^{-1/\alpha}$
  with $\Gamma_i$ the points of a unit rate Poisson point process on $[0,\infty)$.  Since $\Gamma_j$
  has a $\Gamma(j,1)$ distribution, we have, for every $k\geq1$,
  \begin{align*}
    \pr \left(\sup_{a\leq s \leq b} \left| \sum_{j=k+1}^\infty P_j W_j(t)\right| >  x\right)
    & =  \frac1{(k-1)!} \int_0^\infty \pr(G(y)>x) y^{k-1}  \rme^{-y} \rmd y \; ,
  \end{align*}
  with
  $G(y) = \sup_{a\leq s \leq b} \left| \sum_{j=1}^\infty (y+\Gamma_j)^{-1/\alpha}W_j(s)\right| $.
  Since $\pr(\sup_{a\leq s \leq b} |W(s)|=1)=1$, as shown in the proof of \cite[Theorem~2.2, bottom
  of p.814]{samorodnitsky:2004:maxima} (which uses the symmetry assumption and takes its argument
  from the proof of \cite[Lemma~2.2]{rosinski:samorodnitsky:1993}), there exists $r>0$ such that for
  all $y>0$,
  \begin{align*}
    \esp \left[ \exp \left\{\frac{\log2}{r+2y^{-1/\alpha}} G(y)  \right\} \right] \leq 4 \; .
  \end{align*}
  Thus
  \begin{align*}
    \pr \left(\sup_{a\leq s \leq b} \left| \sum_{j=k+1}^\infty P_i W_j(t)\right| > x\right) 
    & = \frac1{(k-1)!}\int_0^\infty   \pr (G(y) > x) y^{k-1} \rme^{-y} \rmd y \\
    & \leq \frac4{(k-1)!} \int_0^\infty   \rme^{-\frac{x\log2}{r+2y^{-1/\alpha}}} y^{k-1}  \rme^{-y} \rmd y  \\
    & = \frac4{(k-1)!} x^{-k\alpha} \int_0^\infty \rme^{-\frac{x\log2}{r+2xy^{-1/\alpha}}} y^{k-1} \rme^{-yx^{-\alpha}} \rmd y  \; .
  \end{align*}
  Note that 
  \begin{align*}
    \rme^{-\frac{x\log2}{r+2xy^{-1/\alpha}}}  \rme^{-yx^{-\alpha}}
    \leq
    \begin{cases}
      \rme^{-\sqrt{y}} & \mbox{ if } y > x^{2\alpha} \; , \\
      \rme^{-\frac{y^{1/\alpha}\log2}{ry^{1/(2\alpha)}+2}} & \mbox{ if } y \leq x^{2\alpha} \; .
    \end{cases}
  \end{align*}
  Therefore, the integral above is uniformly bounded \wrt\ $x$ and there exists a constant (which
  depends  on $k$ and on $\bsZ$, $a$, $b$ through  $r$) such that
  \begin{align}
    \label{eq:bound-k}
    \pr \left(\sup_{a\leq s \leq b} \left| \sum_{j=k+1}^\infty P_j W_j(t)\right| > x\right) 
    & \leq \constant\  x^{-k\alpha} \; .
  \end{align}
  Note that
  \begin{align*}
    \pr \left(c_\alpha^{1/\alpha}\sup_{a\leq s \leq b} \left| \sum_{j=2}^\infty P_j W_j(t)\right| > \epsilon x\right) \\
    \leq \pr \left(c_\alpha^{1/\alpha}P_2  > \epsilon x/2\right)
    &  + \pr \left(c_\alpha^{1/\alpha}\sup_{a\leq s \leq b} \left| \sum_{j=3}^\infty P_j W_j(t)\right| > \epsilon x/2\right) \; .
  \end{align*}
  Since $P_2$ is regularly varying with tail index $2\alpha$, applying \Cref{eq:bound-k} with $k=2$
  yields
  \begin{align}
    \label{eq:borne-P2}
    \pr \left(\sup_{a\leq s \leq b} \left| \sum_{j=2}^\infty P_i W_j(t)\right| > x\right) 
    & \leq \constant\  x^{-2\alpha} \; .
  \end{align}
  To prove \Cref{eq:inegalite-maximale-01}, write
  \begin{align*}
    \pr \left(\sup_{a\leq s \leq b} \left| X(s)\right| > x\right)  
    \leq \pr \left(c_\alpha^{1/\alpha}P_1 > (1-\epsilon)x\right) + \pr
    \left(c_\alpha^{1/\alpha}\sup_{a\leq s \leq b} \left| \sum_{j=2}^\infty P_j W_j(t)\right| >
      \epsilon x\right) \; .
  \end{align*}
  Applying \Cref{eq:borne-P2} yields
  \begin{align*}
    \limsup_{x\to\infty}      x^\alpha \pr \left(\sup_{a\leq s \leq b} \left| X(s)\right| > x\right)
    \leq  (1-\epsilon)^{-\alpha}c_\alpha  \; .
  \end{align*}
  Since $\epsilon$ is arbitrary, the $\limsup$ is actually equal to 1. A lower bound for the
  $\liminf$ is obtained similarly. This proves \Cref{eq:inegalite-maximale-01}.  
\end{proof}

\end{appendices}

\subsubsection*{Acknowledgements} This work has been greatly influenced by many discussions with
Cl\'ement Dombry and Enkelejd Hashorva. I am grateful to Olivier Wintenberger for organizing a
seminar on regular variation for continuous time stochastic processes in 2018-2019. There is no
telling my gratitude for the anonymous referees who thoroughly read the first version of this paper,
made many deep remarks which greatly improved the paper.

\end{document}

Do not go gentle into that good night,
Old age should burn and rave at close of day;
Rage, rage against the dying of the light.

Though wise men at their end know dark is right,
Because their words had forked no lightning they
Do not go gentle into that good night.

Good men, the last wave by, crying how bright
Their frail deeds might have danced in a green bay,
Rage, rage against the dying of the light.

Wild men who caught and sang the sun in flight,
And learn, too late, they grieved it on its way,
Do not go gentle into that good night.

Grave men, near death, who see with blinding sight
Blind eyes could blaze like meteors and be gay,
Rage, rage against the dying of the light.

And you, my father, there on the sad height,
Curse, bless, me now with your fierce tears, I pray.
Do not go gentle into that good night.
Rage, rage against the dying of the light.

Dylan Thomas